\documentclass{amsart}

\usepackage{bbm}
\usepackage{geometry}
\usepackage{amssymb}
\usepackage{enumerate}
\usepackage{mathrsfs}
\usepackage{hyperref}
\usepackage[justification=centering, skip=20pt]{caption}
\usepackage[all,cmtip]{xy}
\usepackage{graphicx}
\usepackage{dynkin-diagrams}
\usepackage{setspace}

\hypersetup{
    colorlinks,
    citecolor=black,
    filecolor=black,
    linkcolor=black,
    urlcolor=black
    }

\newcommand{\R}[0]{\mathbb{R}}
\newcommand{\Z}[0]{\mathbb{Z}}

\newcommand{\T}[0]{\mathbb{T}}
\newcommand{\C}[0]{\mathbb{C}}
\renewcommand{\t}[1]{\textup{#1}}

\newtheorem{thm}{Theorem}[section]
\newtheorem{cor}[thm]{Corollary}
\newtheorem{prop}[thm]{Proposition}
\newtheorem{lem}[thm]{Lemma}
\newtheorem{conj}[thm]{Conjecture}

\newtheorem{defn}[thm]{Definition}

\newtheorem{rem}[thm]{Remark}

\let\emptyset\varnothing

\numberwithin{equation}{section}

\newcommand\numberthis{\addtocounter{equation}{1}\tag{\theequation}}

\usepackage{etoolbox}

\makeatletter
\patchcmd{\@settitle}{\uppercasenonmath\@title}{}{}{}
\patchcmd{\@setauthors}{\MakeUppercase}{}{}{}
\patchcmd{\section}{\scshape}{}{}{}
\makeatother

\title[On Fourier restriction type problems on compact Lie groups]{On Fourier restriction type problems on compact Lie groups}

\author[Y. Zhang]{Yunfeng Zhang}
\address{School of Mathematical Sciences, Peking University}
\email{yunfengzhang108@gmail.com}

\begin{document}

\onehalfspacing

\maketitle

\begin{abstract}
In this article, we obtain new results for Fourier restriction type problems on compact Lie groups. We first provide a sharp form of $L^p$ estimates of irreducible characters in terms of their Laplace-Beltrami eigenvalue and as a consequence provide some sharp $L^p$ estimates of joint eigenfunctions for the ring of conjugate-invariant differential operators. Then we improve upon the previous range of exponent for scale-invariant Strichartz estimates for the Schr\"odinger equation, and provide new $L^p$ bounds of Laplace-Beltrami eigenfunctions in terms of their eigenvalue similar to known bounds on tori. A key ingredient in our proof of these results is a barycentric-semiclassical subdivision of the Weyl alcove in a maximal torus. On each component of this subdivision we carry out the analysis of characters and exponential sums, and the circle method of Hardy--Littlewood and Kloosterman. 
\end{abstract}

\section{Introduction} 
\subsection{Three problems}
The goal of this article is to obtain new results for problems of Fourier restriction type on the setting of compact Lie groups. On  Euclidean spaces, Fourier restriction estimates were first explicitly posed and studied by Stein and Tomas \cite{Tom75} in the 1970s, as special types of oscillatory integrals. Let $d\mu$ be a measure on $\mathbb{R}^n$ for example the surface measure on a hypersurface such as a sphere or paraboloid. Fourier restriction estimates ask about decay properties of the (inverse) Fourier transform of $d\mu$, which may be quantified in terms of Lebesgue spaces via an inequality of the form 
\begin{align*}
\|\widehat{f\cdot d\mu}\|_{L^p(\mathbb{R}^n)}\leq C\|f\|_{L^q(d\mu)}.
\end{align*}
These estimates have broad applications in analysis and partial differential equations, and
are currently under intensive study with an abundance of hard open problems; we refer to \cite{Dem20} for a recent survey. 
In particular, when $d\mu$ is the surface measure on the standard sphere, the above inequality provides $L^p$-bounded eigenfunctions of the Laplacian; or if $d\mu$ is the surface measure on the paraboloid $\xi_n=\xi_1^2+\xi_2^2+\cdots+\xi_{n-1}^2$, the above inequality becomes the Strichartz estimate for the Schr\"odinger equation 
$$\|e^{it\Delta}f\|_{L^p(\mathbb{R}^{n})}\leq C\|f\|_{L^q{(\mathbb{R}^{n-1})}}.$$

On the other hand, Bourgain \cite{Bou93e,Bou93} in 1993 considered discrete analogues of the above problems, where $d\mu$ is a measure on $\mathbb{Z}^n$. The discrete Fourier restriction problem for the sphere or the paraboloid then refers to when $d\mu$ is the counting measure on the discrete sphere or the discrete paraboloid, which is equivalent to $L^p$ bound of Laplacian eigenfunctions on square tori or Strichartz estimate for the Schr\"odinger equation on square tori, respectively. On a general compact manifold, Sogge \cite{Sog88} in 1988 initiated the study of norms of spectral projectors on a band of eigenvalues of the Laplace-Beltrami operator. Zero band width corresponds to individual eigenvalues, and it amounts to $L^p$ bound of Laplace-Beltrami eigenfunctions. One may also pose Strichartz estimates for the Schr\"odinger equation 
$$\|e^{it\Delta}f\|_{L^p(I\times M)}\leq C\|f\|_{L^q(M)},$$
on a general compact manifold $M$, which was studied by Burq--G\'erard--Tzvetkov in 2004. 

Like square tori, many other manifolds are equipped with additional operators that commute with the Laplace-Beltrami operator. For symmetric spaces, there is a full commutative ring of differential operators invariant under the symmetry group, which include the Laplace-Beltrami operator as a special element; to go further, for arithmetic locally symmetric spaces, there are additionally Hecke operators that commute with the ring of invariant differential operators. 
Instead of restricting on individual eigenvalues of the Laplace-Beltrami operator as considered in Sogge's work, one may restrict on individual spectral parameters of the full ring of invariant operators, and this amounts to bounds of joint eigenfunctions of this ring. Compared with the previous two Fourier restriction problems, this one is more of purely representation theoretic flavor. 

We have now introduced three problems of Fourier restriction type on the setting of compact manifolds, and we summarize them as follows. 

\noindent{\bf Problem 1.} Joint eigenfunction bounds for a commutative ring of operators that commute with the Laplace-Beltrami operator. 

\noindent{\bf Problem 2.} Strichartz estimates for the Schr\"odinger equation. 

\noindent{\bf Problem 3.} Laplace-Beltrami eigenfunction bounds.

These problems are closely related to each other. Problem 1 has direct applications to Problem 2 and 3, provided there is a good knowledge of how to express the Laplace-Beltrami eigenvalue as a function of the spectral parameter for the operator ring, as is the case for compact Lie groups and globally symmetric spaces. Problem 3 is directly applicable to Problem 2, as long as there is a good knowledge of distribution of Laplace-Beltrami eigenvalues; in fact, this is how the optimal Strichartz estimate may be obtained from Laplace-Beltrami eigenfunction bounds on spheres as in \cite{Her13}. As integration of the Schr\"odinger kernel against characters in time gives the Laplace-Beltrami spectral projector kernels, a good understanding of Problem 2 would benefit Problem 3 also. Now we provide a more detailed review of known results for the above three problems on compact manifolds in the literature. 

\subsection{Literature review} 
Let $M$ be a compact manifold of dimension $d$ throughout this article.

\subsubsection{Problem 1.}
For general compact symmetric spaces of either noncompact or compact type, the seminal contribution by Sarnak \cite{Sar04} states that, for a joint eigenfunction $\psi$ of the full ring of invariant differential operators of Laplace-Beltrami eigenvalue of size $N^2\geq 1$, the sharp pointwise bound as follows holds 
\begin{align}\label{Sar04Sar04}
\|\psi\|_{L^\infty(M)}\leq C N^{\frac{d-r}{2}}\|\psi\|_{L^2(M)}.
\end{align} 
For $L^p$ bounds of joint eigenfunctions $\psi$, we have the conditional results of Marshall \cite{Mar16} which state that under the regularity assumption that the spectral parameter of $\psi$ stays within a fixed cone away from the walls of the Weyl chamber, then it holds true that 
\begin{align}\label{Marshallbound}
\|\psi\|_{L^p(M)}\leq C N^{\gamma(d,r,p)}\|\psi\|_{L^2(M)},
\end{align}
for the sharp exponent (on irreducible spaces $M$) as follows 
$$
\gamma(d,r,p)=\left\{\begin{array}{ll}
\frac{d-r}{2}-\frac{d}{p}, \ &\t{if }p>\frac{2(d+r)}{d-r},\\
\frac{d-r}{2}\left(\frac{1}{2}-\frac{1}{p}\right), \ &\t{if }2\leq p< \frac{2(d+r)}{d-r}.
\end{array}\right.
$$
In particular, if the rank of $M$ is one and thus the ring of invariant differential operators is solely generated by the Laplace-Beltrami operator, the above bound matches Sogge's Laplace-Beltrami eigenfunction bound \eqref{eigenboundreview} on a general compact manifold. However, it remains a challenge how to remove the above regularity assumption for higher-rank spaces.

For the special case when $\psi$ is a spherical function on compact symmetric spaces or in particular a character on groups, we may expect better estimates. As established in \cite{ST78} (see also \cite{Doo79}), any irreducible character $\chi$ on a compact simple Lie group of rank $r$ satisfies the bound
$$\|\chi\|_{L^p(M)}\leq C,\ \t{for }p<\frac{2d}{d-r}.$$
Let $d_\chi$ be the associated dimension of representation, then we have the bound 
$$|\chi|\leq|d_\chi|\leq C N^{\frac{d-r}{2}}$$
by an application of the Weyl dimension formula. By interpolation, we then have the bound
\begin{align}\label{Dooley}
\|\chi\|_{L^p(M)}\leq 
\left\{
\begin{array}{ll}
C_\varepsilon N^{\frac{d-r}{2}-\frac{d}{p}+\varepsilon}, 
& \t{ for }p\geq \frac{2d}{d-r},\\
C, & \t{ for } p<\frac{2d}{d-r}.
\end{array}
\right.
\end{align}
The above exponent of $N$ without the $\varepsilon$ can be checked to be sharp, by testing a character on a small neighborhood of the origin and choosing the spectral parameter regular enough. Similar results for spherical functions on arbitrary symmetric spaces of compact type are naturally conjectured to be true, but they seem still missing in the literature. 

Lastly, for arithmetic locally symmetric spaces, there is a richer theory; this is beyond our ability to make a thorough survey, and we only mention some results that look fundamental to our own eyes. A large part of literature has been focused on pointwise estimates of Hecke-Masse forms which are defined as joint eigenfunctions of the ring of invariant differential operators and the Hecke operators, with the purpose of improving the exponent in \eqref{Sar04Sar04}. We have the seminal contribution of Iwaniec and Sarnak \cite{IS95} for hyperbolic surfaces which are of rank one, and the work of Blomer and Pohl \cite{BP16} in rank two, and of Marshall \cite{Mar17} in arbitrary rank. 

\subsubsection{Problem 2.} There are two approaches to this problem, one via semiclassical analysis and one via exponential sums. The semiclassical references are mainly the work \cite{BGT04} of Burq, G\'erard, and Tzvetkov and \cite{ST02} of Staffilani and Tataru, where it was established for a finite interval $I$ 
\begin{align*}
\|e^{it\Delta}f\|_{L^p(I, L^q(M))}\leq C \|f\|_{H^{1/p}(M)}
\end{align*}
for all admissible pairs $(p,q)$ such that 
$$\frac{2}{p}+\frac{d}{q}=\frac{d}{2}, \ p,q\geq 2, \ (p,q,d)\neq(2,\infty,2).$$
Such estimates are non-scale-invariant: if $M$ is replaced by the Euclidean space $\mathbb{R}^d$, then the Sobolev space $H^{1/p}(M)$ in the above estimate may be replaced by $L^2(\mathbb{R}^d)$, as established by Ginibre and Velo \cite{GV95} and Keel and Tao \cite{KT98}, 
and the resulting estimate is then invariant under the scaling symmetry $u(t,x)\mapsto u(\lambda^2 t,\lambda x)$, which results in the above condition of admissibility for the pair $(p,q)$. These estimates have applications to local well-posedness theory for nonlinear Schr\"odinger equations with initial data of low yet scaling-subcritical regularity; see Proposition 3.1 in \cite{BGT04}. However, due to the non-scale-invariant nature of these estimates, they can never be used to solve the more interesting and difficult problem of local-wellposedness of scaling-critical regularity. There are scale-invariant Strichartz estimates on compact manifolds in the literature, but such results are rare. We first have the seminal contribution of Bourgain \cite{Bou93} on tori $M=\T^d$
\begin{align}\label{Bou93Strichartz}
\|e^{it\Delta}f\|_{L^p(I\times M)}\leq C \|f\|_{H^{d/2-(d+2)/p}(M)}
\end{align}
for a limited range of $p$ using the Hardy-Littlewood circle method, which is then enlarged to the optimal range $p>2+\frac{4}{d}$ by Bourgain and Demeter \cite{BD15A} by the new powerful method of decoupling theory. For spheres and Zoll manifolds of dimension at least three, the same estimate as \eqref{Bou93Strichartz} holds true for 
any $p>4$, as can be established by the argument in Herr \cite{Her13}; this range is also optimal as observed in \cite{BGT04}. In \cite{Zha18}, we established \eqref{Bou93Strichartz} on an arbitrary compact Lie group equipped with the canonical Killing metric, for any $p\geq 2+\frac{8}{r}$. The same range of estimate is subsequently generalized to all compact globally symmetric spaces in \cite{Zha20}.

\subsubsection{Problem 3.}
Let $f$ be an eigenfunction for the Laplace-Beltrami operator of eigenvalue $-N^2$. The fundamental result of Sogge \cite{Sog88} states 
\begin{align}\label{eigenboundreview}
\|f\|_{L^p(M)}\leq C N^{\gamma(d,p)}\|f\|_{L^2(M)}
\end{align}
for 
$$
\gamma(d,p)=\left\{\begin{array}{ll}
\frac{d-1}{2}-\frac{d}{p}, \ &\t{if }p\geq\frac{2(d+1)}{d-1},\\
\frac{d-1}{2}\left(\frac{1}{2}-\frac{1}{p}\right), \ &\t{if }2\leq p\leq \frac{2(d+1)}{d-1}.
\end{array}\right.
$$
These exponents were shown to be optimal by Sogge \cite{Sog88} on the standard spheres. The major open question is then to find refinement of the above exponents for various kinds of geometry. For example, with the presence of negative curvature, Hassell and Tacy \cite{HT15} established an $\varepsilon$-improvement (of $(\log N)^{-1/2}$ to be precise) for $p$ above the kink point $\frac{2(d+1)}{d-1}$, which may be seen as the effect of chaotic properties of the geodesic flow. At the other extreme of the fully integrable system the square tori $M=\T^d$, we first have the result of Zygmund \cite{Zyg74} where it was shown that \eqref{eigenboundreview} holds with $\gamma(2,4)=0$. Then Bourgain \cite{Bou93e} conjectured \eqref{eigenboundreview} should hold with $\gamma(2,p)=0$ for all $p<\infty$, and with 
\begin{align}\label{exponentfortori}
\gamma(d,p)=\frac{d-2}{2}-\frac{d}{p}
\end{align}
for $p>2d/(d-2)$ when $d\geq 3$, with an $N^\varepsilon$-loss for $d=3,4$. These conjectures for $p=\infty$ are indeed true, which are consequences of counting representations of integers as sums of squares, as observed in \cite{Bou93e}. 
Then in a series of papers, Bourgain \cite{Bou93e, Bou13}, Bourgain and Demeter \cite{BD13,BD15,BD15A} established the conjectured estimates on tori with an $\varepsilon$-loss for $p\geq 2(d-1)/(d-3)$ when $d\geq 4$. 

For arbitrary compact globally symmetric spaces, by counting representations of an integer by a positive definite integral quadratic form, we may use Sarnak's bound \eqref{Sar04Sar04} to establish the pointwise eigenfunction estimate 
$$\|f\|_{L^\infty (M)}\leq C N^{\frac{d-2}{2}}\|f\|_{L^2(M)},$$
provided the rank $r$ is at least 2, and with an $N^\varepsilon$-loss if $r=2,3,4$. 
It is worth noticing that by combining this counting argument with Sogge's bound \eqref{eigenboundreview}, some $L^p$ Laplace-Beltrami eigenfunction bounds with the same exponent as \eqref{exponentfortori} may be established for arbitrary products of rank-one spaces, though such an argument provides a rather poor range of exponent. Then in \cite{Zha20}, we provided $L^p$ bounds \eqref{eigenboundreview} with the exponent \eqref{exponentfortori} on compact globally symmetric spaces of rank $r\geq 5$ for all $p\geq 2+\frac{8}{r-4}$.

\subsection{Main results}
In this article, we prove new results for all the above three Fourier restriction type problems on the setting of compact Lie groups, via a rather uniform approach. 

\subsubsection{On Problem 1} 

We establish the sharp form of the character bound \eqref{Dooley} by removing the $\varepsilon$ factor, and provide as a corollary some sharp joint eigenfunction bounds. 

\begin{thm}\label{JointBound}
Let $M$ be a compact simple Lie group of dimension $d$ and rank $r$. \\
(i) Suppose $\chi_\mu$ is any irreducible character with Laplace-Beltrami eigenvalue $-N^2$. 
Then 
$$\|\chi_\mu\|_{L^p(M)}\leq 
\left\{
\begin{array}{ll}
C N^{\frac{d-r}{2}-\frac{d}{p}}, & \t{ for }p>\frac{2d}{d-r}, 
\\
C , & \t { for }p<\frac{2d}{d-r}. 
\end{array}
\right.
$$
\\
(ii) Suppose $\psi$ is any joint eigenfunction of the ring of conjugate-invariant differential operators with Laplace-Beltrami eigenvalue $-N^2$. Then 
$$\|\psi\|_{L^p(M)}\leq C N^{\frac{d-r}{2}-\frac{d}{p}} \|\psi\|_{L^2(M)},\t{ for }p>\frac{4d}{d-r}.$$ 
\end{thm}

Note that the exponent in the estimates in (ii) is sharp and matches that of \eqref{boundMarshall} albeit in a limited range of $p$. However,  they are established without any regularity assumptions on the spectral parameter and serve as the first such results on higher-rank spaces beyond products of rank-ones for $p<\infty$. 

We will establish the above results by a careful study of the behavior of characters over the Weyl alcove. 
We develop a so-called barycentric-semiclassical subdivision of the alcove, to take care of the behavior of characters both near the walls and near different vertices of the alcove. Then we obtain key sharp $L^p$ estimates on each component of this subdivision of some weight functions coming out of the Weyl denominator. 
Near-the-wall behavior of characters were also explored in \cite{Zha18}, while $L^p$ estimates of the Weyl denominator near the origin of the alcove were made in \cite{ST78}. 
In a sense, this work combines and sharpens the techniques of \cite{ST78} and \cite{Zha18}.

\subsubsection{On Problem 2} 
As mentioned above, we proved in \cite{Zha18} the following scale-invariant Strichartz estimate for the Schr\"odinger equation on any compact Lie group $M$ of rank $r$ equipped with the canonical Killing metric 
\begin{align}\label{StriSym}
\|e^{it\Delta}f\|_{L^p(I\times M)}\leq C \|f\|_{H^{d/2-(d+2)/p}(M)}, \t{ for }p\geq 2+8/r.
\end{align} 
The proof adapts the framework of Bourgain \cite{Bou93} on the setting of tori. We studied the spectrally localized Schr\"odinger kernel $\mathscr{K}_N$ defined as 
$$\phi\left(\frac{-\Delta}{N^2}\right)e^{it\Delta}f=f*\mathscr{K}_N(t,\cdot)$$
where $\phi\left(\frac{-\Delta}{N^2}\right)$ is a standard spectral localization operator on $M$. We realized this kernel as a Weyl type exponential sum and derived its pointwise bound as follows 
\begin{align}\label{Linfty}
\|\mathscr{K}_N(t,\cdot)\|_{L^\infty(M)}\leq C  \frac{N^{d}}{\left(\sqrt{q}\left(1+N\left\|\frac{t}{\mathcal{T}}-\frac{a}{q}\right\|^{\frac{1}{2}}\right)\right)^r}
\end{align}
on major arcs of the time variable $t$ 
$$\left\|\frac{t}{\mathcal{T}}-\frac{a}{q}\right\|\leq \frac{1}{qN}$$ 
centered at the fraction $a/q$ for $(a,q)=1$ and $q<N$. Here $\mathcal{T}$ is a period for the Sch\"odinger propagator $e^{it\Delta}$. The proof applies interpolation for the operator norm between $L^1\to L^\infty$ and $L^2\to L^2$ in order to exploit oscillation on both physical and frequency spaces, a classical harmonic analytic method as employed in the ancestor theorem of Stein and Tomas \cite{Tom75}. In this paper, we improve the range of $p$ in \eqref{StriSym} for compact semisimple Lie groups. A distinction between flat tori and compact semisimple Lie groups as well as the more general symmetric spaces of compact type is that joint eigenfunctions of invariant differential operators for the latter tend to be concentrated on conjugate points, as an example the zonal spherical harmonics on spheres blowing up at the north and south poles as the eigenvalue goes to infinity, while the characters on tori are uniform in size. This is behind the previously mentioned Marshall's conditional $L^p$-upgrades of Sarnak's pointwise bound \eqref{Sar04Sar04} into  
\begin{align}\label{boundMarshall}
\|\psi\|_{L^p(M)}\leq C N^{\frac{d-r}{2}-\frac{d}{p}}\|\psi\|_{L^2(M)}, \ \t{ for } p>\frac{2(d+r)}{d-r},
\end{align}
assuming that $M$ as a compact symmetric space is irreducible. Note in particular the extra term $N^{-\frac{d}{p}}$ in the above inequality compared with \eqref{Sar04Sar04} may be considered as the ``scale-invariant'' factor, as $\psi$ of Laplace-Beltrami eigenvalue of size $N^2$ is heuristically locally constant at scale $N^{-1}$ by the uncertainty principle. 
In comparison, the only such scale-invariant estimates valid on tori is when $p=\infty$. In a similar vein, the following scale-invariant $L^p$-upgrades of \eqref{Linfty} is expected 
\begin{align}\label{Lp}
\|\mathscr{K}_N(t,\cdot)\|_{L^p(M)}\leq C \frac{N^{d-\frac{d}{p}}}{\left(\sqrt{q}\left(1+N\left\|\frac{t}{\mathcal{T}}-\frac{a}{q}\right\|^{\frac{1}{2}}\right)\right)^r}
\end{align}
to hold true for a range of $p<\infty$. We confirm it on arbitrary compact semisimple Lie groups for a sharp range of $p$ as follows.

\begin{prop}\label{kernelmajorarc}
Suppose $M$ is a compact simply connected simple Lie group. Then for any $p>\frac{2d}{d-r}$, inequality \eqref{Lp} holds uniformly for $\left\|\frac{t}{\mathcal{T}}-\frac{a}{q}\right\|\leq \frac{1}{qN}$. More generally, let $M$ be a compact simply connected semisimple Lie group. 
Set
\begin{align}\label{thes}
s=\max\left\{ \frac{2d_i}{d_i-r_i} \right\},
\end{align} 
where the maximum is taken over all the simple factors $M_i$'s of $M$,  $d_i,r_i$ denoting respectively the dimension and rank of $M_i$.  Then \eqref{Lp} holds for any $p>s$. 
\end{prop}
This proposition will also be proved by an application of the above mentioned barycentric-semiclassical subdivision of the alcove and sharp $L^p$ estimates of some weight functions on the alcove. 
Then we can incorporate these $L^p$ estimates into Strichartz estimates. We replace the major-minor arc decomposition as in \cite{Bou93, Zha18} by the Farey dissection into major arcs only, observing that the contributions from the minor arcs would not fall in the right scale for $L^p$ estimates. 
We are then able to obtain the following improved scale-invariant Strichartz estimates on compact semisimple Lie groups, and they seem to saturate the method of \cite{Bou93} on this setting. 

\begin{thm}\label{Main}
Let $M$ be a compact semisimple Lie group of rank $r\geq 2$. Let $s$ be defined as in \eqref{thes}. Then 
\begin{align}\label{Strichartz} 
\|e^{it\Delta}f\|_{L^p(I\times M)}\leq C \|f\|_{H^{d/2-(d+2)/p}(M)}
\end{align}
holds for any 
$
p> 
2+\frac{8(s-1)}{sr}
$.
\end{thm}

What would be the optimal range of $p$ in the above estimate? If one only looks at class functions on compact Lie groups, then we will see that the estimates can be reduced to the following well conjectured Strichartz estimates for mixed Lebesgue norms on tori. Such estimates are indeed true on Euclidean spaces, as established in \cite{GV95, KT98}. 

\begin{conj} 
Let $(\cdot,\cdot)$ denote a positive-definite quadratic form of integral coefficients. Then we have 
\begin{align*}
\left\|\sum_{\xi\in\Z^r, \ |\xi|\leq N}a_\xi e^{it(\xi,\xi)+i(\xi, x)}\right\|_{L^p((I, dt),L^q(\mathbb{T}^r, dx))}\leq C N^{\frac{r}{2}-\frac{2}{p}-\frac{r}{q}}\|a_\xi\|_{l^2(\mathbb{Z}^r)}
\end{align*}
for all pairs $p,q\geq 2$ with 
$\frac{r}{2}-\frac{2}{p}-\frac{r}{q}>0$. 
\end{conj}

With the work \cite{BD15A} in mind, the resolution of this conjecture would require a decoupling theory for mixed-Lebesgue norms, which seems still missing in the literature. Using again the key subdivision of the alcove and sharp $L^p$ estimates of weight functions, we will show that the above conjecture implies the following conjecture.  

\begin{conj}\label{conjectureStrichartz}
Estimate \eqref{Strichartz} holds for class functions on any compact Lie group for any $p>2+\frac{4}{d}$.  
\end{conj}

What about functions that are not conjugate invariant? As mentioned earlier, the bound \eqref{Marshallbound} is sharp, in the sense that for each $2\leq p < \infty$ there are joint eigenfunctions that saturate the bound. In particular, the higher-rank Gaussian beam functions $\psi$, which concentrate around a maximal flat subspace and correspond to highest weight vectors in spherical representations, saturate the bound \eqref{Marshallbound} for all $p$ below the kink point, i.e., 
for all $2\leq p < \frac{2(d+r)}{d-r}$,  
$$c N^{\frac{d-r}{2}\left(\frac{1}{2}-\frac{1}{p}\right)}\|\psi\|_{L^2}\leq \|\psi\|_{L^p}\leq C N^{\frac{d-r}{2}\left(\frac{1}{2}-\frac{1}{p}\right)}\|\psi\|_{L^2},$$
for arbitrarily large $N$ (see \cite{Mar16}). Letting $f=\psi$ in \eqref{Strichartz}, a comparison of the exponents $\frac{d}{2}-\frac{d+2}{p}$ and $\frac{d-r}{2}\left(\frac{1}{2}-\frac{1}{p}\right)$ indicates that \eqref{Strichartz} cannot hold for all $p > 2+\frac{4}{d}$.\footnote{The author thanks the referee for pointing this out.} It remains an open question to come up with the optimal Strichartz bound.

\subsubsection{On Problem 3} 
We first present another application of Proposition \ref{kernelmajorarc} and the Hardy-Littlewood circle method via Farey dissection. We add the following $L^p$ Laplace-Beltrami eigenfunction bound on compact Lie groups to the existing literature, matching the exponent as in \eqref{exponentfortori} for tori.

\begin{thm}\label{eigengroup}
Let $M$ be a compact semisimple Lie group of rank $r\geq 5$.  Let $s$ be defined as in \eqref{thes}. Then we have the eigenfunction estimate 
\begin{align}\label{eigenbound}
\|f\|_{L^p(M)}\leq C N^{\frac{d-2}{2}-\frac{d}{p}}\|f\|_{L^2(M)}
\end{align}
for any $p>\frac{2sr}{sr-4s+4}$. 
\end{thm}

Enlarging the range of $p$, these results improve upon those in \cite{Zha20} and together they are the first such unconditional $L^p$-bounds for $p<\infty$ for genuine higher-rank spaces beyond the case of products of rank-ones such as tori, and altogether they form the only known examples that improve upon Sogge's bound \eqref{eigenboundreview} by a polynomial factor of $N$. These bounds may first look surprising, since they are established on a manifold of nonnegative curvature, and they are better than those $\varepsilon$-improvement results on manifolds of negative curvature as in the previously mentioned work of Hassell and Tacy \cite{HT15}. We observe that it is not because of the chaotic dynamical behavior of the geodesic flow but instead because of the high integrability of the Laplace-Beltrami operator as provided by a higher rank.

If an $\varepsilon$-loss is allowed in the above estimate, we are then able to prove it for a larger range of $p$ as follows, by fusing to full extent the tools developed in this paper and the argument as in \cite{Bou93e}. Meanwhile, in our proof we provide more details than those of \cite{Bou93e}, especially concerning implementation of Kloosterman's version of the circle method, so that readers having trouble understanding some rather sketchy parts of \cite{Bou93e} may benefit from reading our paper. 

\begin{thm}\label{eigengrouploss}
Let $M$ be a compact semisimple Lie group of rank $r\geq 4$.  Let $s$ be defined as in \eqref{thes}.  Then we have the eigenfunction estimate 
\begin{align}\label{eigenboundloss}
\|f\|_{L^p(M)}\leq C_{\varepsilon} N^{\frac{d-2}{2}-\frac{d}{p}+\varepsilon}\|f\|_{L^2(M)}
\end{align}
for any $p>\frac{2s(r+1)}{sr-3s+4}$. 
\end{thm}

Similar to the above discussion on Strichartz estimates, we will show that certain optimal eigenfunction bounds on tori as conjectured by Bourgain imply the following conjecture of eigenfunction bound for class functions on compact Lie groups. Also, the following conjectured optimal range $p>2+\frac{4}{d-2}$  for class functions cannot extend to general functions, which can be seen again using the higher-rank Gaussian beam functions as above. 
 
\begin{conj}\label{conjeigensym}
Let $M$ be a compact Lie group of rank $r\geq 2$. Then \eqref{eigenbound} holds for class functions
for any $p>2+\frac{4}{d-2}$, with an $\varepsilon$-loss if $2\leq r\leq 4$. 
\end{conj}

It seems reasonable to conjecture that all the above theorems obtained for compact Lie groups should extend to compact globally symmetric spaces. For the special case of products of odd-dimensional spheres, such results are indeed available \cite{Zha23}. A similar analysis as for the characters of the behavior of spherical functions near walls and near different vertices of the alcove would be needed for these extensions, but it is harder as spherical functions are in general less explicit.

\subsection{Overview of paper}
We provide an overview of the remainder of the paper as follows. 
In Section \ref{alcovereview}, we review the fundamental structures of compact Lie groups and in particular the affine Weyl groups and the geometry of the Weyl alcove. In Section \ref{sectionbarycentric}, we develop the key geometric tool of a so-called barycentric-semiclassical subdivision of the alcove, in order to distinguish points of different distance from the walls and near different vertices of the alcove, as eigenfunctions such as the characters behave differently on these different points. Associated to each component of this subdivision are some weight functions
coming out of the Weyl denominator, and we obtain some preliminary estimates on them. In Section \ref{sectiondecomposition}, \ref{projectionweightlattice}, and \ref{theSchrodingerkernel}, we refine some of the arguments in \cite{Zha18}, to decompose the characters according to this barycentric-semiclassical subdivision, to make orthogonal projections of the weight lattice with respect to parabolic root subsystems to further analyze the characters, and in turn to obtain formulas for the Schr\"odinger kernel on each component of the subdivision of alcove, in which the weight functions and Weyl type exponential sums appear. Section \ref{sectionLpnorm} forms the technical heart of this paper, where we obtain sharp $L^p$ estimates of the weight functions on each component of the subdivision, and they will be used throughout later sections for various $L^p$ estimates. In Section \ref{sectionjoint}, we prove Theorem \ref{JointBound}. 
In Section \ref{SectionLp}, we prove Proposition \ref{kernelmajorarc}. 
In Section \ref{fareydissection}, we review Farey dissection and in particular several important inputs from Kloosterman's version of the circle method. 
Then Section \ref{farey} provides proof of Theorem \ref{Main} and Theorem \ref{eigengroup}. 
In Section \ref{loss}, we prove Theorem \ref{eigengrouploss}.
Lastly, in Section \ref{evidence}, we show that for class functions on compact Lie groups, the best possible range for both Strichartz estimates and Laplacian eigenfunction bounds can be derived from the conjectured optimal Strichartz estimates (for the mixed Lebesgue norm) and Laplacian eigenfunction bounds on tori.

We list a few notations that will be used throughout this paper. We use $a\lesssim b$ to mean $a\leq Cb$ for some positive constant $C$, $a\lesssim_\varepsilon b$ to mean $a\leq C(\varepsilon) b$ for some function $C(\varepsilon)$ of $\varepsilon$, and $a\asymp b$ to mean $|a|\lesssim |b|\lesssim |a|$. And we use $e(\cdot)$ to mean $e^{2\pi i\cdot}$.

\subsection*{Acknowledgments}
The author is supported by National Key R{\&}D Program of China (No. 2022YFA1006700) and the Fundamental Research Funds for the Central Universities, Peking University. 
The author is deeply grateful to the referee for his/her extensive and insightful comments that greatly helped improve the manuscript. The author is also thankful to Professor Ciprian Demeter for sharing expertise on Kloosterman's circle method.

\section{Geometry of the Weyl alcove}\label{alcovereview}
We refer to \cite{Bou02, Var84, Shi87, Hum90} for information on analysis on compact Lie groups and in particular affine Weyl groups and Weyl alcoves that we review in this section without proof.
Let $U$ be a compact simply connected simple Lie group with Lie algebra $\mathfrak{u}$. Let $\mathfrak{t}$ be a Cartan subalgebra, i.e. a maximal abelian subalgebra of $\mathfrak{u}$ and let $T$ be the corresponding analytic subgroup which is a maximal torus of $U$. Let $\mathfrak{t}^*$ denote the real dual space of $\mathfrak{t}$ and let $i$ denote the imaginary unit so that $i\mathfrak{t}^*$ is the space of linear forms on $\mathfrak{t}$ that take imaginary values. Let 
$\Sigma\subset i\mathfrak{t}^*$ be the root system of $(\mathfrak{u},\mathfrak{t})$. Fix a simple system $\{\alpha_1,\ldots,\alpha_r\}$ of $\Sigma$. Let $-\alpha_0\in \Sigma$ be the corresponding highest root and we call $\alpha_0$ the lowest root. 
For $\alpha\in\Sigma$ and $n\in\Z$, define the root hyperplanes
$$\mathfrak{p}_{\alpha,n}:=\{H\in\mathfrak{t}:\ \alpha(H)/2\pi i+n=0\}.$$ 
These hyperplanes cut the ambient space $\mathfrak{t}$ into alcoves. Let 
$$A:=\{H\in\mathfrak{t}:\ \alpha_j(H)/2\pi i+\delta_{0j}>0\ \forall j=0,\ldots,r\}$$ 
be the open fundamental alcove and 
$$\bar{A}:=\{H\in\mathfrak{t}:\ \alpha_j(H)/2\pi i+\delta_{0j}\geq 0\ \forall j=0,\ldots,r\}$$ 
be the closed fundamental alcove.  
Here $\delta_{0j}$ equals 1 if $j=0$ and 0 otherwise. 
Let $W$ denote the finite Weyl group that acts on $\mathfrak{t}$ as well as $\mathfrak{t}^*$. The Weyl group translates $sA$ ($s\in W$) of $A$ are disjointly embedded in $T$ and form the so-called regular elements of $T$, such that $T\setminus \bigsqcup_{s\in W} sA$ contains the non-regular elements in $T$ and is a lower-dimensional subset of $T$. These non-regular elements in $T$ are also the conjugate points in $T$ of the origin on $U$ as a Riemannian manifold. We recall Weyl's integration formula, which is the basic tool to be used to evaluate the $L^p$ norm of class functions. 

\begin{lem}
For class functions $f$ on $U$, Weyl's integration formula can be written as 
\begin{align}\label{Weylint}
\int_U f(u)\ du=\int_{A}f(\exp H)|\delta(H)|^2\ dH
\end{align}
where $\delta(H)$ is the so-called Weyl denominator as follows
\begin{align}\label{alphadelta}
\delta(H):=\prod_{\alpha\in\Sigma^+}\left(e^{\frac{\alpha(H)}{2}}-e^{-\frac{\alpha(H)}{2}}\right), \ \t{for }H\in\mathfrak{t}.
\end{align}
Here $\Sigma^+$ is any positive system of $\Sigma$. 
\end{lem}
Let 
$$\rho:=\frac{1}{2}\sum_{\alpha\in\Sigma^+}\alpha$$
be the Weyl vector, then we may also express the above $\delta(H)$ as 
$$\delta(H)=\sum_{s\in W}\det s \ e^{(s\rho)(H)}, \ \t{for }H\in\mathfrak{t}.$$

The fundamental alcove (as well as any alcove) is a simplex whose geometry may be described using the extended Dynkin diagram for $\Sigma$. Each $\alpha_j$ ($j=0,\ldots,r$) corresponds to a node in the extended Dynkin diagram (Figure \ref{dynkin}), and for each proper subset $J$ of $\{0,\ldots,r\}$, $\{\alpha_j, \ j\in J\}$ is a simple system for a root subsystem $\Sigma_J$ whose Dynkin diagram can be obtained from the extended Dynkin diagram of $\Sigma$ by removing all the nodes not belonging to $J$. These $\Sigma_J$'s are usually called the parabolic subsystems of $\Sigma$. Associated to the simple system $\{\alpha_j, \ j\in J\}$ of $\Sigma_J$ is the positive system $\Sigma^+_J$ of $\Sigma_J$.
For $j=0,\ldots,r$, let $\tilde{s}_j:\mathfrak{t}\to\mathfrak{t}$ denote the reflection across the hyperplane 
$$\mathfrak{p}_j:=\mathfrak{p}_{\alpha_j,\delta_{0j}}=\{H\in\mathfrak{t}:\ \alpha_j(H)/2\pi i+\delta_{0j}=0\}.$$
These hyperplanes form the walls of the alcove $\bar{A}$ as its non-regular elements.  
For each $J\subset\{0,\ldots,r\}$, let $\tilde{W}_J$ be the group generated by the reflections $\{\tilde{s}_j,\ j\in J\}$. $\tilde{W}:=\tilde{W}_{\{0,\ldots,r\}}$ is called the affine Weyl group associated to $\Sigma$ and the $\tilde{W}_J$'s may be called the parabolic subgroups of $\tilde{W}$. The facets of $\bar{A}$ correspond to proper subsets of $\{0,\ldots,r\}$: for $J\subsetneqq\{0,\ldots,r\}$, 
\begin{align*}
A_J:&=\{H\in \bar{A}:\ \alpha_j(H)/2\pi i+\delta_{0j}=0\ \forall j\in J, \ \alpha_j(H)/2\pi i+\delta_{0j}>0\ \forall j\notin J\}
\end{align*}
is the corresponding $(r-|J|)$-dimensional facet. In particular, the $r+1$ vertices of $\bar{A}$ are of the form $A_I$ where $I$ ranges through cardinality-$r$ subsets of $\{0,\ldots,r\}$, and $A_\emptyset=A$. 
We have 
$$\bar{A}=\bigsqcup_{J\subsetneqq\{0,\ldots,r\}} A_J.$$ 
The stabilizer in $\tilde{W}$ of any point of $A_J$ coincides with $\tilde{W}_J$. For $J\subsetneqq\{0,\ldots,r\}$, let $W_J$ denote the Weyl group associated to the parabolic subsystem $\Sigma_J$. Then $\tilde{W}_J$ is isomorphic to $W_J$ under the translation map $\tilde{s}\mapsto\tilde{s}-\tilde{s}(0)$. By the definition of $A_J$, if a root hyperplane $\mathfrak{p}_{\alpha,n}$ contains $A_J$, then $\alpha\in\Sigma_J$. 



\begin{figure}
$\tilde{A}_1$:\,\dynkin[extended, labels={0,1}, edge length=1cm]A1\ 
$\tilde{A}_r$:\,\dynkin[extended, labels={0,1,2,r-1,r}, edge length=1cm]A{}\ 
$\tilde{B}_r$:\,\dynkin[extended, labels={0,1,2,3,r-2,r-1,r}, edge length=1cm]B{}\ 
$\tilde{C}_r$:\,\dynkin[extended, labels={0,1,2,r-2,r-1,r},  edge length=1cm]C{}\ 
$\tilde{D}_r$:\,\dynkin[extended, labels={0,1,2,3,r-3,,r-1,r}, labels*={,,,,,r-2,,}, edge length=1cm]D{}\ 

$\tilde{E}_6$:\,\dynkin[extended, labels={0,1,2,3,4,5,6},  edge length=1cm]E6\ 
$\tilde{E}_7$:\,\dynkin[extended, labels={0,1,2,3,4,5,6,7},  edge length=1cm]E7\ 
$\tilde{E}_8$:\,\dynkin[extended, labels={0,1,2,3,4,5,6,7,8},  edge length=1cm]E8\ 
$\tilde{F}_4$:\,\dynkin[extended , labels={0,1,2,3,4}, edge length=1cm]F4\ 
$\tilde{G}_2$:\,\dynkin[extended , labels={0,1,2}, edge length=1cm]G2

\caption{Extended Dynkin diagrams}\label{dynkin}
\end{figure}

\section{Barycentric-semiclassical subdivision} 
\label{sectionbarycentric}

From the semiclassical perspective, the characters of a compact Lie group should concentrate near conjugate points of the origin. In the alcove, these conjugate points are the walls of $\bar{A}$, thus in order to get $L^p$ estimates of these characters, their behavior near each facet of $\bar{A}$ needs to be clarified. We achieve this by making a so-called semiclassical subdivision of the alcove according to how close the points are from each facet. 
Let $N\gtrsim 1$ be a fixed large parameter. Let $J\subsetneqq\{0,\ldots,r\}$ and let $A_J$ be the corresponding facet. We define a subset $P_J$ of $A$ that consists of points close to $A_J$ but away from all the other facets. Let 
$$P_J:=\{H\in A:\ \alpha_j(H)/2\pi i+\delta_{0j}\leq N^{-1}\ \forall j\in J, \ \alpha_j(H)/2\pi i+\delta_{0j}> N^{-1}\ \forall j\notin J\}.$$ 
In other words, $P_J$ consists of points in the alcove that are $\leq N^{-1}$ close to the walls $\mathfrak{p}_j$ for $j\in J$ and are $>N^{-1}$ far from the other walls $\mathfrak{p}_j$ for $j\notin J$. We record the following self-evident fact as a lemma. 

\begin{lem}[Semiclassical subdivision] We have 
$$A=\bigsqcup_{J\subsetneqq\{0,\ldots,r\} } P_J.$$
\end{lem}

We need yet another subdivision of the alcove in order to evaluate $L^p$ norms of eigenfunctions. This one is technical in nature and concerns the classification of root systems. Its necessity will be transparent in its applications in later sections; the idea is that eigenfunctions such as the characters also behave differently near different vertices of the alcove, which motivates the following version of barycentric subdivision of the alcove. For each vertex $A_I$ ($|I|=r$) of $\bar{A}$, consider the convex hull $C_I$ of the barycenters of the facets $A_J$ such that $J\subset I$, in other words, facets that contain $A_I$ in their boundary. 

\begin{lem}[Barycentric subdivision]
We have 
$$A=\bigsqcup_{|I|=r} C_I.$$
\end{lem}

The above disjoint union is understood modulo a lower-dimensional subset. For $J\subsetneqq\{0,\ldots,r\}$ and $I\subset\{0,\ldots,r\}$ such that $|I|=r$, set 
\begin{align} 
P_{I,J}:&=P_J\cap C_I \label{PIJdefinition} \\
&=\{H\in C_I:\ \alpha_j(H)/2\pi i+\delta_{0j}\leq N^{-1}\ \forall j\in J, \ \alpha_j(H)/2\pi i+\delta_{0j}> N^{-1}\ \forall j\in I\setminus J\}.\nonumber
\end{align}
In other words, $P_{I,J}$ consists of points in $C_I$ that are $\leq N^{-1}$ close to the hyperplanes $\mathfrak{p}_j$ for $j\in J$ and $>N^{-1}$ far from the hyperplanes $\mathfrak{p}_j$ for $j\in I\setminus J$. 
Note that $P_{I,J}=\emptyset$ if $J$ is not a subset of $I$, if we  pick $N^{-1}$ small enough. For $J\subset I$, we now have 
$$C_I=\bigsqcup_{J\subset I} P_{I,J}, \ P_J=\bigsqcup_{I\supset J}P_{I,J},$$  
and:
\begin{lem}[Barycentric-semiclassical subdivision] 
\label{bs}
We have 
$$A=\bigsqcup_{J,I\subset\{0,\ldots,r\}, \ |I|=r, \ J\subset I} P_{I,J}.$$
\end{lem}

\begin{figure}
\scalebox{0.5}[0.5]
{
\begin{tikzpicture}

\draw (0, 0) -- (4, 6.928);
\draw (0,0) -- (8, 0);
\draw  (4, 6.928) -- (8, 0);

\draw (4, 2.30933) -- (2, 3.464); 
\draw  (4, 2.30933) -- (4,0); 
\draw  (4, 2.30933) -- (6, 3.464);

\node [below] at (4, -0.5) {\scalebox{2}{(a)}};

\end{tikzpicture}
\hspace{1cm}
\begin{tikzpicture}

\draw (0, 0) -- (4, 6.928);
\draw (0,0) -- (8, 0);
\draw  (4, 6.928) -- (8, 0);

\draw (0.2887,0.5) -- (7.7113,0.5);
\draw (7.423,0) -- (3.7113,6.428);
\draw (4.2887, 6.428) -- (0.577,0);

\draw[<->] (4.2, 0) -- (4.2, 0.5); 
\node [right] at (4.2, 0.25) {\scalebox{1.33}{$\ \asymp N^{-1}$}};

\node [below] at (4, -0.5) {\scalebox{2}{(b)}};

\end{tikzpicture}
\hspace{1cm}
\begin{tikzpicture}

\draw (0, 0) -- (4, 6.928);
\draw (0,0) -- (8, 0);
\draw  (4, 6.928) -- (8, 0);

\draw (4, 2.30933) -- (2, 3.464); 
\draw  (4, 2.30933) -- (4,0); 
\draw  (4, 2.30933) -- (6, 3.464);

\draw (0.2887,0.5) -- (7.7113,0.5);
\draw (7.423,0) -- (3.7113,6.428);
\draw (4.2887, 6.428) -- (0.577,0);

\draw [<->] (4.2, 0) -- (4.2, 0.5); 
\node [right] at (4.2, 0.25) {\scalebox{1.33}{$\ \asymp N^{-1}$}};

\node [below] at (4, -0.5) {\scalebox{2}{(c)}};

\end{tikzpicture}
}
\caption{(a) Barycentric subdivision \ \  (b) Semiclassical subdivision \\ (c) Barycentric-semiclassical subdivision}
\label{figurebarycentric}
\end{figure}

For $j\in\{0,\ldots,r\}$, let  
$$t_j(H):=\alpha_j(H)/2\pi i+\delta_{0j}.$$
Then for each $I\subset\{0,\ldots,r\}$ such that $|I|=r$, 
$\{t_j,\ j\in I\}$ provide a natural coordinate system for $C_I$, and there exists uniform positive constants 
\begin{align*}
c_1,c_2<1,
\end{align*}
such that 
\begin{align*}
\{H\in\mathfrak{t}:\ 0\leq t_j(H)\leq c_1\ \forall j\in I\}\subset C_I\subset \{H\in\mathfrak{t}:\ 0\leq t_j(H)\leq c_2\ \forall j\in I\}.
\end{align*} 
Also for each $J\subset I$, we have 
\begin{align*}
&\{H\in\mathfrak{t}:\ 0\leq t_j(H)\leq N^{-1}\ \forall j\in J, \ N^{-1}< t_j(H)\leq c_1\ \forall j\in I\setminus J\}\\
\subset P_{I,J}\subset &\{H\in\mathfrak{t}:\ 0\leq t_j(H)\leq N^{-1}\ \forall j\in J, \ N^{-1}< t_j(H)\leq c_2\ \forall j\in I\setminus J\}. \numberthis \label{PIJcoordinates}
\end{align*} 

Associated to the above barycentric-semiclassical subdivision are some naturally defined weight functions which appear as factors of the function $\delta(H)$ as in \eqref{alphadelta}. For $I\subset\{0,\ldots,r\}$ with $|I|=r$, let $\Sigma^{+}$ be a positive system of $\Sigma$ that contains $\Sigma^+_I$. For example, one can choose $\Sigma^{+}$ to be the set of all roots that are positive in the lexicographic ordering induced by the basis $\{\alpha_j,\ j\in I\}$. 

\begin{defn}[Weight functions]\label{definitiondeltas}
Set 
$$\delta(H)=\delta_I(H)\cdot\delta_{I,J}(H)\cdot\delta^J(H),$$
where
$$\delta_I(H):=\prod_{\alpha\in\Sigma^{+}\setminus\Sigma^+_I}\left(e^{\frac{\alpha(H)}{2}}-e^{-\frac{\alpha(H)}{2}}\right),$$
$$\delta^J(H):=\prod_{\alpha\in\Sigma_J^+}\left(e^{\frac{\alpha(H)}{2}}-e^{-\frac{\alpha(H)}{2}}\right),$$
$$\delta_{I,J}(H):=\prod_{\alpha\in \Sigma_I^+\setminus \Sigma^+_J}\left(e^{\frac{\alpha(H)}{2}}-e^{-\frac{\alpha(H)}{2}}\right).$$
\end{defn}

\begin{rem}
We emphasize here that the choice of $\Sigma^+$ varies and is determined by $I$ ($I\subset\{0,\ldots,r\}$, $|I|=r$).  
In the proof of Theorem \ref{JointBound} and Proposition \ref{kernelmajorarc}, we will first make a use of Weyl's integration formula, and apply the barycentric-semiclassical subdivision to reduce $L^p$ norms on the compact Lie group $U$ to those on each $P_{I,J}$. Then for each $P_{I,J}$, we fix the positive system $\Sigma^+$ as the one that contains $\Sigma^+_I$, and thus for $J\subset I$, $\Sigma^+$ also contains $\Sigma^+_J$.  
\end{rem}

We derive some preliminary estimates for these weight functions. 

\begin{lem}\label{deltaIdeltaJ}
We have 
\begin{align}\label{deltaI}
|\delta_I(H)|\gtrsim 1, \ \t{for }H\in C_I, 
\end{align}
\begin{align}\label{deltaJ}
|\delta^J(H)|\lesssim N^{-|\Sigma_J^+|}, \ \t{for }H\in P_J, 
\end{align}
\begin{align}\label{deltaIJ}
|\delta_{I,J}(H)|\gtrsim N^{|\Sigma_J^+|-|\Sigma^+|}, \ \t{for }H\in P_{I,J}. 
\end{align}
\end{lem} 
\begin{proof}
First, for $\alpha\in\Sigma$, 
$\left|e^{\frac{\alpha(H)}{2}}-e^{-\frac{\alpha(H)}{2}}\right|\asymp\|\alpha(H)/2\pi i\|$ is comparable to the shortest distance from the root hyperplanes $\mathfrak{p}_{\alpha,n}$ among $n\in\Z$. Here $\|\cdot\|$ denotes the distance from the nearest integer. 
By definition, $C_I=\bigcup_{J\subset I}P_{I,J}\subset \bigcup_{J\subset I}A_J$. As reviewed in Section \ref{alcovereview}, the only hyperplanes that contain the facets $\bigcup_{J\subset I}A_J$ are of the form $\mathfrak{p}_{\alpha,n}$ for $\alpha\in\Sigma_I$. Thus $C_I$ as a compact set stays away from all root hyperplanes $\mathfrak{p}_{\alpha,n}$ for $\alpha\in\Sigma\setminus\Sigma_I$ by a fixed distance, and as $\Sigma^+\setminus\Sigma^+_I\subset \Sigma\setminus\Sigma_I$,  this yields the first estimate. The second estimate follows since each $\alpha\in\Sigma_J^+$ is a linear combination of simple roots $\alpha_j$, $j\in J$, and then the definition of $P_J$ assures that $\|\alpha(H)/2\pi i\|\lesssim N^{-1}$ for $H\in P_J$. For the last estimate, since $|\Sigma_I^+|\leq|\Sigma^+|$, it suffices to show 
\begin{align*}
|\delta_{I,J}(H)|\gtrsim N^{|\Sigma_J^+|-|\Sigma^+_I|}, \ \t{for }H\in P_{I,J}. 
\end{align*} 
This follows if $P_{I,J}$ stays away from the root hyperplanes $\mathfrak{p}_{\alpha,n}$ for $\alpha\in\Sigma_I^+\setminus\Sigma_J^+$ by a distance of at least $\asymp N^{-1}$. 
We construct open subsets $\mathcal{N}_{I'}$ of $\mathfrak{t}$  for $I'\subset I$ inductively as follows. 
We remark that the definition of $\mathcal{N}_{I'}$ will depend on both $I'$ and $I$, and the $I$ will be clear from context whenever any $\mathcal{N}_{I'}$ appears in the following. 
It is instructive to picture the following construction using the root system $A_3$. $I'$ can be $\emptyset$ in the following construction. 
First let $\mathcal{N}_{I}$ be a small open neighborhood of the vertex $A_{I}$. For $J'\subset I$, suppose $\mathcal{N}_{K'}$ is already defined for all $I\supset K'\supsetneq J'$. Then we let $\mathcal{N}_{J'}$ be a small open neighborhood of 
$$\overline{A_{J'}\cap C_I}\setminus\left(\bigcup_{I\supset K'\supsetneq J'}\mathcal{N}_{K'}\right).$$
Moreover, 
since the only root hyperplanes that contain the facet $A_{J'}$ are of the form $\mathfrak{p}_{\alpha,n}$ with $\alpha\in \Sigma_{J'}$, we can make the above $\mathcal{N}_{J'}$ to stay a fixed distance away from all the other root hyperplanes $\mathfrak{p}_{\alpha,n}$ with $\alpha\notin \Sigma_{J'}$. 
It is clear that 
$$C_I\subset \bigcup_{I'\subset I}\mathcal{N}_{I'},$$
so that 
$$P_{I,J}= \bigcup_{I'\subset I}P_{I,J}\cap \mathcal{N}_{I'}.$$
Note that by construction $P_{I,J}\cap \mathcal{N}_{I'}=\emptyset$ unless $J\subset I'$. 
Now suppose $H\in P_{I,J}\cap \mathcal{N}_{I'}$ for some $I'\subset I$. Since $\mathcal{N}_{I'}$ stays a fixed distance away from $\mathfrak{p}_{\alpha,n}$ with $\alpha\notin\Sigma_{I'}$, to finish the proof it suffices to show $\|\alpha(H)/2\pi i\|\gtrsim N^{-1}$ for any $\alpha\in \Sigma_{I'}^+\setminus \Sigma_J^+$.
For any root $\alpha\in \Sigma_{I'}^+\setminus\Sigma_J^+$, write $\alpha=\sum_{j\in I'} n_j\alpha_j$, in which there exists some $j_0\in I'\setminus J$ such that $n_{j_0}$ is positive. Since by construction $\mathcal{N}_{I'}$ stays close to the facet $A_{I'}$, 
$\|\alpha_j(H)/2\pi i\|$ is a small enough number for $j\in I'$ and $H\in \mathcal{N}_{I'}$. This implies that 
$$\|\alpha(H)/2\pi i\|=\sum_{j\in I'} n_j\|\alpha_j(H)/2\pi i\|.$$
As $H\in P_{I,J}$, $\|\alpha_{j_0}(H)/2\pi i\| > N^{-1}$. Since $n_{j_0}>0$, we conclude that $\|\alpha(H)/2\pi i\|\gtrsim N^{-1}$. 

\end{proof}

We have the following immediate corollary. 

\begin{lem}
We have 
\begin{align}\label{deltapreliminary}
|\delta(H)|\gtrsim N^{-|\Sigma^+|},\t{ for }H\in P_\emptyset. 
\end{align}
\end{lem}
\begin{proof}
Since $P_\emptyset=\bigcup_{I\subset\{0,\ldots,r\},\ |I|=r}P_{I,\emptyset}$, it suffices to prove the estimate for each $P_{I,\emptyset}$. Write $\delta=\delta_I\cdot \delta_{I,\emptyset}$, then the result follows by \eqref{deltaI} and \eqref{deltaIJ}. 
\end{proof}

\section{Characters} \label{sectiondecomposition}
In this section, adapted to the above barycentric-semiclassical subdivision, we give a formula of the character that illuminate its behavior on each component of this subdivision.  
Let $(\cdot,\cdot)$ denote the Killing form on $\mathfrak{t}$ as well as on $\mathfrak{t}^*$ (and also on $i\mathfrak{t}^*$ by linear extension) and $|\cdot|$ be the corresponding norm, for which the Weyl group $W$ acts on $\mathfrak{t}$ as well as on $\mathfrak{t}^*$ by isometry. 
The weight lattice reads 
$$\Lambda:=\left\{\mu\in i\mathfrak{t}^*:\ \frac{2(\mu,\alpha)}{(\alpha,\alpha)}\in\Z \ \forall\alpha\in\Sigma\right\},$$
and let 
$$\Lambda^+:=\left\{\mu\in i\mathfrak{t}^*:\ \frac{2(\mu,\alpha)}{(\alpha,\alpha)}\in\mathbb{Z}_{\geq 1} \ \forall\alpha\in\Sigma^+\right\}$$ 
be the subset of strictly dominant weights. We have chosen here the strictly dominant weights instead of the more standard larger set of dominant
weights to slightly improve simplicity of the following presentation.
Concerning the relation between $\Lambda$ and $\Lambda^+$, we have the following standard lemma of root system theory. We say $\mu\in i\mathfrak{t}^*$ is regular provided $(\mu,\alpha)\neq 0$ for all $\alpha\in\Sigma$, and non-regular otherwise. 

\begin{lem}\label{LambdaLambda+}
The regular elements of $\Lambda$ form exactly the subset $\bigsqcup_{s\in W}s\Lambda^+$ so that we have 
$$\Lambda=\left(\bigsqcup_{s\in W}s\Lambda^+\right)\bigsqcup \left\{\mu\in \Lambda:\ (\mu,\alpha)=0\t{ for some }\alpha\in\Sigma\right\}.$$
\end{lem}

Each $\mu\in\Lambda^+$ is associated with an irreducible representation of $U$ of highest weight $\mu-\rho$, and the associated character $\chi_\mu$ can be expressed by the following Weyl's formula 
$$\chi_\mu(H):=\frac{\sum_{s\in W}\det s \ e^{(s\mu)(H)}}{\sum_{s\in W}\det s \ e^{(s\rho)(H)}}=\frac{\sum_{s\in W}\det s \ e^{(s\mu)(H)}}{\delta(H)}, \ \t{for }H\in\mathfrak{t}.$$ 
Note that this formula make sense for any $\mu\in i\mathfrak{t}^*$ and in particular for any $\mu\in\Lambda$, though the characters are initially defined only for $\mu\in\Lambda^+$. 
Observe that $\chi_{s\mu}=\det s \chi_\mu$ for $s\in W$. 

We now study the behavior of $\chi_\mu$ near each facet of $\bar{A}$. 
For $J\subsetneqq\{0,\ldots,r\}$, recall that $A_J$ denotes the corresponding $(r-|J|)$-dimensional facet of the fundamental alcove $A$. Consider the subspace 
\begin{align}\label{tJ}
\mathfrak{t}_J:=\bigoplus_{j\in J}\R H_{\alpha_j}
\end{align} 
of $\mathfrak{t}$, where $H_{\alpha_j}\in\mathfrak{t}$ is defined such that $(H_{\alpha_j},H):=\alpha_j(H)/2\pi i$ for all $H\in\mathfrak{t}$. Let $H_J:=\t{Proj}_{\mathfrak{t}_J}(H)$ denote the orthogonal projection of $H\in\mathfrak{t}$ on $\mathfrak{t}_J$ with respect to the Killing form. Let 
\begin{align}\label{HJperp}
H_J^\perp:=H-H_J,
\end{align}
which lies in the orthogonal complement 
\begin{align}\label{tJperp}
\mathfrak{t}_J^\perp:=\mathfrak{t}\ominus\mathfrak{t}_J
\end{align}
of $\mathfrak{t}_J$ in $\mathfrak{t}$. 
Dual to $\mathfrak{t}_J$, we also consider the root subspace 
$$V_J:=\t{span}_\R \Sigma_J=\bigoplus_{j\in J}\R\alpha_j$$
of $i\mathfrak{t}^*$ spanned by the parabolic subsystem $\Sigma_J$. 
Let $\mu_J:=\t{Proj}_{V_J}(\mu)$ denote the orthogonal projection of $\mu\in\Lambda$ on $V_J$. 
Let $\Sigma_J^+:=\Sigma^+\cap \Sigma_J$ be the positive system for $\Sigma_J$ and let $\Lambda_J$ be the weight lattice for $\Sigma_J$. For $\gamma\in \Lambda_J$, let 
$$\chi^J_\gamma:=\frac{\sum_{s_J\in W_J}\det s_J\ e^{s_J\gamma}}{\sum_{s_J\in W_J}\det s_J\ e^{s_J\rho_J}}
=\frac{\sum_{s_J\in W_J}\det s_J\ e^{s_J\gamma}}{\delta^J(H)}$$
be the associated character where $\rho_J:=\frac{1}{2}\sum_{\alpha\in\Sigma_J^+}\alpha$ is the Weyl vector associated to $\Sigma_J$. Note that the above expression makes sense for any $\gamma\in V_J$.  For each $j=0,\ldots,r$, let  
$$\mathfrak{t}_j:=\R H_{\alpha_j}, \ \mathfrak{t}_j^\perp:=\mathfrak{t}\ominus \mathfrak{t}_j=\mathfrak{p}_{\alpha_j,0}.$$
Then the Weyl group $W_J$ is generated by reflections across the hyperplanes $\mathfrak{t}_j^\perp$, $j\in J$. In particular, as $\mathfrak{t}^\perp_J=\bigcap_{j\in J} \mathfrak{t}_j^\perp$, 
any $s_J\in W_J$ fixes every point on $\mathfrak{t}^\perp_J$. We now derive the following key formula of characters. 

\begin{lem}\label{lemdecomposition}
For any $H\in\mathfrak{t}$ and $\mu\in i\mathfrak{t}^*$, we have 
\begin{align}\label{char}
\chi_\mu(H)=\frac{1}{|W_J| \delta_I(H)\delta_{I,J}(H)}\sum_{s\in W}\det s\ e^{(s\mu)(H_J^\perp)}\chi^J_{(s\mu)_J}(H_J).
\end{align}
\end{lem}

\begin{proof}
We have for $H\in\mathfrak{t}$
\begin{align*}
\chi_\mu(H)&=\frac{\sum_{s\in W}\det s \ e^{(s\mu)(H)}}{\delta(H)}\\
&=\frac{\sum_{s\in W}\sum_{s_J\in W_J}\det (s_Js) \ e^{(s_J s\mu)(H)}}{|W_J|\prod_{\alpha\in \Sigma^+\setminus \Sigma^+_J}\left(e^{\frac{\alpha(H)}{2}}-e^{-\frac{\alpha(H)}{2}}\right)\prod_{\alpha\in\Sigma^+_J}\left(e^{\frac{\alpha(H)}{2}}-e^{-\frac{\alpha(H)}{2}}\right)}.
\end{align*}
Now write $H=H_J+H_J^\perp$, we have for $s\in W$ and $s_J\in W_J$ that 
$$(s_J s\mu)(H)
=(s_J s\mu)(H_J)+(s_J s\mu)(H_J^\perp).$$
But 
$$(s_J s\mu)(H_J)=(s_J(s\mu)_J)(H_J)$$ 
since 
$(s\mu-(s\mu)_J)(H_J)=0$ by definition, while  
$$(s_J s\mu)(H_J^\perp)=(s\mu)(s_J^{-1} H_J^\perp)=(s\mu)(H_J^\perp)$$
since $s_J^{-1}$ as an element of $W_J$ fixes any point on $\mathfrak{t}_J^\perp$. Also for $\alpha\in\Sigma_J$, $\alpha(H_J^\perp)=0$ thus
$\alpha(H)=\alpha(H_J)+\alpha(H_J^\perp)=\alpha(H_J)$. Now we  derive
$$\chi_\lambda=\frac{\sum_{s\in W}\det s\ e^{(s\mu)(H_J^\perp)}\sum_{s_J\in W_J}\det s_J e^{(s_J(s\mu)_J)(H_J)}}{|W_J|\prod_{\alpha\in \Sigma^+\setminus \Sigma^+_J}\left(e^{\frac{\alpha(H)}{2}}-e^{-\frac{\alpha(H)}{2}}\right)\prod_{\alpha\in\Sigma^+_J}\left(e^{\frac{\alpha(H_J)}{2}}-e^{-\frac{\alpha(H_J)}{2}}\right)}$$
which is \eqref{char}, recalling Definition \ref{definitiondeltas} of the weight functions.
\end{proof}

Note in particular that the above formula \eqref{char} for characters holds for any $\mu\in i\mathfrak{t}^*$ which is not necessarily regular. Thus it takes care of both non-regular physical parameters (i.e. near walls) and non-regular spectral parameters in $i\mathfrak{t}^*$, which is an important point in analysis on higher-rank spaces. 

Before we leave this section, we list a few basic facts about Fourier analysis on a compact Lie group. In our previous discussion, $U$ is assumed a compact simple Lie group, but now we assume $U$ can also be a product of a compact semisimple Lie group and a torus $\mathbb{T}^n$. In the latter case, the space of spectral parameters is of the form $\Lambda^+\times \mathbb{Z}^n$ where $\Lambda^+$ is the set of (strictly) dominant weights associated to the semisimple component. Let $\chi_\mu$ and $d_\mu$ denote the character and dimension of representation for the spectral parameter $\mu$ respectively. Then we have the following version of Fourier series. For any $f\in L^2(U)$, 
$$f=\sum_\mu f*(d_\mu \chi_\mu),$$
and we have the Parseval's identity
$$\|f\|^2_{L^2(U)}=\sum_{\mu} \left\|f*(d_\mu \chi_\mu)\right\|^2_{L^2(U)}.$$
The map $f\mapsto f*(d_\mu \chi_\mu)$ is also the projection onto the eigenspace of spectral parameter $\mu$ for the conjugate-invariant differential operators. For any $L^2$ class function $\kappa$, we may define its Fourier transform $\hat{\kappa}(\mu)$ such that  
$$\kappa=\sum_\mu \hat{\kappa}(\mu) d_\mu\chi_\mu.$$ 
As a consequence, we have 
$$f*\kappa=\sum_\mu \hat{\kappa}(\mu) f*(d_\mu\chi_\mu),$$
and that 
$$\|f*\kappa\|_{L^2(U)}\leq\sup_\mu|\hat{\kappa}(\mu)|\cdot\|f\|_{L^2(U)}.$$

\section{$L^p$ norm of $1/\delta_{I,J}$}
\label{sectionLpnorm}
This section forms the technical heart of the paper. 
Let $I$ be a subset of $\{0,\ldots,r\}$ with $|I|=r$ and $J$ be a subset of $I$. Recall definition of the weight function 
$$\delta_{I,J}(H):=\prod_{\alpha\in \Sigma_I^+\setminus \Sigma^+_J}\left(e^{\frac{\alpha(H)}{2}}-e^{-\frac{\alpha(H)}{2}}\right).$$
As is clear from Lemma \ref{lemdecomposition}, it is crucial to estimate $\delta_{I,J}$ in order to estimate characters. We obtain sharp $L^p$-estimates for $1/\delta_{I,J}$ over the polytope $P_{I,J}$ in this section. We start from the following key lemma in root system theory. 

\begin{lem}[Appendix of \cite{ST78}]
Let $\Sigma$ be an irreducible root system and let $\{\alpha_j, \ j=1,\ldots,r\}$ be a simple system. Each $\alpha\in\Sigma^+$ is uniquely a linear sum of $\alpha_j, \ j=1,\ldots,r$ with positive coefficients, so that 
$$\prod_{\alpha\in\Sigma^+}\alpha=\sum c_{k_1,k_2,\ldots,k_r} \alpha_1^{k_1}\alpha_2^{k_2}\cdots\alpha_r^{k_r}$$
where the sum above is over nonnegative integral $r$-tuples 
$(k_1,k_2,\ldots,k_r)$ with $\sum_{i}k_i=|\Sigma^+|$, and $c_{k_1,\ldots,k_r}>0$. Then there exist positive integers $p_r>p_{r-1}>\cdots>p_1=1$ with $p_1+\cdots+p_r=|\Sigma^+|$ such that for any permutation $(j_1,\ldots,j_r)$ of $\{1,\ldots,r\}$, the term $\alpha_{j_r}^{p_r}\cdots \alpha_{j_1}^{p_1}$ appears in the above sum with non-vanishing coefficient. 
\end{lem}

We have the following immediate corollary. 
\begin{cor}\label{nk}
Let $\Sigma$ be an irreducible root system and let $\{\alpha_j, \ j=1,\ldots,r\}$ be a simple system. Then there exist positive integers $p_r>p_{r-1}>\cdots>p_1=1$ with $p_1+\cdots+p_r=|\Sigma^+|$ such that the following is true. Let $(j_1,\ldots,j_r)$ be any permutation of $\{1,\ldots,r\}$. 
For $k=1,\ldots,r$, let $I_k=\{j_1,\ldots,j_k\}$, and define 
$n_k$ to be the number of $\alpha\in\Sigma^+$ such that at least one of $\alpha_{j_r},\alpha_{j_{r-1}},\ldots,\alpha_{j_k}$ appears in the linear sum $\alpha=\sum_i c_i \alpha_{j_i}$ with non-vanishing coefficient, or equivalently, 
\begin{align}\label{nk=}
n_k=|\Sigma^+|-|\Sigma^+_{I_{k-1}}|,\ k=1,\ldots,r.
\end{align}
Here we specified $\Sigma^+_{I_{0}}$ to be the empty set. 
Define also the positive integers
\begin{align}\label{qk=nk-}
q_k=n_k-n_{k+1}=|\Sigma^+_{I_k}|-|\Sigma^+_{I_{k-1}}|,\ k=1,\ldots,r.
\end{align}
Then for $k=1,\ldots,r$,
$$q_r+q_{r-1}+\cdots+q_{k}=n_k\geq p_r+p_{r-1}+\cdots+p_k.$$
In particular, for $k=2,\ldots,r$,
$$q_r+q_{r-1}+\cdots+q_{k}=n_k > \frac{|\Sigma^+|\cdot (r-k+1)}{r},$$
while 
$$q_r+q_{r-1}+\cdots+q_1 =n_1=|\Sigma^+|.$$
\end{cor}

The next corollary will play a key rule in proving our main result in this section.

\begin{cor}\label{subsystem}
Let $\Sigma$ be an irreducible root system. 
For any nonempty subset $J$ of $\{0,\ldots,r\}$ such that $|J|\leq r-1$, we have 
$$\frac{|J|}{|\Sigma^+_J|}>\frac{r}{|\Sigma^+|}.$$ 
\end{cor}
\begin{proof}
We divide the proof into two cases. 

\underline{Case 1. $0\notin J$.} 
Let $(j_1,j_2,\ldots,j_r)$ be any permutation of $\{1,\ldots,r\}$ such that $(j_1,j_2,\ldots,j_{|J|})$ is a permutation of $J$. Let $n_{|J|+1}$ be the number of $\alpha\in\Sigma^+$ such that at least one of $\alpha_{j_r},\alpha_{j_{r-1}},\ldots,\alpha_{j_{|J|+1}}$ appears in the linear sum $\alpha=\sum_i c_i \alpha_{j_i}$ with non-vanishing coefficient. 
By Corollary \ref{nk}, we have 
$$|\Sigma^+|-|\Sigma^+_J|=n_{|J|+1}>\frac{|\Sigma^+|\cdot (r-|J|)}{r}.$$
The desired result follows. 

\underline{Case 2. $0\in J$.} 
A key observation is that by removing any node in the extended Dynkin diagram, 
each of the resulting connected subgraphs (the total number of which could be one, two, or three) is a subgraph of 
a connected subgraph of $r$ vertices, and in the language of root systems, this means that 
each of the resulting irreducible root subsystems, denoted $\Sigma_i$ ($i\in\{1\}\t{ or }\{1,2\}\t{ or }\{1,2,3\}$), is a subsystem of some irreducible $\Sigma_{I_i}$ with $I_i\in\{0,\ldots,r\}$ and $|I_i|=r$. 
Given $J\subset\{0,\ldots,r\}$ with $|J|\leq r-1$. The Dynkin diagram for $\Sigma_J$ may be obtained by removing at least two nodes in the extended Dynkin diagram for $\Sigma$. Thus we can write 
$$\Sigma_J=\bigsqcup_{i}\Sigma_{i,1},$$
$i\in\{1\}\t{ or }\{1,2\}\t{ or }\{1,2,3\}$, 
where $\Sigma_{i,1}$ is a root subsystem of $\Sigma_i$ and thus also a root subsystem of $\Sigma_{I_i}$. 
Denote the rank of $\Sigma_{i,1}$ by $r_i$. Since $r_i\leq |J|\leq r-1$, by Case 1, we have 
$$\frac{r_i}{|\Sigma_{i,1}^+|}> \frac{r}{|\Sigma^+_{I_i}|}\geq\frac{r}{|\Sigma^+|}.$$
We conclude 
$$\frac{|J|}{|\Sigma_J^+|}=\frac{\sum_i r_i}{\sum_i |\Sigma_{i,1}^+|} > \frac{r}{|\Sigma^+|}.$$ 
\end{proof}


\begin{cor}
Let $\Sigma$ be an irreducible root system and let $\{\alpha_j, \ j=1,\ldots,r\}$ be a simple system. Assume $t_j(H):=\alpha_j(H)/2\pi i>0$ for each $j=1,\ldots,r$ (which defines the Weyl chamber for $\Sigma$). 
For $J\subset\{1,\ldots,r\}$, assume 
$$0<t_{j_k}(H)\leq N^{-1}, \ k=1,2,\ldots,|J|,$$ 
for some fixed permutation $(j_1,\ldots,j_{|J|})$ of $J$, while $$t_{j_{r}}(H)>t_{j_{r-1}}(H)>\cdots>t_{j_{|J|+1}}(H)>N^{-1}$$ for some fixed permutation $(j_{|J|+1},\ldots, j_{r})$ of $\{1,\ldots,r\}\setminus J$. 
Then we have 
\begin{align}\label{keyestimatecor}
\prod_{\alpha\in\Sigma^+\setminus\Sigma^+_J}\alpha(H)/2\pi i\asymp t_{j_r}^{q_r}(H)t_{j_{r-1}}^{q_{r-1}}(H)\cdots t_{j_{|J|+1}}^{q_{|J|+1}}(H)
\end{align}
for the same positive integers $q_r,q_{r-1},\ldots,q_{|J|+1}$ 
as defined in \eqref{qk=nk-}. 
\end{cor}
\begin{proof}
For each $\alpha\in\Sigma^+$, $\alpha(H)/2\pi i$ is a linear sum of the $t_j(H)$'s ($j=1,\ldots,r$) with nonnegative coefficients. Let $n_j,q_j$ ($j=1,2,\ldots,r$) be defined as in \eqref{nk=} and \eqref{qk=nk-}. 
Write
$$\prod_{\alpha\in\Sigma^+\setminus \Sigma_J^+}\alpha(H)/2\pi i=\sum c_{k_1,\ldots,k_r}t_1(H)^{k_1}\cdots t_r(H)^{k_r}$$ where each $c_{k_1,\ldots,k_r}>0$. By our construction of the $q_j$'s, $t_{j_r}^{q_r}(H)t_{j_{r-1}}^{q_{r-1}}(H)\cdots t_{j_{|J|+1}}^{q_{|J|+1}}(H)$ appears in the above sum with non-vanishing coefficient, which explains the lower bound in \eqref{keyestimatecor}. On the other hand, all the other terms $t_1(H)^{k_1}\cdots t_r(H)^{k_r}$ with non-vanishing coefficient is bounded from above by a constant multiple of $t_{j_r}^{q_r}(H)t_{j_{r-1}}^{q_{r-1}}(H)\cdots t_{j_{|J|+1}}^{q_{|J|+1}}(H)$, as $t_{j_{r}}(H)>t_{j_{r-1}}(H)>\cdots>t_{j_{|J|+1}}(H)>N^{-1}$ are the first $r-|J|$ largest among the $t_j(H)$'s ($j=1,\ldots,r$), and all the largest terms (modulo multiplicative positive constants) in the linear sums for $\alpha(H)/2\pi i$ ($\alpha\in \Sigma^+\setminus \Sigma_J^+$) are picked in forming the term $t_{j_r}^{q_r}(H)t_{j_{r-1}}^{q_{r-1}}(H)\cdots t_{j_{|J|+1}}^{q_{|J|+1}}(H)$. 
\end{proof}

We state and prove another lemma concerning evaluation of a multiple integral, before we prove the main result of this section.

\begin{lem}\label{multipleintegral}
Let $m$ be a positive integer, and let $a_1,\ldots,a_m$ be positive real numbers. 
Let $s_0:=m$, and let $s_i$ ($i=1,2,\ldots$) be the smallest among $s=s_{i-1}-1,s_{i-1}-2,\ldots,0$ such that 
$$\frac{s_{i-1}-s}{a_{s_{i-1}}+a_{s_{i-1}-1}+\cdots+a_{s+1}}$$
obtains its maximum 
$$p_i:=\frac{s_{i-1}-s_i}{a_{s_{i-1}}+a_{s_{i-1}-1}+\cdots+a_{s_i+1}}.$$
Then there is some integer $k$ such that 
\begin{align}\label{m=}
m=s_0>s_1>s_2>\cdots>s_{k-1}>s_k=0,
\end{align}
and
\begin{align}\label{infty=}
+\infty=:p_0>p_1>p_2>\cdots>p_k>p_{k+1}:=0.
\end{align}
Moreover, let $c$ be a positive real number, and let $N$ be a positive real number such that $N^{-1}<c$. Then  
\begin{align}
&\mathcal{I}(a_1,\ldots,a_m;p,N):=\left(\int_{N^{-1}\leq t_1\leq t_2\leq\cdots\leq t_m\leq c}t_{m}^{-a_mp}\cdots t_2^{-a_2p}t_1^{-a_1p}\ dt_m\cdots\ dt_2\ dt_1\right)^{\frac{1}{p}} \nonumber \\
&\leq
\left\{
\begin{array}{ll}
C\cdot N^{a_1+\cdots+a_{s_i}-\frac{s_i}{p}}, & \text{ if }p_i>p>p_{i+1},\ i=0,\ldots,k.\\
C\cdot N^{a_1+\cdots+a_{s_i}-\frac{s_i}{p}}(\log cN)^{C}, & \text{ if }p=p_i,\ i=1,\ldots,k.\\
\end{array}
\right.
\end{align}
Here the constants $C$ in the above inequality may depend on $a_1,\ldots,a_m,p,c$, but not on $N$. 
As usual it is specified that $a_1+\cdots+a_{s_i}=0$ when $i=k$. In particular, in the above piecewise bound, except for the $\varepsilon$-sized jumps at the kink points, the exponent of $N$ as a function of $\frac{1}{p}$ is piecewise linear, non-increasing, continuous and convex.

\end{lem}
\begin{proof}
\eqref{m=} and \eqref{infty=} are immediate from definition. We prove the integral estimate by induction on the number $m$. The $m=1$ case is clear. For an integer $l\geq 2$, assume the cases for  $m=1,2,\ldots,l-1$ are valid. For the case $m=l$, let $s_1$ be the smallest $s$ among $s=l-1,l-2,\ldots,0$ such that 
$$\frac{l-s}{a_l+a_{l-1}+\cdots+a_{s+1}}$$
obtains its maximum 
$$p_1:=\frac{l-s_1}{a_l+a_{l-1}+\cdots+a_{s_1+1}}.$$
The definition of $s_1$ implies that 
\begin{align}\label{firstinductivestep1}
\frac{l-s}{a_{l}+a_{l-1}+\cdots+ a_{s+1}} < 
\frac{l-s_1}{a_l+a_{l-1}+\cdots+ a_{s_1+1}}=p_1, \ s= s_1-1,s_1-2,\ldots,0,
\end{align}
which is equivalent to 
\begin{align}\label{firstinductivestep}
\frac{s_1-s}{a_{s_1}+a_{s_1-1}+\cdots+ a_{s+1}} < 
\frac{l-s_1}{a_l+a_{l-1}+\cdots+ a_{s_1+1}}=p_1, \ s= s_1-1,s_1-2,\ldots,0.
\end{align}
The definition of $s_1$ of course also implies that 
\begin{align}\label{firstinductivestep2}
\frac{l-s}{a_l+a_{l-1}+\cdots+ a_{s+1}}\leq  
\frac{l-s_1}{a_l+a_{l-1}+\cdots+ a_{s_1+1}}=p_1, \ s= s_1+1,s_1+2,\ldots,l,
\end{align}
which is equivalent to 
\begin{align}\label{firstinductivestep3}
\frac{s-s_1}{a_{s}+a_{s-1}+\cdots+ a_{s_1+1}} \geq  
\frac{l-s_1}{a_l+a_{l-1}+\cdots+ a_{s_1+1}}=p_1, \ s= s_1+1,s_1+2,\ldots,l.
\end{align}
Now for $p>p_1$, using \eqref{firstinductivestep1} and \eqref{firstinductivestep2}, we have
\begin{align}\label{Ia11}
\mathcal{I}(a_1,\ldots,a_l;p,N)
\leq \left(\frac{N^{(a_1+\cdots+a_l)p-l}}{\prod_{j=1}^l((a_l+a_{l-1}+\cdots+a_{l-j+1})p-j)}\right)^{\frac{1}{p}}
\lesssim N^{a_1+\cdots+a_l-\frac{l}{p}},
\end{align}
with the implicit constant not depending on $N$. 

\underline{Case I. $s_1= 0$.} Then 
\begin{align}\label{p1=l}
p_1=\frac{l}{a_l+a_{l-1}+\cdots+a_1}.
\end{align}
Because of \eqref{Ia11}, it suffices to prove that
\begin{align}\label{Ia12}
\mathcal{I}(a_1,\ldots,a_l;p,N)
\leq\left\{
\begin{array}{ll}
C\cdot (\log cN)^C, &\text{ if }p=p_1,\\
C, & \text{ if }0<p<p_1. 
\end{array}
\right.
\end{align}
where the constants $C$ do not depend on $N$. We now fix $p\in (0, p_1]$. 
Because of \eqref{p1=l}, there exists some $u=0,\ldots,l-1$, such that 
\begin{align}\label{p>frac}
p>\frac{l-s}{a_l+a_{l-1}+\cdots+a_{s+1}}, \ s=u+1,u+2,\ldots,l-1,
\end{align}
while
\begin{align}\label{pleq}
p\leq \frac{l-u}{a_l+a_{l-1}+\cdots+a_{u+1}}.
\end{align}

\underline{Case I-1. $u=0$.} Then a direct calculation tells that 
$$\mathcal{I}(a_1,\ldots,a_l;p,N)
=\left(
\frac{1}{\prod_{j=1}^{l-1}((a_l+a_{l-1}+\cdots+a_{l-j+1})p-j)}
\int_{N^{-1}\leq t_1\leq c}
t_1^{-(a_1+\cdots+a_l)p+(l-1)}\ dt_1
\right)^{\frac{1}{p}},$$
which clearly implies \eqref{Ia12}. 

\underline{Case I-2. $u\neq 0$.} We
use
\begin{align}\label{Ia13}
\mathcal{I}(a_1,\ldots,a_l;p,N)
= \left(\int_{N^{-1}\leq t_1\leq \cdots\leq t_u\leq c}\mathcal{I}(a_{u+1},\ldots,a_l;p,t_u^{-1})^p\cdot t_u^{-a_up}\cdots t_1^{-a_1p}\ dt_1\cdots dt_u\right)^{\frac{1}{p}}.
\end{align}
By the induction hypothesis, because of \eqref{pleq} and \eqref{p>frac}, we have
\begin{align}\label{Iau+1}
\mathcal{I}(a_{u+1},\ldots,a_l;p,t_u^{-1})
\leq \left\{
\begin{array}{ll}
C\cdot (\log ct_u^{-1})^C, &\text{ if }p=\frac{l-u}{a_l+a_{l-1}+\cdots+a_{u+1}},\\
C, & \text{ if }0<p<\frac{l-u}{a_l+a_{l-1}+\cdots+a_{u+1}}. 
\end{array}
\right.
\end{align}

\underline{Case I-2-1. $p=p_1$.} Then \eqref{firstinductivestep2} and \eqref{pleq} together tell that 
\begin{align}\label{p1equal}
p_1=\frac{l-u}{a_l+a_{l-1}+\cdots+a_{u+1}}.
\end{align}
In particular, by \eqref{Iau+1}, 
\begin{align}\label{Iau+11}
\mathcal{I}(a_{u+1},\ldots,a_l;p,t_u^{-1})
\leq C\cdot (\log ct_u^{-1})^C\leq C\cdot (\log cN)^C.
\end{align}
By the induction hypothesis again, there exists some $u'=1,\ldots,u$, such that 
\begin{align}\label{Ia14}
\mathcal{I}(a_1,\ldots,a_u;p,N)
\leq \left\{
\begin{array}{ll}
C\cdot (\log cN)^C, &\text{ if }p=\frac{u'}{a_{u'}+a_{u'-1}+\cdots+a_{1}},\\
C, & \text{ if }0<p<\frac{u'}{a_{u'}+a_{u'-1}+\cdots+a_{1}}. 
\end{array}
\right.
\end{align}
Now \eqref{firstinductivestep3} implies that 
\begin{align}\label{p1leq1}
p_1\leq \frac{u'}{a_{u'}+a_{u'-1}+\cdots+a_{1}}.
\end{align}
By \eqref{Ia13}, \eqref{Iau+11},  \eqref{Ia14}, and \eqref{p1leq1}, we have
\begin{align*}
\mathcal{I}(a_1,\ldots,a_l;p_1,N)
\leq C\cdot (\log cN)^C.
\end{align*}

\underline{Case I-2-2. $p<p_1$.} For any $\varepsilon>0$, we have 
$$(\log ct_u^{-1})^C\leq C_{\varepsilon} t_u^{-\varepsilon}$$
for all $0<t_u\leq c$. This combined with \eqref{pleq}, \eqref{Ia13} and \eqref{Iau+1} gives 
\begin{align}\label{mathcalIleq}
\mathcal{I}(a_1,\ldots,a_l;p,N)
\leq C_\varepsilon \mathcal{I}(a_1,\ldots,a_{u-1},a_{u}+\varepsilon;p,N). 
\end{align}
By the induction hypothesis again, there exists some $u'=1,\ldots,u$, such that 
\begin{align}\label{Ia15}
\mathcal{I}(a_1,\ldots,a_{u-1},a_u+\varepsilon;p,N)
\leq \left\{
\begin{array}{ll}
C\cdot (\log cN)^C, &\text{ if }p=\frac{u'}{a_{u'}+\delta_{u'u}\varepsilon+a_{u'-1}+\cdots+a_{1}},\\
C, & \text{ if }0<p<\frac{u'}{a_{u'}+\delta_{u'u}\varepsilon+a_{u'-1}+\cdots+a_{1}}. 
\end{array}
\right.
\end{align}
Here $\delta_{u'u}$ equals 1 if $u'=u$ and 0 otherwise. 
Again \eqref{firstinductivestep3} implies that 
\begin{align*}
p_1\leq \frac{u'}{a_{u'}+a_{u'-1}+\cdots+a_{1}}.
\end{align*}
As $p<p_1$, there is some $\varepsilon=\varepsilon(p)>0$, such that 
$$p<\frac{u'}{a_{u'}+\delta_{u'u}\varepsilon+a_{u'-1}+\cdots+a_{1}}.$$
This inequality combined with \eqref{mathcalIleq} and 
\eqref{Ia15} implies that 
\begin{align*}
\mathcal{I}(a_1,\ldots,a_l;p,N)
\leq C
\end{align*}
for some constant $C$ independent of $N$. 

\underline{Case II. $s_1\neq 0$.} In this case $s_1=1,\ldots,l-1$. We use 
\begin{align}\label{Ia16}
\mathcal{I}(a_1,\ldots,a_l;p,N)
\leq \mathcal{I}(a_1,\ldots,a_{s_1};p,N)\cdot  \mathcal{I}(a_{s_1+1},\ldots,a_{l};p,N),
\end{align}
the proof of which is immediate. By the induction hypothesis, we have 
\begin{align}\label{Ias1+1}
\mathcal{I}(a_{s_1+1},\ldots,a_{l};p,N)
\leq \left\{
\begin{array}{ll}
C\cdot (\log cN)^C, & \text{ if }p=p_1,\\
C, & \text{ if }0<p<p_1.
\end{array}
\right.
\end{align}
Using \eqref{Ia16} and \eqref{Ias1+1}, and applying the induction hypothesis to $\mathcal{I}(a_1,\ldots,a_{s_1};p,N)$, we finish the proof. 
\end{proof}

We are ready to prove the following main result of this section. 

\begin{prop} \label{keyprop}
Let $\Sigma$ be an irreducible root system of rank $r$. 
For $I\subset\{0,\ldots,r\}$, $|I|=r$, $J\subset I$, we have 
\begin{align*}
\left\|\frac{1}{\delta_{I,J}}\right\|_{L^p(P_{I,J})}
\lesssim 
\left\{
\begin{array}{ll}
N^{|\Sigma^+|-|\Sigma_J^+|-\frac{r}{p}}, & \t{ for }p>\frac{r}{|\Sigma^+|},\\
N^{-|\Sigma_J^+|}, & \t{ for }0<p<\frac{r}{|\Sigma^+|}.
\end{array}
\right.
\end{align*}
\end{prop}
\begin{proof}
We use the sets $\mathcal{N}_{I'}$ ($I'\subset I$) as constructed in the proof of Lemma \ref{deltaIdeltaJ}. As $P_{I,J}= \bigcup_{I'\subset I}P_{I,J}\cap \mathcal{N}_{I'}$, it suffices to obtain the desired bound replacing $P_{I,J}$ by $P_{I,J}\cap\mathcal{N}_{I'}$ for each $I'\subset I$. We may assume $J\subset I'$ since otherwise $P_{I,J}\cap\mathcal{N}_{I'}$ is empty. To treat the $L^p$ bound of $1/\delta_{I,J}$ on $P_{I,J}\cap\mathcal{N}_{I'}$, first recall that $\mathcal{N}_{I'}$ stays a fixed distance from all root hyperplanes $\mathfrak{p}_{\alpha,n}$ for $\alpha\notin \Sigma_{I'}$, thus for $H\in P_{I,J}\cap\mathcal{N}_{I'}$, we have
\begin{align}\label{deltaIJI'}
|\delta _{I,J}(H)|\asymp \prod_{\alpha\in \Sigma^{+}_{I'}\setminus \Sigma^{+}_J}\alpha(H)/2\pi i.
\end{align}
We now divide into steps. 

\underline{Step 1. We first treat the case when $\Sigma_I$ is irreducible and $I'=I$.}  We remark that the results and their proof in Step 1 will not depend on whether the original root system $\Sigma$ is irreducible nor not. 
Recall that $\{t_j,\ j=1,\ldots,r\}$ provide a coordinate system for $P_{I,J}$ on which $0\leq t_j\leq N^{-1}$ for any $j\in J$ and $1>c>t_j>N^{-1}$ for any $j\in I\setminus J$; see \eqref{PIJcoordinates}. 
We have 
\begin{align}\label{PIJcapNIsubset}
P_{I,J}\cap\mathcal{N}_I\subset \bigcup_{\substack{(j_{|J|+1},\ldots,j_{r})\t{ a}\\ \t{permutation of }I\setminus J}} \mathcal{R}_{j_{|J|+1},\ldots,j_r},
\end{align}
where 
$$\mathcal{R}_{j_{|J|+1},\ldots,j_r}=\left\{\substack{0\leq t_j\leq N^{-1},\ j\in J\\ N^{-1}<t_{j_{|J|+1}}\leq\cdots\leq t_{j_r}\leq c}\right\}.$$ 
By \eqref{keyestimatecor} and \eqref{deltaIJI'}, for $H\in \mathcal{R}_{j_{|J|+1},\ldots,j_r}$, we have 
$$|\delta_{I,J}(H)| \asymp \prod_{\alpha\in\Sigma_I^+\setminus \Sigma_J^+}\alpha(H)/2\pi i \asymp t_{j_r}^{q_r}(H)t_{j_{r-1}}^{q_{r-1}}(H)\cdots t_{j_{|J|+1}}^{q_{|J|+1}}(H),$$ 
where the $q_k$'s ($k=|J|+1,|J|+2,\ldots,r$) are defined as in \eqref{qk=nk-}, i.e., 
\begin{align*}\
q_k=|\Sigma^+_{I_k}|-|\Sigma^+_{I_{k-1}}|
\end{align*}
where $I_k=J\cup \{j_{|J|+1},j_{|J|+2},\ldots,j_k\}$. 
We evaluate 
\begin{align}
\left\| \frac{1}{\delta_{I,J}}\right\|_{L^p(\mathcal{R}_{j_{|J|+1},\ldots,j_r})}
&\lesssim
\left(\int_{\substack{0\leq t_j\leq N^{-1},\ j\in J\\ N^{-1}<t_{j_{|J|+1}}\leq\cdots\leq t_{j_r}\leq c}} t_{j_r}^{-q_rp}\cdots t_{j_{|J|+1}}^{-q_{|J|+1}p}\ dt_{1}\cdots dt_{r}\right)^{\frac{1}{p}} \nonumber \\
&\lesssim
\left(\int_{ N^{-1}<t_{j_{|J|+1}}\leq\cdots\leq t_{j_{r}}\leq c} t_{j_{r}}^{-q_rp}\cdots t_{j_{|J|+1}}^{-q_{|J|+1}p}\ dt_{j_{r}}\cdots dt_{j_{|J|+1}}\right. \nonumber \\
&\left.\ \ \ \ \ \ \cdot \int_{0\leq t_j\leq N^{-1},\ j\in J}\prod_{j\in J}dt_{j} \right)^{\frac{1}{p}}\nonumber \\ 
&\lesssim N^{-\frac{|J|}{p}}\cdot
\mathcal{I}(q_{|J|+1},\ldots,q_r;p,N). \nonumber
\end{align}
Here we have used the notation in Lemma \ref{multipleintegral}, and by which we find a finite sequence $s_i$ ($i=0,\ldots,k$) of integers with $$r=:s_0>s_1>s_2>\cdots >s_{k-1}>s_k:=|J|$$ 
such that for 
$$I_{s_i}=J\cup\{j_{|J|+1},\ldots,j_{s_i}\},\ i=0,\ldots,k,$$ 
and 
$$p_i=\frac{s_{i-1}-s_{i}}{q_{s_{i-1}}+q_{s_{i-1}-1}+\cdots+q_{s_i+1}}=\frac{s_{i-1}-s_i}{|\Sigma^+_{I_{s_{i-1}}}|-|\Sigma^+_{I_{s_{i}}}|}, \ i=1,\ldots,k,$$
it holds 
$$+\infty=:p_0>p_1>p_2>\cdots>p_k>p_{k+1}:=0$$
and 
\begin{align}\label{firstcase1}
\left\| \frac{1}{\delta_{I,J}}\right\|_{L^p(\mathcal{R}_{j_{|J|+1},\ldots,j_r})}
\lesssim 
\left\{
\begin{array}{ll}
N^{|\Sigma_{I_{s_{i}}}^+|-|\Sigma^+_J|-\frac{s_{i}}{p}}, & \text{ if }p_i>p>p_{i+1},\ i=0,\ldots,k,\\
N^{|\Sigma_{I_{s_{i}}}^+|-|\Sigma^+_J|-\frac{s_{i}}{p}+\varepsilon}, & \text{ if }p=p_i,\ i=1,\ldots,k.
\end{array}
\right.
\end{align}
Here the $\varepsilon$'s account for the $(\log cN)^C$ factor. 
In particular, except for the $\varepsilon$-sized jumps at the kink points, the above exponent of $N$ as a function of $\frac{1}{p}$ is piecewise linear, non-increasing, continuous and convex. Because of \eqref{PIJcapNIsubset}, $\left\|\frac{1}{\delta_{I,J}}\right\|_{L^p(P_{I,J}\cap\mathcal{N}_I)}$ is bounded by a finite sum of piecewise bounds as in \eqref{firstcase1}, so that 
\begin{align}\label{firstcase}
\left\| \frac{1}{\delta_{I,J}}\right\|_{L^p(P_{I,J}\cap \mathcal{N}_I)}
\lesssim \sum 
\left\{
\begin{array}{ll}
N^{|\Sigma_{I_{s_{i}}}^+|-|\Sigma^+_J|-\frac{s_{i}}{p}}, & \text{ if }p_i>p>p_{i+1},\ i=0,\ldots,k,\\
N^{|\Sigma_{I_{s_{i}}}^+|-|\Sigma^+_J|-\frac{s_{i}}{p}+\varepsilon}, & \text{ if }p=p_i,\ i=1,\ldots,k.
\end{array}
\right.
\end{align}

\underline{Step 2. We get a similar estimate as \eqref{firstcase} for the case when $I'=I$ and $\Sigma_I$ is possibly reducible.} By removing the node in the extended Dynkin diagram not belonging to $I$, we obtain the Dynkin diagram for the root system $\Sigma_I$. Checking Figure \ref{dynkin}, $\Sigma_I$ may be irreducible, which was already treated in Step 1, or a product of two or three irreducible root systems. We now  demonstrate the necessary modifications for the argument when $\Sigma_I$ is a product of two irreducibles, and the case of three irreducibles may be treated similarly. For example, in Figure \ref{alcoves} on rank-two root systems it can be seen that for types $B_2$ and $G_2$ the reducible rank-two root system of product type $A_1\times A_1$ appears as one of these $\Sigma_I$'s. Now suppose $\Sigma_I=\Sigma_{I^1}\bigsqcup\Sigma_{I^2}$ where $\Sigma_{I^1}$ and $\Sigma_{I^2}$ are nonempty, irreducible and orthogonal to each other, with $I=I^1\bigsqcup I^2$. Associated to this is also the direct sum of the ambient linear spaces $\mathfrak{t}=\mathfrak{t}_1\oplus\mathfrak{t}_2$. 
Let $J^l=I^l\bigcap J$ ($l=1,2$), then $J=J^1\bigsqcup J^2$. 
The polytope $P_{I,J}$ is now more or less the orthogonal product $P_{I^1,J^1}\times P_{I^2,J^2}$; more precisely, 
by \eqref{PIJcoordinates}, $P_{I,J}$ is contained in the product domain 
\begin{align*}
&\underbrace{\{H_1\in\mathfrak{t}_1:\ 0\leq t_j(H_1)\leq N^{-1}\ \forall j\in J^1, \ N^{-1}< t_j(H_1)\leq c\ \forall j\in I^1\setminus J^1\}}_{=:P'_{I^1,J^1}}\\
\times &\underbrace{\{H_2\in\mathfrak{t}_2:\ 0\leq t_j(H_2)\leq N^{-1}\ \forall j\in J^2, \ N^{-1}< t_j(H_2)\leq c\ \forall j\in I^2\setminus J^2\}}_{=:{P'_{I^2,J^2}}}
\end{align*}
as well as contains such a domain with a smaller constant $c$. The vertex $A_I$ may be expressed as $A_{I^1}\times A_{I^2}$ where $A_{I^l}$ is a point in $\mathfrak{t}_l$ ($l=1,2$), and the neighborhood $\mathcal{N}_I$ may be contained in a product $\mathcal{N}_{I^1}\times\mathcal{N}_{I^2}$ where $\mathcal{N}_{I^l}$ is a neighborhood of $A_{I^l}$ in the space $\mathfrak{t}_l$ ($l=1,2$). 
With the positive systems also decomposed as $\Sigma_I^+=\Sigma_{I^1}^+\bigsqcup\Sigma_{I^2}^+$, $\Sigma_J^+=\Sigma_{J^1}^+\bigsqcup\Sigma_{J^2}^+$, we have $\delta_{I,J}=\delta_{I^1,J^1}\cdot \delta_{I^2,J^2}$. 
Then 
\begin{align}\label{step4product}
\left\|\frac{1}{\delta_{I,J}}\right\|_{L^p(P_{I,J}\cap\mathcal{N}_I)}\leq \left\|\frac{1}{\delta_{I^1,J^1}}\right\|_{L^p(P'_{I^1,J^1}\cap\mathcal{N}_{I^1})}\left\|\frac{1}{\delta_{I^2,J^2}}\right\|_{L^p(P'_{I^2,J^2}\cap\mathcal{N}_{I^2})}.
\end{align}
Apply the result in Step 1, we obtain that for $l=1,2$, 
\begin{align*}
\left\|\frac{1}{\delta_{I^l,J^l}}\right\|_{L^p(P_{I^l,J^l}\cap \mathcal{N}_{I^l})}
\lesssim 
\sum
\left\{
\begin{array}{ll}
N^{|\Sigma_{I^l_{s_{l,i}}}^+|-|\Sigma^+_{J^l}|-\frac{s_{l,i}}{p}}, & \text{ if }p_{l,i}>p>p_{l,i+1},\ i=0,1,\ldots,k_l,\\
N^{|\Sigma_{I^l_{s_{l,i}}}^+|-|\Sigma^+_{J^l}|-\frac{s_{l,i}}{p}+\varepsilon}, & \text{ if }p=p_{l,i},\ i=1,\ldots,k_l.
\end{array}
\right.
\end{align*}
It follows from \eqref{step4product} that 
\begin{align}\label{productsystemnearorigin}
\left\|\frac{1}{\delta_{I,J}}\right\|_{L^p(P_{I,J}\cap\mathcal{N}_I)}
\lesssim 
\sum
\left\{
\begin{array}{ll}
N^{|\Sigma_{I_{s_{i}}}^+|-|\Sigma^+_J|-\frac{s_{i}}{p}}, & \text{ if }p_i>p>p_{i+1},\ i=0,\ldots,k,\\
N^{|\Sigma_{I_{s_{i}}}^+|-|\Sigma^+_J|-\frac{s_{i}}{p}+\varepsilon}, & \text{ if }p=p_{i},\ i=1,\ldots,k.
\end{array}
\right.
\end{align}
in exactly the same form as \eqref{firstcase}. We remark that as in  Step 1, the results and their proof in Step 2 do not depend on whether the original root system $\Sigma$ is irreducible or not.
\begin{figure}
\scalebox{0.25}[0.25]
{
\begin{tikzpicture}

\draw (0, 0) -- (4, 6.928);

\draw (0,0) -- (8, 0);

\draw  (4, 6.928) -- (8, 0);

\node [above] at (4, 6.928) {\scalebox{3}{$A_2$}};
\node [left] at (0, 0) {\scalebox{3}{$A_2$}};
\node [right] at (8, 0) {\scalebox{3}{$A_2$}};

\node [below] at (4, -1) {\scalebox{4}{Type $A_2$}};

\end{tikzpicture}
\hspace{2cm}
\begin{tikzpicture}

\draw (0, 0) -- (7,7);
\draw (0,0) -- (7, 0);
\draw  (7,7) -- (7, 0);

\node [above] at (7, 7) {\scalebox{3}{$B_2$}};
\node [left] at (0, 0) {\scalebox{3}{$B_2$}};
\node [right] at (7, 0) {\scalebox{3}{$A_1\times A_1$}};

\node [below] at (3.5, -1) {\scalebox{4}{Type $B_2$}};

\end{tikzpicture}
\hspace{2cm}
\begin{tikzpicture}

\draw (0, 0) -- (8, 4.6188);
\draw (0,0) -- (8, 0);
\draw  (8, 4.6188) -- (8, 0);

\node [above] at (8, 4.6188) {\scalebox{3}{$A_2$}};
\node [left] at (0, 0) {\scalebox{3}{$G_2$}};
\node [right] at (8, 0) {\scalebox{3}{$A_1\times A_1$}};

\node [below] at (4, -1) {\scalebox{4}{Type $G_2$}};

\end{tikzpicture}
}
\caption{Alcoves and parabolic subsystems}
\label{alcoves}
\end{figure}

\underline{Step 3. We now estimate 
$\left\|\frac{1}{\delta_{I,J}}\right\|_{L^p(P_{I,J}\cap\mathcal{N}_{I'})}$ for a general $I'\subset I$.} It is instructive to picture the following argument using the root system $A_3$ with $I={1,2,3}$ and $I'={1,2}$. We need to give a coordinate system for $P_{I,J}\cap \mathcal{N}_{I'}$. First for $H\in\mathfrak{t}$, recall the notations in \eqref{tJ}, \eqref{HJperp} and \eqref{tJperp}, we write 
$$\mathfrak{t}=\mathfrak{t}_{I'}\oplus\mathfrak{t}_{I'}^\perp, \ H=H_{I'}+H_{I'}^\perp.$$
Keeping in mind that $J\subset I'$, the region $P_{I,J}\cap\mathcal{N}_{I'}$ may be covered by a region in the form of  
$$\{H=H_{I'}+H_{I'}^\perp\ | \ 
H_{I'}\in P_{I',J}\cap \mathcal{N}'_{I'}, \ 
H_{I'}^\perp\in\mathcal{N}^\perp_{I'}\}$$
where $P_{I',J}$ is defined for the root system $\Sigma_{I'}$ as in \eqref{PIJdefinition}, $\mathcal{N}'_{I'}$ is defined as a neighborhood of the origin in $\mathfrak{t}_{I'}$ and $\mathcal{N}^\perp_{I'}$ is defined as a neighborhood in $\mathfrak{t}_{I'}^\perp$.
Now \eqref{deltaIJI'} tells that 
$$|\delta_{I,J}(H)|\asymp |\delta_{I',J}(H_{I'})|.$$
Then we have 

\begin{align*}
\left\|\frac{1}{\delta_{I,J}}\right\|_{L^p(P_{I,J}\cap\mathcal{N}_{I'})}
&\lesssim 
\left(\int_{H_{I'}\in P_{I',J}\cap \mathcal{N}'_{I'}}\left|\frac{1}{\delta_{I',J}(H_{I'})}\right|^p\int_{\substack{
(\alpha(H_{I'}^\perp)+\alpha(H_{I'}))/2\pi i >  N^{-1}\text{ for }\alpha\in I\setminus I' \\
H_{I'}^\perp\in\mathcal{N}^\perp_{I'}
}} \ dH_{I'}^\perp\ dH_{I'}\right)^{\frac{1}{p}}
\\
&\lesssim \left\|\frac{1}{\delta_{I',J}}\right\|_{L^p(P_{I',J}\cap\mathcal{N}'_{I'})}.
\end{align*}
Here we have used the fact that the $\alpha(H_{I'}^\perp)$'s ($\alpha\in I\setminus I'$) form a basis of linear functions on $\mathfrak{t}_{I'}^\perp$. Applying the result in Step 2 to bound 
$\left\|\frac{1}{\delta_{I',J}}\right\|_{L^p(P_{I',J}\cap\mathcal{N}'_{I'})}$, we get 
\begin{align}\label{laststepbound1}
 \left\|\frac{1}{\delta_{I,J}}\right\|_{L^p(P_{I,J}\cap\mathcal{N}_{I'})}
 \lesssim \sum
\left\{
\begin{array}{ll}
N^{|\Sigma_{I'_{s_{i}}}^+|-|\Sigma^+_J|-\frac{s_{i}}{p}}, & \text{ if }p_i>p>p_{i+1},\ i=0,\ldots,k,\\
N^{|\Sigma_{I'_{s_{i}}}^+|-|\Sigma^+_J|-\frac{s_{i}}{p}+\varepsilon}, & \text{ if }p=p_i,\ i=1,\ldots,k,
\end{array}
\right.
\end{align}
in exactly the same form as \eqref{firstcase}. 

\underline{Last step.} To prove the proposition, it suffices to show that each of piecewise summand 
\begin{align}\label{laststepbound}
\left\{
\begin{array}{ll}
N^{|\Sigma_{I'_{s_{i}}}^+|-|\Sigma^+_J|-\frac{s_{i}}{p}}, & \text{ if }p_i>p>p_{i+1},\ i=0,\ldots,k,\\
N^{|\Sigma_{I'_{s_{i}}}^+|-|\Sigma^+_J|-\frac{s_{i}}{p}+\varepsilon}, & \text{ if }p=p_i,\ i=1,\ldots,k,
\end{array}
\right.
\end{align}
in \eqref{laststepbound1} is further bounded by the desired bound
\begin{align}\label{desiredboundbound}
\left\{
\begin{array}{ll}
N^{|\Sigma^+|-|\Sigma_J^+|-\frac{r}{p}}, & \t{ for }p>\frac{r}{|\Sigma^+|},\\
N^{-|\Sigma_J^+|}, & \t{ for }0<p<\frac{r}{|\Sigma^+|}.
\end{array}
\right.
\end{align}

\underline{Evaluating at $p=\infty$.} Then the exponent of $N$ in \eqref{laststepbound} becomes $|\Sigma^+_{I'}|-|\Sigma^+_{J}|$, while the above desired exponent $|\Sigma^+|-|\Sigma^+_J|-\frac{r}{p}$ becomes $|\Sigma^+|-|\Sigma^+_J|$.
We indeed have 
\begin{align*}
|\Sigma^+_{I'}|-|\Sigma^+_{J}|\leq |\Sigma^+|-|\Sigma^+_{J}|.
\end{align*}

\underline{Evaluating at $p=\frac{r}{|\Sigma^+|}$.} Ignoring the  $\varepsilon$-sized jumps, the exponent of $N$ in \eqref{laststepbound} becomes 
$$|\Sigma^+_{I'_{s_i}}|-|\Sigma^+_{J}|-\frac{s_i}{r}\cdot |\Sigma^+|$$ for some $i$, while the desire exponent in \eqref{desiredboundbound} evaluates to be 
$-|\Sigma_J^+|$. It indeed holds that 
\begin{align}\label{rootsysteminequality}
|\Sigma^+_{I'_{s_i}}|-|\Sigma^+_{J}|-\frac{s_i}{r}\cdot |\Sigma^+|\leq -|\Sigma_J^+|,
\end{align}
as a consequence of Corollary \ref{subsystem}, noting that $|I'_{s_i}|=s_i$. Here we used crucially the irreducibility of $\Sigma$.

As observed in Step 1 by the result of Lemma \ref{multipleintegral}, the exponents of $N$ in \eqref{laststepbound} form a piecewise linear, convex, and continuous function of $\frac{1}{p}$ except for the $\varepsilon$-sized jumps at the kink points. 
We have shown that the piecewise exponent function in \eqref{laststepbound} is bounded at the two points $\frac{1}{p}=0$ and $\frac{1}{p}=\frac{|\Sigma^+|}{r}$ by $|\Sigma^+|-|\Sigma^+_J|-\frac{r}{p}$ which is linear in $\frac{1}{p}$, ignoring the $\varepsilon$-sized jumps at the kink points. This is enough for us to conclude that for all $p>\frac{r}{|\Sigma^+|}$, the exponent of $N$ in \eqref{laststepbound} is bounded by the desired exponent $|\Sigma^+|-|\Sigma^+_J|-\frac{r}{p}$ for $p>\frac{r}{|\Sigma^+|}$. 

At last, we have the fact that the exponent of $N$ in \eqref{laststepbound} is a non-increasing function of $\frac{1}{p}$ except for the $\varepsilon$-jumps at the kink points, and the fact that this exponent function is bounded by $-|\Sigma^+_J|$ at $\frac{1}{p}=\frac{|\Sigma^+|}{r}$ ignoring the $\varepsilon$-jumps at the kink points. These facts are enough for us to conclude that for all $p<\frac{r}{|\Sigma^+|}$, the exponent of $N$ in \eqref{laststepbound}  is bounded by the constant $-|\Sigma^+_J|$. The proof is finished. 

\end{proof}

\begin{rem}\label{productweight}
We have a natural extension of Proposition \ref{keyprop} to reducible root systems $\Sigma$. Let $\Sigma_l$ ($l=1,\ldots,k$) be irreducible root systems of rank $r_l$ and consider the product root system $\Sigma:=\bigsqcup_l\Sigma_l$ of rank $\sum_{l}r_l$. Let $\Sigma_l^+$ be a positive system of $\Sigma_l$ and then $\Sigma^+=\bigsqcup_l\Sigma_l^+$ is a positive system of $\Sigma$. The associated Weyl alcove $A$ to $\Sigma$ may be defined as the product $A:=A_1\times \cdots\times A_k$ where $A_l$ is the alcove for $\Sigma_l$. Let $I_l$ denote a subset of $\{0,\ldots,r_l\}$ such that $|I_l|=r_l$ and let $J_l$ be a subset of $I_l$. Let $I:=\bigsqcup_l I_l $ and $J:=\bigsqcup_l J_l$. 
Set 
$$C_I:=C_{I_1}\times\cdots \times C_{I_k},$$ 
$$P_J:=P_{J_1}\times\cdots \times P_{J_k},$$ 
$$P_{I,J}:=P_J\cap C_I=P_{I_1,J_1}\times\cdots\times P_{I_k,J_k}.$$
Set $\delta_I:=\prod_l\delta_{I_l}$, $\delta^J:=\prod_l\delta^{J_l}$, and $\delta_{I,J}:=\prod_l\delta_{I_l,J_l}$. 
Then Lemma \ref{deltaIdeltaJ} may be generalized to this setting without change, and Proposition \ref{keyprop} implies that 
$$\left\|\frac{1}{\delta_{I,J}}\right\|_{L^p(P_{I,J})}\lesssim N^{|\Sigma^+|-|\Sigma_J^+|-\frac{r}{p}},\t{ for }p>\max\left\{\frac{r_l}{|\Sigma_l^+|}, \ l=1,\ldots,k\right\}.$$
\end{rem}

\section{Projections of the weight lattice}\label{projectionweightlattice}
For each proper subset $J$ of $\{0,\ldots,r\}$, recall that $V_J$ denotes the $\R$-subspace of $i\mathfrak{t}^*$ spanned by $\Sigma_J$, $\Lambda_J$ denotes the weight lattice associated to $\Sigma_J$, and $\t{Proj}_{V_J}:i\mathfrak{t}^*\to i\mathfrak{t}^*$ denotes orthogonal projection onto $V_J$. With Lemma \ref{lemdecomposition} in mind, to examine the behavior of the characters, we still need to understand the image set 
$\t{Proj}_{V_J}(\Lambda)$ in detail. We have the following characterization. Let $\Gamma:=\t{span}_\mathbb{Z}\Sigma$, $\Gamma_J:=\t{span}_\Z\Sigma_J$ be the root lattices generated by the root systems $\Sigma$, $\Sigma_J$ respectively. 

\begin{lem}\label{projtosub}
We have $\Lambda_J\supset\t{Proj}_{V_J}(\Lambda)\supset\Gamma_J$. 
\end{lem}
\begin{proof}
$\Lambda_J\supset\t{Proj}_{V_J}(\Lambda)$ follows by definition of the weight lattices. Then $\t{Proj}_{V_J}(\Lambda)\supset\t{Proj}_{V_J}(\Gamma)\supset\t{Proj}_{V_J}(\Gamma_J)=\Gamma_J$. 
\end{proof}

In particular, since $\Lambda_J$ and $\Gamma_J$ are both lattices of rank $|J|$, so is $\t{Proj}_{V_J}(\Lambda)$. This has the following as an immediate consequence. 

\begin{lem}\label{decomweight}
We have the $\mathbb{Z}$-linear direct sum 
$$\Lambda={}^J\Lambda\bigoplus {}^J\Lambda^\perp,$$ 
such that $\t{Proj}_{V_J}({}^J\Lambda^\perp)=0$
while 
$\t{Proj}_{V_J}:{}^J\Lambda\xrightarrow{\sim} \t{Proj}_{V_J}(\Lambda)$ is an isomorphism of rank-$|J|$ lattices. 
\end{lem}

We now wish to evaluate the $L^p$ norm of the characters $\chi_\mu$ as in the formula \eqref{char}. With Lemma \ref{deltaIdeltaJ}, Proposition \ref{keyprop}, and Lemma \ref{projtosub} in mind, 
we now analyze characters of the form $\chi^J_\lambda(H_J)$ for $\lambda\in\Lambda_J$ and $H\in P_J$. Let $H\in P_J$. By the definition of $H_J$ and $P_J$ we know that 
$$0<\alpha_j(H_J)/2\pi i+\delta_{0j}\leq N^{-1}\t{ for }j\in J.$$
If $0\notin J$, then the above already implies that $|H_J|\lesssim N^{-1}$. In general, there is a unique $H^0_J\in t_J$ such that $\alpha_j(H^0_J)/2\pi i+\delta_{0j}=0$. Then 
$$H_J=H_J^1+H^0_J$$
such that $|\alpha_j(H_J^1)|\leq N^{-1}$ for all $j\in J$, and thus 
$$|H_J^1|\lesssim N^{-1}.$$

\begin{lem}\label{shift}
For $H\in P_J$ and $\lambda\in\Lambda_J$, $\chi^J_{\lambda}(H_J)=e^{(\lambda+\rho_J)(H^0_J)}\chi^J_{\lambda}(H_J^{1}).$ 
\end{lem}
\begin{proof}
We have $\chi_\lambda^J=(\delta^J)^{-1}\sum_{s\in W_J}\det s\ e^{s\lambda}$. By a standard fact in root system theory, $s\lambda-\lambda\in\Gamma_J$ for each $s\in W_J$ and $\lambda\in\Lambda_J$. By definition $\alpha(H^0_J)/2\pi i\in\Z$ for all $\alpha\in\Gamma_J$, we thus have $e^{s\lambda(H^0_J)}=e^{\lambda(H^0_J)}$, which contributes the factor $e^{\lambda(H^0_J)}$ to the right side of the desired equation. Now we can write 
$$\delta^J(H_J)=e^{-\rho_J(H_J)}\prod_{\alpha\in\Sigma^+_J}\left(e^{\alpha(H_J)}-1\right)
=e^{-\rho_J(H^0_J)}e^{-\rho_J(H_J^1)}\prod_{\alpha\in\Sigma^+_J}\left(e^{\alpha(H_J^1)}-1\right)
=e^{-\rho_J(H^0_J)}\delta^J(H_J^1).$$
This contributes the other factor $e^{\rho_J(H^0_J)}$ and thus concludes the proof. 
\end{proof}

Lastly, to analyze $\chi^J_{\lambda}(H_J^{1})$, we apply:

\begin{lem}\cite[Harish-Chandra's integral formula]{HC57}
Let $U$ be a compact semisimple Lie group, and let $\mathfrak{u}$ be its Lie algebra and let $\mathfrak{t}$ be its Cartan subalgebra. Let $\mathfrak{u}^*_\C,\mathfrak{t}^*_\C$ be the spaces of complex linear forms on $\mathfrak{u},\mathfrak{t}$ respectively. For $u\in U$, let $\t{Ad}_u$ denotes the adjoint action of $u$ on $\mathfrak{u}$ as well as on $\mathfrak{u}^*_\C$ such that $(\t{Ad}_u\lambda)(X)=\lambda(\t{Ad}^{-1}_u X)$ for any $\lambda\in \mathfrak{u}^*_\C$ and $X\in\mathfrak{u}$. Let $(\cdot,\cdot)$ denote the Killing form on $\mathfrak{u}^*_\C$. 
Then for any $\lambda,\mu\in\mathfrak{t}^*_\C$, we have 
$$\sum_{s\in W}\det s\ e^{(s\lambda,\mu)}=\frac{\prod_{\alpha\in\Sigma^+}(\alpha,\lambda)\cdot\prod_{\alpha\in\Sigma^+}(\alpha,\mu)}{\prod_{\alpha\in\Sigma^+}(\alpha,\rho)}\int_U e^{(\t{Ad}_u\lambda,\mu)}\ du.$$
\end{lem}

As a consequence, we have:
\begin{lem}\label{formulaforcharacterparabolic}
Let $U_J$ be a compact semisimple Lie group of Lie algebra $\mathfrak{u}_J$ whose root system is $\Sigma_J$ (which always exists thanks to Lie's third theorem). Then 
\begin{align*}
\chi^J_{\lambda}(H_J^{1})=\frac{\prod_{\alpha\in\Sigma^+_J}\alpha(H_J^{1})}{\delta^J(H_J^1)}\cdot \frac{\prod_{\alpha\in\Sigma^+_J}(\alpha,\lambda)}{\prod_{\alpha\in\Sigma^+_J}(\alpha,\rho_J)}
\cdot 
\int_{U_J}e^{\lambda(\t{Proj}_{\mathfrak{t}_J}(\t{Ad}_u H_J^{1}))}\ du. 
\end{align*}
Here $\t{Proj}_{\mathfrak{t}_J}:\mathfrak{u}_J\to\mathfrak{t}_J$ is the orthogonal projection onto $\mathfrak{t}_J$ induced from the Killing form. 
\end{lem}

In particular, we have the following character bound. 

\begin{lem}\label{charbound}
For $|\mu|\lesssim N$, $\mu\in\Lambda$, $J\subsetneqq\{0,\ldots,r\}$, we have 
$$|\chi^J_{\mu_J}(H_J)|\lesssim N^{|\Sigma^+_J|},\ \t{for }H\in P_J.$$
\end{lem}
\begin{proof}
Since $|\mu|\lesssim N$, we have $|\mu_J|\lesssim N$. By Lemma \ref{projtosub}, $\mu_J\in\Lambda_J$. Using the isometry property
$|\t{Ad}_u H_J^1|=|H_J^1|$ of the adjoint map, and the fact that 
$|H_J^1|\lesssim N^{-1}$, we have 
$$\left|\frac{\prod_{\alpha\in\Sigma^+_J}\alpha(H_J^{1})}{\delta^J(H_J^1)}\right|\lesssim 1, \ \left|\int_{U_J}e^{\mu_J(\t{Proj}_{\mathfrak{t}_J}(\t{Ad}_u H_J^{1}))}\ du\right|\lesssim 1.$$
Then the result follows from Lemma \ref{shift} and \ref{formulaforcharacterparabolic}. 
\end{proof}

\begin{rem}
For any regular element $\mu_J$ in $\Lambda_J$, 
$|\chi^J_{\mu_J}|$ is bounded by the dimension $|d_{\mu_J}|$, which yields the above estimate by an application of the Weyl dimension formula. The essence of the above lemma is to treat non-regular $\mu_J$, which is usually a subtle issue in higher-rank analysis. 
\end{rem}

\section{Proof of Theorem \ref{JointBound}}\label{sectionjoint}


We are ready to prove Theorem \ref{JointBound}. 

\begin{proof}
Let $\Delta$ denote the Laplace-Beltrami operator. Then for $\mu\in\Lambda^+$, 
$$\Delta \chi_\mu=-(|\mu|^2-|\rho|^2)\chi_\mu.$$
Choose $N\asymp|\mu|$. By Weyl's integration formula \eqref{Weylint}, we write $\|\chi_\mu\|_{L^p(U)}=\|\chi_\mu|\delta|^{\frac{2}{p}}\|_{L^p(A)}$. Using Lemma \ref{bs} the barycentric-semiclassical subdivision, we have 
$$\|\chi_\mu\|_{L^p(U)}\lesssim\sum_{J,I\subset\{0,\ldots,r\}, \ |I|=r, \ J\subset I}\left\|\chi_\mu|\delta|^\frac{2}{p}\right\|_{L^p(P_{I,J})}.$$ 
Using \eqref{char}, we have
$$\left|\chi_\mu(H)|\delta(H)|^{\frac{2}{p}}\right|
\leq\frac{|\delta^J|^{\frac{2}{p}}}{|W_J| |\delta_I(H)|^{1-\frac{2}{p}}|\delta_{I,J}(H)|^{1-\frac{2}{p}}} \sum_{s\in W}\left|\chi^J_{(s\mu)_J}(H_J)\right|.$$
Part (i) is then a consequence of Lemma \ref{deltaIdeltaJ}, Lemma \ref{charbound}, and the key Proposition \ref{keyprop}. Now by the argument of $TT^*$, the estimate in part (ii) is equivalent to 
$$\|\psi\|_{L^p(U)}\lesssim N^{d-r-\frac{2d}{p}} \|\psi\|_{L^{p'}(U)}.$$
For a joint eigenfunction $\psi$ of spectral parameter $\mu$, we have 
$$\psi=\psi*(d_\mu\chi_\mu).$$
Then the above estimate is a consequence of Young's convolution inequality, part (i), and the dimension bound $d_\mu\lesssim N^{\frac{d-r}{2}}$ as from the Weyl dimension formula. 
\end{proof}

\begin{rem}
The exponent $\frac{d-r}{2}-\frac{d}{p}$ of $N$ in part (i) is sharp, as can be seen by testing the character on a $N^{-1}$-sized neighborhood of the origin and choosing the spectral parameter $\mu$ away from the walls of the Weyl chamber such that $(\mu,\alpha)\gtrsim |\mu|$ for all $\alpha\in\Sigma^+$. Then it holds that $|\chi_\mu|\gtrsim N^{\frac{d-r}{2}}$ on this neighborhood thanks to an application of Lemma \ref{formulaforcharacterparabolic} (letting $J=\{1,\ldots,r\}$). As an $N^{-1}$-neighborhood is of volume $\asymp N^{-d}$, this shows sharpness of the exponent. 
\end{rem}


\section{The Schr\"odinger kernel}\label{theSchrodingerkernel}
To treat Strichartz estimates for the Schr\"odinger equation, we now study the Schr\"odinger propagator $e^{it\Delta}$. As is standard from Littlewood-Paley theory on compact manifolds as developed in \cite{BGT04}, for the purpose of Strichartz estimates it suffices to consider the spectrally localized or say mollified version of the Schr\"odinger propagator. Fix a large positive number $N$. Let $\phi(-N^{-2}\Delta)$ be a spectral projector for the Laplace-Beltrami operator $\Delta$ on $U$ associated to the standard metric as induced from the Killing form. We define the mollified Schr\"odinger kernel 
$\mathscr{K}_N(t,x)$ ($t\in\R$, $x\in U$) as follows
$$f*\mathscr{K}_N(t,\cdot):=\phi(-N^{-2}\Delta)e^{it\Delta}f.$$
Then expressing it as a Fourier series, we have that  
\begin{align}\label{kernelformulaeq}
\mathscr{K}_N(t,\exp H)=\sum_{\mu\in\Lambda^+}\phi\left(\frac{|\mu|^2-|\rho|^2}{N^2}\right)e^{-it(|\mu|^2-|\rho|^2)}d_\mu\chi_\mu(H), \ \t{for }H\in\mathfrak{t},
\end{align}
where 
\begin{align}\label{dimensionformula}
d_\mu=\frac{\prod_{\alpha\in\Sigma^+}(\mu,\alpha)}{\prod_{\alpha\in\Sigma^+}(\rho,\alpha)}, 
\end{align}
is the Weyl formula for the dimension of representation and 
$\chi_\mu(H)$ is the character associated to $\mu$. As is the case for the Weyl character formula, the above dimension formula makes sense for any $\mu\in i\mathfrak{t}^*$ and in particular for any $\mu\in\Lambda$.

We begin to derive important formulas for the Schr\"odinger kernel similar to those for the characters.
We first recall a lemma from \cite{Zha18} that expresses $\mathscr{K}_N$ as a sum over the whole weight lattice $\Lambda$. 

\begin{lem}\label{weighttowhole}
\begin{align*}
\mathscr{K}_N(t,\exp H)=\frac{1}{|W|}\sum_{\mu\in\Lambda}\phi\left(\frac{|\mu|^2-|\rho|^2}{N^2}\right)e^{-it(|\mu|^2-|\rho|^2)}d_\mu\chi_\mu(H).
\end{align*}
\end{lem}
\begin{proof}
This is a direct consequence of Lemma \ref{LambdaLambda+}, the fact that $|s\mu|=|\mu|$ for any $s\in W$ and $\mu\in i\mathfrak{t}^*$, and the formula \eqref{dimensionformula} of $d_\mu$. 
\end{proof}

We also recall a standard lemma of root system theory. 
\begin{lem}\label{dmudsmu}
For $\mu\in i\mathfrak{t}^*$ and $s\in W$, $d_\mu=\det s\ d_{s\mu}$. 
\end{lem} 

Now we have:

\begin{lem}\label{lemdecompositionkernel}
For any $H\in\mathfrak{t}$ and $t\in\R$, we have 
\begin{align}\label{kernelformulalem}
\mathscr{K}_N(t,\exp H)=
\frac{1}{|W_J| \delta_I(H)\delta_{I,J}(H)}\cdot \mathscr{K}_{N}^J(t,H)
\end{align}
where 
\begin{align*}
\mathscr{K}_{N}^J(t,H)=\sum_{\mu\in \Lambda}e^{\mu(H_J^\perp)-it(|\mu|^2-|\rho|^2)}\phi\left(\frac{|\mu|^2-|\rho|^2}{N^2}\right)d_\mu  \chi^J_{\mu_J}(H_J).
\end{align*}
\end{lem}

\begin{proof}
Using Lemma \ref{weighttowhole} and \eqref{char}, we have 
\begin{align*}
\mathscr{K}_N=\frac{1}{|W_J| \delta_I(H)\delta_{I,J}(H)}\cdot\frac{1}{|W|}\sum_{s\in W}\sum_{\mu\in\Lambda}&\phi\left(\frac{|\mu|^2-|\rho|^2}{N^2}\right)e^{-it(|\mu|^2-|\rho|^2)}d_\mu\\
& \cdot \det s\ e^{(s\mu)(H^\perp_J)}\chi^J_{(s\mu)_J}(H_J). 
\end{align*}
Note that $$s\Lambda=\Lambda$$ for any $s\in W$, then \eqref{kernelformulalem} holds by an application of Lemma \ref{dmudsmu} and the fact that $|s\mu|=|\mu|$ for any $s\in W$ and $\mu\in i\mathfrak{t}^*$. 
\end{proof}

Now we wish to incorporate information in Section \ref{projectionweightlattice} to refine the above formula.  
By Lemma \ref{projtosub} and \ref{decomweight}, let ${}^J\Gamma$ be the preimage of $\Gamma_J$ under the isomorphism 
$\t{Proj}_{V_J}:{}^J\Lambda\xrightarrow{\sim}\t{Proj}_{V_J}(\Lambda)$. Let $ {}^J\Lambda/{}^J\Gamma$ be the quotient group, and for $\mu\in {}^J\Lambda$, let $[\mu]=\mu+{}^J\Gamma$ be the corresponding coset.
Then we can write 
$$\Lambda=\left(\bigsqcup_{[\mu]\in{}^J\Lambda/{}^J\Gamma}\left(\mu+{}^J\Gamma\right)\right)\bigoplus{}^J\Lambda^\perp,$$
with 
\begin{align}\label{cosetfiniteness}
|{}^J\Lambda/{}^J\Gamma|\lesssim 1.
\end{align}

From now on, for each of the finitely many cosets in $^J\Lambda/^J\Gamma$, we fix a $\mu\in {}^J\Lambda$ that represents $[\mu]$. 

\begin{lem} \label{kappaNJ}
Let 
$$b(\mu,H):=e^{\mu(H_J^\perp)+(\mu_J+\rho_J)(H^0_J)} \cdot \frac{\prod_{\alpha\in\Sigma^+_J}\alpha(H_J^{1})}{\delta^J(H_J^1)},$$
$$P(\mu,\lambda_1,\lambda_2,H)
:=\phi\left(\frac{|\mu+\lambda_1+\lambda_2|^2-|\rho|^2}{N^2}\right)
d_{\mu+\lambda_1+\lambda_2} \frac{\prod_{\alpha\in\Sigma^+_J}(\alpha,\mu_J+(\lambda_1)_J)}{\prod_{\alpha\in\Sigma^+_J}(\alpha,\rho_J)}
\cdot 
\int_{U_J}e^{(\mu_J+(\lambda_1)_J)(\t{Proj}_{\mathfrak{t}_J}(\t{Ad}_u H_J^{1}))}\ du,$$
$$\kappa_N^J(\mu,t,H):=\sum_{\lambda_1\in {}^J\Gamma, \ \lambda_2\in{}^J\Lambda^\perp}
e^{(\lambda_1+\lambda_2)(H_J^\perp)-it(|\lambda_1+\lambda_2|^2+2(\mu,\lambda_1+\lambda_2))}P(\mu,\lambda_1,\lambda_2,H).$$
Then 
\begin{align}\label{KNkappaN}
\mathscr{K}_N^J(t,H)
=\sum_{[\mu]\in{}^J\Lambda/{}^J\Gamma} e^{-it(|\mu|^2-|\rho|^2)} b(\mu,H)\cdot  \kappa_N^J(\mu,t,H). 
\end{align}
\end{lem}
\begin{proof}
By the above decomposition of $\Lambda$, for $\lambda\in\Lambda$, write uniquely 
$$\lambda=\mu+\lambda_1+\lambda_2$$
where $\lambda_1\in {}^J\Gamma$, $\lambda_2\in{}^J\Lambda^\perp$. 
As $\text{Proj}_{V_J}(^{J}\Lambda^\perp)=0$, $(\lambda_2)_J=0$ for all $\lambda_2\in{}^J\Lambda^\perp$. Also observe that as $\alpha(H_J^0)/2\pi i\in\mathbb{Z}$ for all $\alpha\in\Gamma_J$, $(\lambda_1)_J(H^0_J)=0$ for all $\lambda_1\in {}^J\Gamma$. Combine these observations with Lemma \ref{shift} and Lemma \ref{formulaforcharacterparabolic}, the formula is proved. 
\end{proof}

The above $\kappa_N^J(\mu,t,H)$ is in the form of a Weyl type exponential sum. We will treat it in later sections in two different ways, one by Weyl differencing, one by Poisson summation, oscillatory integrals and Kloosterman and Sali\'e sums. 

Some preliminary estimates are in order. We pick a $\mathbb{Z}$-basis $\{u_1,\ldots,u_{|J|}\}$ of ${}^J\Gamma$, and a 
$\mathbb{Z}$-basis $\{u_{|J|+1},\ldots, u_r\}$ of 
${}^J\Lambda^\perp$ so that $\{(u_{1})_J,\ldots, (u_{|J|})_J\}$ is a basis of $\Gamma_J$. For $\lambda_1\in {}^J\Gamma$, $\lambda_2\in{}^J\Lambda^\perp$, we write 
\begin{align}\label{lambda1=n1}
\lambda_1=n_1u_1+\cdots+n_{|J|}u_{|J|},
\end{align}
$$(\lambda_1)_{J}=n_1(u_{1})_J+\cdots+n_{|J|}(u_{|J|})_J,$$
\begin{align}\label{lambda2=n2}
\lambda_2=n_{|J|+1}u_{|J|+1}+\cdots+n_ru_r,
\end{align}
for $n_1,\ldots,n_r\in \mathbb{Z}$. 
Note that $P(\mu,\lambda_1,\lambda_2,H)$ makes sense if $n_1,\ldots,n_r$ take values in $\R$. We write 
$$P(\mu,n_1,\ldots,n_r,H):=P(\mu,\lambda_1,\lambda_2,H).$$
Let $\partial_{n_j}$ denote differential operator with respect to the variable $n_j\in\R$, and let $D_{n_j}$ denote the forward difference operator with respect to the variable $n_j\in\Z$. 

\begin{lem}
Let $H\in P_J$. We have 
\begin{align}\label{atmuH}
|b(\mu,H)|\lesssim 1. 
\end{align}
Let ${\bf m}=(m_1,\ldots,m_r)\in(\mathbb{Z}_{\geq 0})^r$ and let $|{\bf m}|=\sum_{j}m_j$. Then 
\begin{align}\label{P()}
\left|\left(\prod_{j=1}^r\partial^{m_j}_{n_j}\right)P(\mu,n_1,\ldots,n_r,H)\right|\lesssim_{{\bf m}} N^{|\Sigma^+|+|\Sigma^+_J|-|{\bf m}|}\t{ for all }n_j\in\mathbb{R},\ j=1,\ldots,r,
\end{align} 
and 
\begin{align}\label{DP()}
\left|\left(\prod_{j=1}^rD^{m_j}_{n_j}\right)P(\mu,n_1,\ldots,n_r,H)\right|\lesssim_{{\bf m}} N^{|\Sigma^+|+|\Sigma^+_J|-|{\bf m}|}\t{ for all }n_j\in\mathbb{Z},\ j=1,\ldots,r.
\end{align} 
\end{lem}
\begin{proof}
The first inequality follows directly from the fact that $|H_J^1|\lesssim N^{-1}$. The remaining estimates are standard, observing that the cutoff function $\phi$ results in the restraint $|n_j|\lesssim N$ ($j=1,\ldots,r$), and that $|\t{Proj}_{\mathfrak{t}_J}(\t{Ad}_uH_J^1)|\lesssim |\t{Ad}_uH_J^1|=|H_J^1|\lesssim N^{-1}$.  
\end{proof}

\section{Proof of Proposition \ref{kernelmajorarc}}\label{SectionLp}

By rationality of the weight lattice under the Killing form $(\cdot,\cdot)$, there exsits $\mathcal{T}\in 2\pi\mathbb{Q}$ such that 
\begin{align}\label{rationalityKillingform}
(\lambda,\mu)\in \frac{2\pi}{\mathcal{T}}\Z,\text{ for all }\lambda,\mu\in\Lambda.
\end{align}
In particular, $\mathcal{T}$ is a period of the Schr\"odinger kernel as well as the function $\kappa_N^J(\mu,t,H)$. We have the following treatment of $\kappa_N^J(\mu,t,H)$. 

\begin{lem}
For $H\in P_J$, it holds 
\begin{align*}
|\kappa_N^J(\mu,t,H)|\lesssim \frac{N^{|\Sigma^+|+|\Sigma^+_J|+r}}{\left(\sqrt{q}\left(1+N\left\|\frac{t}{\mathcal{T}}-\frac{a}{q}\right\|^{\frac{1}{2}}\right)\right)^r}
\end{align*}
for $\left\|\frac{t}{\mathcal{T}}-\frac{a}{q}\right\|\lesssim \frac{1}{qN}$. Here $\|\cdot\|$ denotes the distance from the nearest integer. 
\end{lem}
\begin{proof}
This is a multi-dimensional Weyl type inequality. Using \eqref{DP()}, the estimate follows from an application of Weyl differencing. For details of proof, we refer to for example \cite[Lemma 3.18]{Bou93} for one-dimensional Weyl inequality and \cite[Lemma 7.4]{Zha18} for multi-dimensional Weyl type inequality; the key condition used for the latter is the non-degeneracy of the Killing form $|\cdot|^2$.  
\end{proof}

Using \eqref{KNkappaN}, \eqref{cosetfiniteness}, and \eqref{atmuH}, we immediately have:

\begin{lem} \label{estimatemajorarc}
For $H\in P_J$, it holds 
\begin{align*}
|\mathscr{K}_{N}^J(t,H)|\lesssim \frac{N^{|\Sigma^+|+|\Sigma^+_J|+r}}{\left(\sqrt{q}\left(1+N\left\|\frac{t}{\mathcal{T}}-\frac{a}{q}\right\|^{\frac{1}{2}}\right)\right)^r}
\end{align*}
for $\left\|\frac{t}{\mathcal{T}}-\frac{a}{q}\right\|\lesssim \frac{1}{qN}$. 
\end{lem}

We are ready to prove Proposition \ref{kernelmajorarc}. 
 
\begin{proof}
By Weyl's integration formula as in \eqref{Weylint}, we have 
$$\|\mathscr{K}_N(t,\cdot)\|_{L^p(U)}=\|\mathscr{K}_N(t,\cdot)|\delta|^{\frac{2}{p}}\|_{L^p(A)}.$$
Since $A=\bigcup_{J\subset I, |I|=r} P_{I,J}$, it suffices to prove that $\|\mathscr{K}_N(t,\cdot)|\delta|^{\frac{2}{p}}\|_{L^p(P_{I,J})}$ has the desired bound for all $I,J$. Using \eqref{kernelformulalem}, we have 
$$|\mathscr{K}_N(t,H)|\cdot |\delta(H)|^{\frac{2}{p}}=\frac{|\delta^J(H)|^{\frac{2}{p}}}{|W_J|\cdot |\delta_I(H)|^{1-\frac{2}{p}}|\delta_{I,J}(H)|^{1-\frac{2}{p}}}\cdot|\mathscr{K}^J_N(t,H)|.$$
Then we have the desired estimate, combining \eqref{deltaI}, \eqref{deltaJ}, Proposition \ref{keyprop}, and Lemma \ref{estimatemajorarc}. 
\end{proof}

\begin{rem}
In light of Remark \ref{productweight}, it is clear from the proof that Proposition \ref{kernelmajorarc} may be generalized to any product of compact simple Lie groups. Let $M$ be such a product and for each irreducible factor $M_0$ of $M$, let
$$s_0:=\frac{2d_0}{d_0-r_0}$$
where $d_0,r_0$ are respectively the dimension and rank of $M_0$. Let $s$ be the largest among these $s_0$'s. Then for any $p>s$, inequality \eqref{Lp} holds for $\left\|\frac{t}{\mathcal{T}}-\frac{a}{q}\right\|\lesssim \frac{1}{qN}$. 
\end{rem}


\section{Farey dissection and Kloosterman's method} \label{fareydissection}
We review Farey dissection as in 
\cite{HW08} and preliminaries of Kloosterman's version of the Hardy-Littlewood circle method \cite{Klo27}, as expounded by Estermann in \cite{Est62}. 
Let $n$ be an integer $\geq 2$ and consider the sequence of Farey points
$$\left\{\frac{a}{q},\ a\in\mathbb{Z}_{\geq 1},\ q\in\mathbb{Z}_{\geq 1}, \ (a,q)=1, \ a\leq q\leq n\right\}$$ of order $n$ on the unit circle. For each two consecutive points 
$\frac{a}{q},\frac{a'}{q'}$ in the sequence, consider their mediant
$$m_{\frac{a'}{q'},\frac{a}{q}}=m_{\frac{a}{q},\frac{a'}{q'}}=
\left\{
\begin{array}{ll}
\frac{a+a'}{q+q'}, &\text{ if }\frac{a}{q}<\frac{a'}{q'},\\
& \\
\frac{0+1}{1+n},& \text{ if }\frac{a}{q}=\frac{1}{1}\text{ and }\frac{a'}{q'}=\frac{1}{n}.
\end{array}
\right.
$$
The unit circle is now divided into Farey arcs, each bounded by two mediants and containing one Farey point. Equivalently, the interval 
$\left(\frac{1}{n+1},1+\frac{1}{n+1}\right]$ is the disjoint union of the Farey arcs 
$$\mathcal{M}_{a,q}=\left(
\frac{a}{q}-\frac{1}{qq_0},\frac{a}{q}+\frac{1}{qq_1}\right]$$
with $a\in\mathbb{Z}_{\geq 1}$, $q\in\mathbb{Z}_{\geq 1}$,  $(a,q)=1$,  $a\leq q\leq n$, and 
\begin{align}\label{n<qj}
n<q_j\leq n+q \ (j=0,1), \ aq_0\equiv 1\ (\text{mod }q), \ aq_1\equiv -1\ (\text{mod }q).
\end{align}
This implies that the Farey dissection of order $n$ has the following uniformity property 
\begin{align}\label{uniformity}
\frac{1}{q(2n-1)}\leq \frac{1}{qq_j}\leq \frac{1}{q(n+1)}, \ j=0,1.
\end{align}

As Proposition \ref{kernelmajorarc} already tells that the Schr\"odinger kernel is sensitive to how close the time variable is to the Farey points, 
we further dyadically dissect each Farey arc as follows. Such methods have been explored by Bourgain \cite{Bou93e,Bou93}. Let $Q$ be dyadic integers, i.e. powers of 2, such that $1\leq Q\leq n$. For $Q\leq q<2Q$, let $M$ denote dyadic integers such that $Q\leq M\leq n$. 
Let $g(a,q,\gamma)$ be the indicator function of $\mathcal{M}_{a,q}-\frac{a}{q}$, i.e.,
$$g(a,q,\gamma)=\left\{\begin{array}{ll}
1, & \text{ if }\frac{a}{q}+\gamma\in \mathcal{M}_{a,q},\\
0, & \text{ otherwise}.
\end{array}\right.$$
For $Q\leq q<2Q$, define 
$$g_M(a,q,\gamma)=
\left\{\begin{array}{ll}
g\left(a,q,\frac{M}{Q}\gamma\right)-g\left(a,q,\frac{2M}{Q}\gamma\right), & \text{ if }Q\leq M\leq  \frac{n}{2},\\
& \\
g\left(a,q,\frac{M}{Q}\gamma\right), & \text{ if }\frac{n}{2}<M\leq n.
\end{array}\right.
$$
so that 
$$g(a,q,\gamma)=\sum_{M}g_M(a,q,\gamma).$$
Note that because $M$ is a dyadic integer, there is a unique $M$ such that $\frac{n}{2}<M\leq n$. By \eqref{uniformity}, we have the following information on the support of $g_M(a,q,\gamma)$ in $\gamma$
\begin{align}\label{supportofgM}
\text{supp } g_M(a,q,\cdot)\subset 
\left\{
\begin{array}{ll}
\left[-\frac{1}{(n+1)M},-\frac{1}{4(2n-1)M}\right]\cup\left[\frac{1}{4(2n-1)M},\frac{1}{(n+1)M}\right], & \text{ if }Q\leq M\leq \frac{n}{2},\\
& \\
\left[-\frac{1}{(n+1)M},\frac{1}{(n+1)M}\right], & \text{ if }\frac{n}{2}<M\leq n.
\end{array}
\right.
\end{align}
Define 
\begin{align}\label{1QMt=}
\mathbbm{1}_{Q,M}(t)
=\sum_{1\leq a\leq q,\ (a,q)=1,\ Q\leq q<2Q} g_M\left(a,q,t-\frac{a}{q}\right),
\end{align}
so that we have the partition of unity 
$$1=\sum_{Q,M}\mathbbm{1}_{Q,M}$$
on the unit circle. 
Let $\widehat{\mathbbm{1}_{Q,M}}$ denote the Fourier transform of $\mathbbm{1}_{Q,M}$ on the unit circle such that 
$$\mathbbm{1}_{Q,M}(t)=\sum_{m\in \Z}
\widehat{\mathbbm{1}_{Q,M}}(m)e(mt),$$
then using \eqref{supportofgM} we have
\begin{align}\label{1QM}
\|\widehat{\mathbbm{1}_{Q,M}}\|_{l^\infty}\lesssim \|\mathbbm{1}_{Q,M} \|_{L^1}\lesssim\frac{Q^2}{(n+1)M}. 
\end{align}

For $(a,q)=1$, let $a^*$ denote the natural number $\leq q$ such that $aa^*\equiv 1$ (mod $q$). From \eqref{n<qj} we see that for fixed $q$, the dependence of $g(a,q,\gamma)$ on $a$ is via its dependence on $a^*$, and precisely we have 
$$g(a,q,\gamma)=f(a^*,q,\gamma)$$
where 
$$
f(b,q,\gamma)=\left\{
\begin{array}{ll}
1, & \text{ if }\gamma\in \left(-\frac{1}{qq_0}, \frac{1}{qq_1}\right],\\
0, & \text{ otherwise}.
\end{array}
\right.
$$
where 
$$n<q_j\leq n+q \ (j=0,1), \ q_0\equiv b\ (\text{mod }q), \ q_1\equiv -b \ (\text{mod } q)$$
for $b=1,2,\ldots,q$. Now Kloosterman's idea on the circle method comes in, and one of its realizations is to expand $f(b,q,\gamma)$
in a finite sum
$$f(b,q,\gamma)=\sum_{h=1}^q c(h,q,\gamma) e\left(\frac{hb}{q}\right).$$
This is nothing but to express $f(\cdot,q,\gamma)$ as the Fourier transform of $c(\cdot,q,\gamma)$ over the additive group $\mathbb{Z}/q\mathbb{Z}$. The inverse Fourier transform then gives 
\begin{align}\label{chqgamma}
c(h,q,\gamma)=\frac{1}{q}\sum_{b=1}^q f(b,q,\gamma) e\left(-\frac{hb}{q}\right).
\end{align}
We record in the following lemma several properties of $c(h,q,\gamma)$. 

\begin{lem}\label{Lemma13}

(i) For $|\gamma|>\frac{1}{q(n+1)}$, $c(h,q,\gamma)=0$ for all $h=1,\ldots,q$. And for $|\gamma|<\frac{1}{q(n+q)}$,  
$c(h,q,\gamma)=0$ for all $h=1,\ldots,q-1$, and $c(q,q,\gamma)=1$. 

(ii) For all $\gamma$, 
$$\sum_{h=1}^q|c(h,q,\gamma)|\leq 2+\log q.$$ 

(iii) 
$$g(a,q,\gamma)=\sum_{h=1}^qc(h,q,\gamma) e\left(\frac{ha^*}{q}\right).$$

\end{lem} 

\begin{proof} 
This is essentially Lemma 13 of \cite{Est62}. (iii) is clear from the previous discussion. Using \eqref{chqgamma} and the definition of $f(b,q,\gamma)$, we have 
$$
c(h,q,\gamma)=
\left\{
\begin{array}{ll}
\frac{1}{q}\sum_{q_0=n+1,\ldots,n+q,\ q_0<-\frac{1}{q\gamma}} e\left(-\frac{hq_0}{q}\right), & \text{ if }\gamma\leq 0,\\
& \\ 
\frac{1}{q}\sum_{q_1=n+1,\ldots,n+q, \ q_1\leq\frac{1}{q\gamma}} e\left(\frac{hq_1}{q}\right), & \text{ if }\gamma\geq 0.
\end{array}
\right.
$$
(i) is clear from this formula. We estimate 
$|c(q,q,\gamma)|\leq 1$, and for $h=1,\ldots,q-1$, 
$$|c(h,q,\gamma)|\leq \frac{2}{q\left|e\left(\frac{h}{q}\right)-1\right|}
=\frac{1}{q\left|\sin \left(\frac{\pi h}{q}\right)\right|},$$
which implies (ii). 
\end{proof}

We have the following immediate corollary. 

\begin{cor}\label{coroflem13}
For $Q\leq q<2Q$, define 
\begin{align*}
c_{M}(h,q,\gamma)=
\left\{\begin{array}{ll}
c(h,q,\frac{M}{Q}\gamma)-c(h,q,\frac{2M}{Q}\gamma), & \text{ if }Q\leq M\leq\frac{n}{2},\\
& \\
c(h,q,\frac{M}{Q}\gamma), & \text{ if }\frac{n}{2}<M\leq n.
\end{array}\right.
\end{align*}

(i) For all $h=1,\ldots,q$, the support of $c_{M}(h,q,\gamma)$ in $\gamma$ satisfies 
\begin{align}\label{supportofchM}
\text{supp } c_{M}(h,q,\cdot)\subset 
\left\{
\begin{array}{ll}
\left[-\frac{1}{(n+1)M},-\frac{1}{8nM}\right]\cup\left[\frac{1}{8nM},\frac{1}{(n+1)M}\right], & \text{ if }Q\leq M\leq \frac{n}{2},\\
& \\ 
\left[-\frac{1}{(n+1)M},\frac{1}{(n+1)M}\right], & \text{ if }\frac{n}{2}<M\leq n.
\end{array}
\right.
\end{align}

(ii) For all $\gamma$, 
\begin{align}\label{sumchM}
\sum_{h=1}^q |c_{M}(h,q,\gamma)|\leq 4+2\log q.
\end{align}

(iii) 
\begin{align}\label{gMaq}
g_M(a,q,\gamma)=
\sum_{h=1}^q c_{M}(h,q,\gamma) e\left(\frac{ha^*}{q}\right).
\end{align}

\end{cor}

We remark that $c_{M}(h,q,\gamma)$ as functions of $\gamma$ are all understood to have the unit circle as the domain. 

\section{Proof of Theorem \ref{Main} and \ref{eigengroup}}\label{farey}

We now prove Theorem \ref{Main}. 

\begin{proof}[Proof of Theorem \ref{Main}] 
Reducing to a finite cover, it suffices to prove it for the case of a compact simply connected semisimple Lie group $U=U_1\times U_2\times\cdots\times U_k$, where the $U_i$'s are the simple components, equipped with the canonical Killing metrics. 
Consider the product Schr\"odinger kernel 
\begin{align}\label{productkernel}
\mathscr{K}_N=\prod_{i=1}^k\mathscr{K}_{N,i}
\end{align}
where 
$$\mathscr{K}_{N,i}(t,H_i)=\sum_{\mu_i\in\Lambda_i^+}\phi_i\left(\frac{|\mu_i|^2-|\rho_i|^2}{ N^2}\right)e^{-it(|\mu_i|^2-|\rho_i|^2)}d_{\mu_i}\chi_{\mu_i}(H_i)$$ is the kernel for the component $U_i$. By rationality of the weight lattices as in  \eqref{rationalityKillingform}, the component kernels $\mathscr{K}_{N,i}$ share a period in the time variable $t$, say $\mathcal{T}$, and we set $\T=\R/\mathcal{T}\Z$. Let $\Sigma_i$ be the root system of rank $r_i$ for $U_i$ ($1\leq i\leq k$), then Proposition  \ref{kernelmajorarc} implies 
\begin{align}\label{kernelproduct}
\|\mathscr{K}_N(t,\cdot)\|_{L^u(U)}
=\prod_{i=1}^k\|\mathscr{K}_{N,i}(t,\cdot)\|_{L^u(U_i)}
\lesssim \frac{N^{d-\frac{d}{u}}}{\left(\sqrt{q}\left(1+N\left\|\frac{t}{\mathcal{T}}-\frac{a}{q}\right\|^{\frac{1}{2}}\right)\right)^r}
\end{align}
provided 
\begin{align}\label{definitionofs}
u>s:=\max\left\{\frac{2d_i}{d_i-r_i}, \ i=1,\ldots,k\right\}.
\end{align}
Here $d_i$ is the dimension of $U_i$ ($1\leq i\leq k$).

Using Farey dissection of order $n=\lfloor N\rfloor$ as reviewed in Section \ref{fareydissection}, we write 
$$\mathscr{K}_N(t,x)=\sum_{Q,M}\mathscr{K}_{Q,M}(t,x), \ \mathscr{K}_{Q,M}(t,x):=\mathscr{K}_N(t,x)\cdot\mathbbm{1}_{Q,M}\left(\frac{t}{\mathcal{T}}\right),$$
for $(t,x)\in\T\times U$. Let $F: \T \times U\to \C$ be a continuous function. Let $*$ denote the convolution on the product group $\T\times U$. 
By Young's inequality for unimodular groups, inequality \eqref{kernelproduct}, and the estimate 
$$\left\|\mathbbm{1}_{Q,M}\left(\frac{\cdot}{\mathcal{T}}\right)\right\|_{L^u(\T)}\lesssim \left(\frac{Q^2}{NM}\right)^{\frac{1}{u}}$$
due to \eqref{supportofgM}, 
we have for $u>s$
\begin{align*}
\|F*\mathscr{K}_{Q,M}\|_{L^{2u}(\T\times U)}&\leq \|\mathscr{K}_{Q,M}\|_{L^{u}(\T\times U)}\|F\|_{L^{(2u)'}(\T\times U)}\\
&\lesssim N^{d-\frac{d+1}{u}-\frac{r}{2}}M^{\frac{r}{2}-\frac{1}{u}}Q^{-\frac{r}{2}+\frac{2}{u}}\|F\|_{L^{(2u)'}(\T\times U)}. \numberthis \label{2p'to2p}
\end{align*}
Here $2u$ and $(2u)'$ are conjugate exponents. On the other hand, as a class function on the compact Lie group $\mathbb{T}\times U$, $\mathscr{K}_{Q,M}$ has its Fourier transform $\widehat{\mathscr{K}_{Q,M}}(m,\mu)$ ($(m,\mu)\in \Z\times \Lambda^+)$ computed as follows 
$$\widehat{\mathscr{K}_{Q,M}}(m,\mu)=\phi(\mu,N)\widehat{\mathbbm{1}_{Q,M}}(m+(|\mu|^2-|\rho|^2)\mathcal{T}/2\pi),$$
where 
$$\phi(\mu,N)=\prod_{i=1}^k\phi_i((|\mu_i|^2-|\rho_i|^2)/ N^2),\ 
|\mu|^2-|\rho|^2=\sum_{i=1}^k|\mu_i|^2-|\rho_i|^2.$$ 
By \eqref{1QM}, we have for all $(m,\mu)\in \Z\times \Lambda^+$
$$|\widehat{\mathscr{K}_{Q,M}}(m,\mu)|\lesssim \frac{Q^2}{NM}.$$
As a consequence, we have 
\begin{align}\label{2to2}
\|F*\mathscr{K}_{Q,M}\|_{L^{2}(\T\times U)}
\lesssim \frac{Q^2}{NM}\|F\|_{L^2(\T\times U)}. 
\end{align}
Interpolating \eqref{2p'to2p} with \eqref{2to2} for 
$\frac{\theta}{2}+\frac{1-\theta}{2u}=\frac{1}{p}$, 
we get 
\begin{align*}
\|F*\mathscr{K}_{Q,M}\|_{L^p(\T\times U)}\lesssim N^{\left(d-\frac{d+1}{u}-\frac{r}{2}\right)(1-\theta)-\theta}M^{\left(\frac{r}{2}-\frac{1}{u}\right)(1-\theta)-\theta}Q^{\left(-\frac{r}{2}+\frac{2}{u}\right)(1-\theta)+2\theta}\|F\|_{L^{p'}(\T\times U)}.
\end{align*}
We require the exponent of $Q$ to satisfy 
\begin{align*}
\left(-\frac{r}{2}+\frac{2}{u}\right)(1-\theta)+2\theta<0 \Leftrightarrow  \theta<\frac{ru-4}{4u+ru-4},
\end{align*}
which implies that the exponent of $M$ satisfies 
$
\left(\frac{r}{2}-\frac{1}{u}\right)(1-\theta)-\theta>0
$.
Summing over the dyadic integers $M$ and $Q$, we get 
\begin{align*}
\|F*\mathscr{K}_N\|_{L^p(\T\times U)}&\lesssim \sum_{1\leq Q\leq N}\sum_{Q\leq M\leq N}\|F*\mathscr{K}_{Q,M}\|_{L^p(\T\times U)}\\
&\lesssim N^{\left(d-\frac{d+2}{u}\right)(1-\theta)-2\theta}\|F\|_{L^{p'}(\T\times U)}=N^{d-\frac{2(d+2)}{p}}\|F\|_{L^{p'}(\T\times U)},
\end{align*}
provided 
$$\frac{1}{p}=\frac{\theta}{2}+\frac{1-\theta}{2u}<\frac{ru-4}{2(4u+ru-4)}+\frac{2}{4u+ru-4} \Leftrightarrow p>2+\frac{8(u-1)}{ur}$$
for some $u>s$. 
This implies Theorem \ref{Main}, by a standard application of Littlewood-Paley theory and the $TT^*$ argument. 
\end{proof}

Then we prove Theorem \ref{eigengroup}.


\begin{proof}[Proof of Theorem \ref{eigengroup}]
We inherit the notations in the proof of Theorem \ref{Main}. 
Let $f$ be an eigenfunction of eigenvalue $-N^2$. Then $N^2=|\mu|^2-|\rho|^2$ for some $\mu\in\Lambda^+$.   
Set 
$$\mathcal{K}_N=\sum_{\mu\in\Lambda^+,|\mu|^2-|\rho|^2=N^2}d_\mu \chi_\mu.$$ 
Then it is clear that $f=f*\mathcal{K}_N$. By an argument of $TT^*$, it suffices to establish bounds of the form 
$$\|f*\mathcal{K}_N\|_{L^p(U)}\lesssim N^{\frac{d-2}{2}-\frac{d}{p}}\|f\|_{L^{p'}(U)}.$$
Let $\mathscr{K}_{N}$ be again the Schr\"odinger kernel as in \eqref{kernelformulaeq} and more generally as in \eqref{productkernel}, and here we assume that the cutoff functions $\phi_i$ satisfy $\prod_i\phi_i(y^2_i)=1$ for all $y_i$ such that $\sum_i y^2_i=1$. Then we may write 
$$\mathcal{K}_N=\frac{1}{\mathcal{T}}\int_0^\mathcal{T}\mathscr{K}_N(t,\cdot)e^{it N^2}\ dt.$$
Using the Farey dissection of order $n=\lfloor N\rfloor$ again, we decompose 
$$\mathcal{K}_N=\sum_{Q,M}\mathcal{K}_{Q,M},$$
where 
\begin{align}\label{KQM=}
\mathcal{K}_{Q,M}=\int_{0}^1\mathscr{K}_{N}(t,\cdot) e^{it N^2}\mathbbm{1}_{Q,M}\left(\frac{t}{\mathcal{T}}\right)\ d\left(\frac{t}{\mathcal{T}}\right).
\end{align}
By Proposition \ref{kernelmajorarc}, Minkowski's integral inequality, and \eqref{1QM}, we have for $u>s$ 
\begin{align}\label{KQMestimate}
\|\mathcal{K}_{Q,M}\|_{L^u(U)}\lesssim N^{d-\frac{d}{u}-\frac{r}{2}-1}M^{\frac{r}{2}-1}Q^{-\frac{r}{2}+2},
\end{align}
which implies by Young's inequality 
\begin{align}\label{2u'to2u}
\|f*\mathcal{K}_{Q,M}\|_{L^{2u}(U)}\lesssim N^{d-\frac{d}{u}-\frac{r}{2}-1}M^{\frac{r}{2}-1}Q^{-\frac{r}{2}+2}\|f\|_{L^{(2u)'}(U)}. 
\end{align}
On the other hand, the Fourier transform of $\mathcal{K}_{Q,M}$ on $U$ equals 
$$\widehat{\mathcal{K}_{Q,M}}(\mu)=\phi(\mu,N)\int_{0}^1e^{it(N^2-|\mu|^2+|\rho|^2)}\mathbbm{1}_{Q,M}\left(\frac{t}{\mathcal{T}}\right)\ d\left(\frac{t}{\mathcal{T}}\right), \ \t{for all }\mu\in\Lambda^+.$$
Thus 
$$|\widehat{\mathcal{K}_{Q,M}}(\mu)|\lesssim \|\mathbbm{1}_{Q,M}\|_{L^1}\lesssim \frac{Q^2}{NM},$$
which implies 
\begin{align}\label{2to22}
\|f*\mathcal{K}_{Q,M}\|_{L^2(U)}
\lesssim 
\frac{Q^2}{NM}\|f\|_{L^2(U)}.
\end{align}
Interpolating \eqref{2u'to2u} with \eqref{2to22} for 
$\frac{\theta}{2}+\frac{1-\theta}{2u}=\frac{1}{p}$, 
we get 
\begin{align*}
\|f*\mathcal{K}_{Q,M}\|_{L^p(U)}\lesssim N^{\left(d-\frac{d}{u}-\frac{r}{2}-1\right)(1-\theta)-\theta} M^{\left(\frac{r}{2}-1\right)(1-\theta)-\theta}Q^{\left(-\frac{r}{2}+2\right)(1-\theta)+2\theta}\|f\|_{L^{p'}(U)}. 
\end{align*}
We require the exponent of $Q$ to be negative, i.e., 
$$\left(-\frac{r}{2}+2\right)(1-\theta)+2\theta<0\Leftrightarrow \theta<\frac{r-4}{r},$$
which implies that the exponent $\left(\frac{r}{2}-1\right)(1-\theta)-\theta$ of $M$ is positive. Summing over the dyadic integers $M$ and $Q$, we have
\begin{align*}
\|f*\mathcal{K}_N\|_{L^p(U)}\lesssim N^{\left(d-\frac{d}{u}-2\right)(1-\theta)-2\theta}\|f\|_{L^{p'}(U)}=N^{d-2-\frac{2d}{p}}\|f\|_{L^{p'}(U)},
\end{align*} 
provided 
$$\frac{1}{p}=\frac{\theta}{2}+\frac{1-\theta}{2u}
<\frac{r-4}{2r}+\frac{2}{ru}\Leftrightarrow p>\frac{2ur}{ur-4u+4}$$
for some $u>s$. This finishes the proof. 
\end{proof}


\section{Proof of Theorem \ref{eigengrouploss}}\label{loss}
We provide another approach to the Weyl type exponential sum $\kappa_N^J(\mu,t,H)$ in Lemma \ref{kappaNJ} in order to prove Theorem \ref{eigengrouploss}. It is based on Bourgain's work \cite{Bou93e} and involves the Poisson summation formula, and bounds on oscillatory integrals and Kloosterman and Salié sums. 

We wish to perform Poisson summation to the Weyl type sum $\kappa_N^J(\mu,t,H)$, in order to see Gauss sums. Using \eqref{lambda1=n1} and \eqref{lambda2=n2}, we write 
$$(\lambda_1+\lambda_2)(H_J^\perp)
=2\pi i (n_1x_1+\cdots+n_rx_r),$$
for $(x_1,\ldots,x_r)\in \mathbb{R}^r$ depending on $H_J^\perp$, and 
$$|\lambda_1+\lambda_2|^2=\frac{2\pi}{\mathcal{T}}\sum_{1\leq i,j\leq n}a_{ij}n_in_j$$
where ${\bf A}=(a_{ij})$ is an integral positive-definite symmetric matrix by \eqref{rationalityKillingform}, and 
$$2(\mu,\lambda_1+\lambda_2)=\frac{2\pi}{\mathcal{T}}(n_1b_1+\cdots+n_rb_r)$$
for $(b_1,\ldots,b_r)\in\mathbb{Z}^r$ depending on $[\mu]\in {}^J\Lambda/{}^J\Gamma$. 
Let $\mathcal{K}_{Q,M}$ be the Farey pieces of the Schr\"odinger kernel as defined in \eqref{KQM=}. 

\begin{lem}
For $H\in P_{I,J}$, it holds 
\small
\begin{align*}
\mathcal{K}_{Q,M}(H)
&=\frac{1}{|W_J|\delta_I\delta_{I,J}}\sum_{[\mu]}b(\mu,H)\ \cdot \\ &  \int_{0}^1 e(n_0 \gamma) \sum_{\substack{m_j\in\mathbb{Z} \\ j=1,\ldots,r}}\sum_{Q\leq q<2Q}J({\bf x},\gamma,{\bf m};q) \sum_{1\leq h\leq q} c_{M}(h,q,\gamma) 
\sum_{1\leq a\leq q,\ (a,q)=1} S({\bf A}, a, a{\bf b}+{\bf m}, q) e\left(\frac{n_0 a+ha^*}{q}\right)
\ d\gamma 
\end{align*}
\normalsize
where 
$$J({\bf x},\gamma,{\bf m};q):=\int_{\mathbb{R}^r}e\left(\sum_jy_j(x_j+m_j/q)-\gamma\left(\sum_{ij}a_{ij}y_iy_j+\sum_jy_jb_j\right)\right) P(\mu,y_1,\ldots,y_r,H)\ dy_1\cdots dy_r,$$
and 
$$S({\bf A}, a,a{\bf b}+{\bf m}, q):=\frac{1}{q^r}\sum_{\substack{k_j=0,1,\ldots,q-1 \\ j=1,\ldots,r}}e\left(-\left(\frac{a}{q}\sum_{i,j} a_{ij}k_ik_j+\frac{1}{q}\sum_{j}k_j(ab_j+m_j) \right)\right),
$$
with
$$n_0=\frac{\mathcal{T}}{2\pi}(N^2-|\mu|^2+|\rho|^2)\in\mathbb{Z}.$$
Here 
${\bf b}=(b_1,\ldots,b_r)$ which depends on $[\mu]$, ${\bf m}=(m_1,\ldots,m_r)$, and ${\bf x}=(x_1,\ldots,x_r)$ which depends on $H_J^\perp$. The summation over $m_j$ ($j=1,\ldots,r$) converges absolutely. 
\end{lem}
\begin{proof}
By 
\eqref{kernelformulalem}, \eqref{KNkappaN}, and \eqref{KQM=}, we have 
$$
\mathcal{K}_{Q,M}(H)
=\frac{1}{|W_J|\delta_I\delta_{I,J}}\sum_{[\mu]}b(\mu,H)\int_{0}^1
e\left(n_0 \frac{t}{\mathcal{T}}\right)
\mathbbm{1}_{Q,M}\left(\frac{t}{\mathcal{T}}\right) \kappa_{N}^J(\mu,t,H)\ d\left(\frac{t}{\mathcal{T}}\right).
$$
Using \eqref{1QMt=} and \eqref{gMaq}, we then have 
\small
\begin{align}
&\mathcal{K}_{Q,M}(H)\nonumber\\
&=\frac{1}{|W_J|\delta_I\delta_{I,J}}\sum_{[\mu]}b(\mu,H)\int_{0}^1e\left(n_0\frac{t}{\mathcal{T}}\right)\sum_{Q\leq q<2Q}\sum_{1\leq a\leq q,\ (a,q)=1}\sum_{1\leq h\leq q} c_{M}\left(h,q,\frac{t}{\mathcal{T}}-\frac{a}{q}\right)e\left(\frac{ha^*}{q}\right)\kappa_{N}^J(\mu,t,H)\ d\left(\frac{t}{\mathcal{T}}\right)\nonumber \\
&=\frac{1}{|W_J|\delta_I\delta_{I,J}}\sum_{[\mu]}b(\mu,H)\sum_{Q\leq q<2Q}\sum_{1\leq a\leq q,\ (a,q)=1}\sum_{1\leq h\leq q} \int_{0}^1 e\left(n_0\frac{t}{\mathcal{T}}\right)c_{M}\left(h,q,\frac{t}{\mathcal{T}}-\frac{a}{q}\right)e\left(\frac{ha^*}{q}\right)\kappa_{N}^J(\mu,t,H)\ d\left(\frac{t}{\mathcal{T}}\right)\nonumber\\
&=\frac{1}{|W_J|\delta_I\delta_{I,J}}\sum_{[\mu]}b(\mu,H)\sum_{Q\leq q<2Q}\sum_{1\leq a\leq q,\ (a,q)=1}\sum_{1\leq h\leq q} \int_{0}^1 e\left(n_0\gamma\right)c_{M}\left(h,q,\gamma\right)e\left(\frac{n_0 a+ha^*}{q}\right)\kappa_{N}^J\left(\mu,\mathcal{T}\left(\gamma+\frac{a}{q}\right),H\right)\ d\gamma.\label{KQMH=}
\end{align}
\normalsize
The last equality results from the fact that $e\left(n_0\frac{t}{\mathcal{T}}\right)$, $c_{M}\left(h,q,\frac{t}{\mathcal{T}}-\frac{a}{q}\right)$ and $\kappa_{N}^J(\mu,t,H)$ are all functions of $\frac{t}{T}$ which lives on the unit circle. 
Put $n_j=s_jq+k_j$, $k_j=0,1,\ldots,q-1$, $j=1,\ldots,r$. 
Then the term $\kappa_{N}^J\left(\mu,\mathcal{T}\left(\gamma+\frac{a}{q}\right),H\right)$ as introduced in Lemma \ref{kappaNJ} can be rewritten as 
\begin{align*}
\kappa_{N}^J\left(\mu,\mathcal{T}\left(\gamma+\frac{a}{q}\right),H\right)
=&\sum_{\substack{k_j=0,1,\ldots,q-1 \\ j=1,\ldots,r}}\bigg[
e\left(-\frac{a}{q}\left(\sum_{i,j} a_{ij}k_ik_j+\sum_{j}k_jb_j \right)\right)
\sum_{\substack{s_j\in\mathbb{Z} \\ j=1,\ldots,r}}\\
&e\left( \sum_j(s_jq+k_j)x_j-\gamma\left(\sum_{ij}a_{ij}(s_iq+k_i)(s_jq+k_j)+\sum_j (s_jq+k_j)b_j\right)\right)\\
&\cdot P(\mu,s_1q+k_1,\ldots,s_rq+k_r,H)
\bigg].
\end{align*}
Apply Poisson summation formula
$$\sum_{\substack{s_j\in\mathbb{Z} \\ j=1,\ldots,r}}g(s_1,\ldots,s_r)
=\sum_{\substack{m_j\in\mathbb{Z} \\ j=1,\ldots,r}}
\int_{\mathbb{R}^r} g(t_1,\ldots,t_r)e\left(\sum_{j=1,\ldots,r} t_jm_j\right)\  dt_1 \cdots dt_r$$
to the compactly supported smooth function on $\mathbb{R}^r$
\begin{align*}
&g(t_1,\ldots,t_r)\\
&=e\left( \sum_j(t_jq+k_j)x_j-\gamma\left(\sum_{ij}a_{ij}(t_iq+k_i)(t_jq+k_j)+\sum_j (t_jq+k_j)b_j\right)\right)\cdot P(\mu,t_1q+k_1,\ldots,t_rq+k_r,H)
\end{align*}
and a change of variables $y_j=t_jq+k_j$ ($j=1,\ldots,r$), 
we have 
\begin{align*}
\kappa_{N}^J\left(\mu,\mathcal{T}\left(\gamma+\frac{a}{q}\right),H\right)
&=\sum_{\substack{m_j\in\mathbb{Z} \\ j=1,\ldots,r}}
\left\{\frac{1}{q^r}\sum_{\substack{k_j=0,1,\ldots,q-1 \\ j=1,\ldots,r}}e\left(-\left(\frac{a}{q}\sum_{i,j} a_{ij}k_ik_j+\frac{1}{q}\sum_{j}k_j(ab_j+m_j) \right)\right)\right\}\ \cdot \\
& \int_{\mathbb{R}^r}e\left(\sum_jy_j(x_j+m_j/q)-\gamma\left(\sum_{ij}a_{ij}y_iy_j+\sum_jy_jb_j\right)\right) P(\mu,y_1,\ldots,y_r,H)\ dy_1\cdots dy_r\\
&=\sum_{\substack{m_j\in\mathbb{Z} \\ j=1,\ldots,r}} 
S({\bf A}, a, a{\bf b}+{\bf m}, q)
J({\bf x},\gamma,{\bf m};q)
\end{align*}
where the summation over $m_j$ ($j=1,\ldots,r$) converges absolutely. 
Plugging this formula in \eqref{KQMH=} finishes the proof. 
\end{proof}

Denote 
\small
\begin{align*}
\kappa_{Q,M}(\mu,H):
=\int_{0}^1 e(n_0 \gamma) \sum_{\substack{m_j\in\mathbb{Z} \\ j=1,\ldots,r}}\sum_{Q\leq q<2Q}J({\bf x},\gamma,{\bf m};q) \sum_{1\leq h\leq q} c_{M}(h,q,\gamma) 
\sum_{1\leq a\leq q,\ (a,q)=1} S({\bf A}, a, a{\bf b}+{\bf m}, q) e\left(\frac{n_0 a+ha^*}{q}\right)
\ d\gamma.
\end{align*}
\normalsize
Here $S({\bf A}, a, a{\bf b}+{\bf m}; q)$ is essentially a higher dimensional Gauss sum. We will prove the following estimate on $\kappa_{Q,M}(\mu,H)$ using Weyl differencing, and standard bounds for Kloosterman sums and Salié sums, after we use it to prove the main theorem. 

\begin{lem}\label{estimateofkappaQM}
It holds uniformly for $H\in P_{I,J}$ that 
\begin{align*}
|\kappa_{Q,M}(\mu,H)|
\lesssim_{\varepsilon} N^{|\Sigma^+|+|\Sigma^+_J|+\frac{r}{2}-1+\varepsilon}M^{\frac{r}{2}-1}Q^{-\frac{r}{2}+\frac{3}{2}}.
\end{align*}
\end{lem}

\begin{proof}[Proof of Theorem \ref{eigengrouploss}]
Using Weyl's integration formula and the barycentric-semiclassical subdivision again, we write 
$$\|\mathcal{K}_{Q,M}\|_{L^u(U)}=\|\mathcal{K}_{Q,M}|\delta|^{\frac{2}{u}}\|_{L^u(A)}\lesssim \sum_{J\subset I}\|\mathcal{K}_{Q,M}|\delta|^{\frac{2}{u}}\|_{L^u(P_{I,J})}.$$
For $H\in P_{I,J}$, by \eqref{atmuH}, we have 
$$|\mathcal{K}_{Q,M}(H)|\cdot|\delta(H)|^{\frac{2}{u}}
\lesssim \frac{|\delta^J(H)|^{\frac{2}{u}}}{|\delta_I(H)|^{1-\frac{2}{u}}|\delta_{I,J}(H)|^{1-\frac{2}{u}}}\cdot\sum_{[\mu]\in{}^J\Lambda/{}^J\Gamma}|\kappa_{Q,M}(\mu,H)|.$$
Using Lemma \ref{estimateofkappaQM}, along with 
\eqref{deltaI}, \eqref{deltaJ}, Proposition \ref{keyprop},  
Remark \ref{productweight}, and \eqref{cosetfiniteness}, we conclude for any $u>s$
\begin{align*}
\|\mathcal{K}_{Q,M}\|_{L^u(U)}\lesssim_{\varepsilon}  
N^{d\left(1-\frac{1}{u}\right)-\frac{r}{2}-1+\varepsilon }M^{\frac{r}{2}-1}Q^{-\frac{r}{2}+\frac{3}{2}}.
\end{align*}
Compared with the estimate \eqref{KQMestimate} on $\mathcal{K}_{Q,M}$, the above estimate lowers the power of $Q$ by $\frac{1}{2}$ despite adding an extra $N^\varepsilon$. By the same interpolation argument as in the proof of Theorem \ref{eigengroup}, this then yields 
$$\|f*\mathcal{K}_N\|_{L^p(U)}\lesssim_\varepsilon N^{d-2-\frac{2d}{p}+\varepsilon}\|f\|_{L^{p'}(U)}$$
for any $p>\frac{2s(r+1)}{sr-3s+4}$ and $r\geq 4$.  
\end{proof}

\begin{proof}[Proof of Lemma \ref{estimateofkappaQM}]
We have 
\begin{align*}
&|\kappa_{Q,M}(\mu,H)|\\
&\leq\int_{0}^1 \sum_{\substack{m_j\in\mathbb{Z} \\ j=1,\ldots,r}}\sum_{Q\leq q<2Q}|J({\bf x},\gamma,{\bf m};q)| \sum_{1\leq h\leq q} |c_{M}(h,q,\gamma)|
\left|\sum_{1\leq a\leq q,\ (a,q)=1} S({\bf A}, a, a{\bf b}+{\bf m}, q) e\left(\frac{n_0 a+ha^*}{q}\right)\right|
\ d\gamma
\end{align*}
where the summation over $m_j$ ($j=1,\ldots,r$) converges. 
We first bound the term $J({\bf x},\gamma,{\bf m};q)$. Using \eqref{P()} and applying the standard method of stationary phase for the oscillatory integral $J({\bf x},\gamma,{\bf m};q)$ \cite[Chapter VIII Section 2.3]{Ste93}, we get 
\begin{align}\label{Jxgamma}
|J({\bf x},\gamma,{\bf m};q)|\lesssim N^{|\Sigma^+|+|\Sigma^+_J|}\min\{N^r, |\gamma|^{-\frac{r}{2}}\}.
\end{align}
Moreover, noting that because of \eqref{supportofchM} we have $|\gamma|\lesssim\frac{1}{qN}$, and that the support of $P(\mu,y_1,\ldots,y_r,H)$ is on the region $|y_i|\lesssim N$ for all $i$, the logarithmic derivative of the phase function 
$$e\left(\sum_jy_j(x_j+m_j/q)-\gamma\left(\sum_{ij}a_{ij}y_iy_j+\sum_jy_jb_j\right)\right)$$
with respect to $y_j$ is bounded below in absolute value by 
a constant times $|x_j+m_j/q|$, as long as $|x_j+m_j/q|\gtrsim 1/q$. So for each $\varepsilon>0$, if any of the $m_j$ satisfies $|x_jq+m_j|\geq N^{\varepsilon}$, using the derivative estimate \eqref{P()} of $P(\mu,y_1,\ldots,y_r,H)$, then integration by parts with respect to the variable $y_j$ shows that 
$$|J({\bf x},\gamma,{\bf m};q)|\lesssim_{M,\varepsilon} N^{-M}$$
for all $M>0$. This will produce a negligible contribution and we may now assume that in the summation $\sum_{m_j}$ only at most $N^\varepsilon$ values of $m_j$ have to be considered for each $j=1,\ldots, r$. 

Next we bound the term 
\begin{align} \label{preKlo}
\mathscr{S}({\bf A}, {\bf b}, {\bf m}, n_0, q, h):=\sum_{1\leq a\leq q,\ (a,q)=1}S({\bf A}, a, a{\bf b}+{\bf m}, q)e\left(\frac{n_0 a+ha^*}{q}\right).
\end{align}
By adapting the method of Weyl differencing \cite[Theorem 8.1]{IK04} to high dimensions (using the key non-degeneracy of the matrix ${\bf A}=(a_{ij})$; compare with Lemma 7.4 of \cite{Zha18} for the slightly different version with smooth cutoff), we may establish
\begin{align}\label{Gaussest}
|S({\bf A}, a, a{\bf b}+{\bf m}, q)|\lesssim_{\varepsilon} q^{-\frac{r}{2}+\varepsilon}
\end{align} 
uniformly in $a$, ${\bf b}$, and ${\bf m}$. This $q^\varepsilon$ factor might be eliminated for some cases, but we will not need it. The above estimate implies 
\begin{align*}
|\mathscr{S}({\bf A}, {\bf b}, {\bf m}, n_0, q, h)|\lesssim_{\varepsilon} q^{-\frac{r}{2}+1+\varepsilon}
\end{align*}
uniformly in ${\bf b}$, ${\bf m}$, $n_0$, and $h$. Slightly more generally, for any $q_0$ coprime to $q$, it holds 
\begin{align*}
\mathscr{S}(q_0{\bf A}, {\bf b}, {\bf m}, n_0, q, h)
=\mathscr{S}({\bf A}, q_0^*{\bf b}, {\bf m}, q_0^*n_0, q, q_0h)
\end{align*}
where $q_0^*q_0\equiv 1$ (mod $q$), thus we also have 
\begin{align}\label{crudeestimate}
|\mathscr{S}(q_0{\bf A}, {\bf b}, {\bf m}, n_0, q, h)|\lesssim_{\varepsilon} q^{-\frac{r}{2}+1+\varepsilon}
\end{align}
uniformly in $q_0$ (coprime to $q$), ${\bf b}$, ${\bf m}$, $n_0$, and $h$.

For a more refined estimate, we explore multiplicativity of the expression $\mathscr{S}({\bf A}, {\bf b}, {\bf m}, n_0, q, h)$ in $q$. Using the Chinese remainder theorem, it is a direct computation to show that for $q_1,q_2$ coprime, it holds
\begin{align*}
\mathscr{S}({\bf A}, {\bf b}, {\bf m}, n_0, q_1q_2, h)
=\mathscr{S}(q_2^2{\bf A}, q_2{\bf b}, {\bf m}, n_0, q_1, h_2^2 h)\cdot
 	\mathscr{S}(q_1^2{\bf A}, q_1{\bf b}, {\bf m}, n_0, q_2, h_1^2 h)
\end{align*}
where $h_1q_1+h_2q_2=1$. So if $q=\prod_{i=1}^s q_i$ where the $q_i$'s are coprime to each other, then 
\begin{align}\label{Sq1q2}
\mathscr{S}({\bf A}, {\bf b}, {\bf m}, n_0, q, h)
=\prod_{i=1,\ldots,s}
\mathscr{S}\left(\left(\prod_{\substack{j=1,\ldots,s \\ j\neq i}}q_j^2\right) {\bf A}, \left(\prod_{\substack{j=1,\ldots,s \\ j\neq i}}q_j\right) {\bf b}, {\bf m}, n_0, q_i, p_i h\right)
\end{align}
where $p_i=p_i(q_1,\ldots,q_s)$ ($i=1,\ldots,s$) are integers. 

Now assume $q$ is a large enough prime. We complete the square:
\begin{align*}
\frac{a}{q}\sum_{i,j} a_{ij}k_ik_j+\frac{1}{q}\sum_{j}k_j(ab_j+m_j) 
\equiv \frac{a}{q}\sum_{i,j}a_{ij}(k_i+l_i)(k_j+l_j)
-\frac{a}{q}\sum_{i,j}a_{ij}l_il_j, \ (\t{mod }1)
\end{align*}
for some ${\bf l}=(l_1,\ldots,l_r)\in\mathbb{Z}^r$. A calculation then shows 
$${\bf l}\equiv 2^*({\bf b}+a^*{\bf m}){\bf A}^{*}\ (\text{mod }q)$$
satisfies the above equation. Here $a^*$ and ${\bf A}^*$ are inverses of $a$ and ${\bf A}$ respectively over the residue field $\mathbb{F}_q$, which always exist when $q$ is large enough. We also have 
$$S(q_0^2{\bf A}, a, {\bf 0}, q)=\left(\frac{a}{q}\right)^r S({\bf A}, 1,{\bf 0}, q)$$
where $\left(\frac{a}{q}\right)$ is the Legendre symbol and $q_0$ is any number coprime to $q$; this may be established by diagonalizing the quadratic form associated to the non-degenerate matrix ${\bf A}$ in $\mathbb{F}_{q}$, thus reducing to a similar identity for one-dimensional Gauss sums.  
\eqref{preKlo} now becomes 
\begin{align*}
\mathscr{S}({\bf A}, {\bf b}, {\bf m}, n_0, q, h)=e\left(\frac{2^*{\bf m}{\bf A}^*{\bf b}^T}{q}\right)S({\bf A}, 1,{\bf 0}, q)\sum_{a=1}^{q-1}\left(\frac{a}{q}\right)^re\left(\frac{a(4^*{\bf b}{\bf A}^*{\bf b}^T+n_0)+a^*(4^*{\bf m}{\bf A}^*{\bf m}^T+h)}{q}\right).
\end{align*}
Slightly more generally, let $q_0$ be any integer that is coprime to $q$. Then
\small
\begin{align*}
\mathscr{S}(q_0^2{\bf A}, q_0{\bf b}, {\bf m}, n_0, q, h)
=e\left(\frac{2^*q_0^*{\bf m}{\bf A}^*{\bf b}^T}{q}\right)S({\bf A}, 1,{\bf 0}, q)\sum_{a=1}^{q-1}\left(\frac{a}{q}\right)^re\left(\frac{a(4^*{\bf b}{\bf A}^*{\bf b}^T+n_0)+a^*(4^*(q_0^*)^2{\bf m}{\bf A}^*{\bf m}^T+h)}{q}\right).
\end{align*}
\normalsize
Using for prime $q$ the Weil bound for Kloosterman sums 
$$\left|\sum_{a=1}^{q-1}e\left(\frac{am+a^*n}{q}\right)\right|\leq 2\sqrt{\t{gcd}(m,n,q)}\sqrt{q},$$
the standard bound for Salié sums 
$$\left|\sum_{a=1}^{q-1}\left(\frac{a}{q}\right)e\left(\frac{am+a^*n}{q}\right)\right|\leq 2\sqrt{q},$$
and \eqref{Gaussest}, we arrive at the bound for prime $q$
\begin{align*}
|\mathscr{S}(q_0^2{\bf A}, q_0{\bf b}, {\bf m}, n_0, q, h)|\lesssim_{\varepsilon} q^{-\frac{r}{2}+\frac{1}{2}+\varepsilon}\sqrt{\t{gcd}(n_1,q)}
\end{align*}
where 
$n_1=4^*{\bf b}{\bf A}^*{\bf b}^T+n_0$. Write 
${\bf A}^*=\det ({\bf A})^* \text{adj}({\bf A})$, and let 
$$n_2=4\det ({\bf A}) n_1={\bf b}\ \text{adj}({\bf A})\ {\bf b}^T+4\det({\bf A}) n_0={\bf b}\ \text{adj}({\bf A})\ {\bf b}^T+4\det({\bf A}) \cdot\frac{\mathcal{T}}{2\pi}(N^2-|\mu|^2+|\rho|^2).$$ 
Thus $n_2$ is realized as an integer that is independent of $q$. 
Clearly $|n_2|\lesssim N^2$. Since $\text{gcd}(n_1,q)= \text{gcd}(n_2,q)$, we replace the above bound by 
\begin{align}\label{primeq}
|\mathscr{S}(q_0^2{\bf A}, q_0{\bf b}, {\bf m}, n_0, q, h)|\lesssim_{\varepsilon} q^{-\frac{r}{2}+\frac{1}{2}+\varepsilon}\sqrt{\t{gcd}(n_2,q)}
\end{align}
for prime $q$. 

Now for any positive integer $q$, write uniquely $q=q_1q_2q_3$, where $q_1$ is squarefree and coprime to $n_2$, $q_2$ is squarefree and divides $n_2$, $q_3$ is the product of prime powers of exponent at least two, and $q_1,q_2,q_3$ are coprime to each other. Note that we may further write uniquely $q_3=q_4^2q_5$ where $q_5$ is a squarefree divisor of $q_4$. Combine the bound \eqref{primeq} for prime $q$ with the estimate \eqref{crudeestimate} for general $q$, using the multiplicativity property \eqref{Sq1q2}, we have the following refined estimate for general $q$
\begin{align}\label{generalq}
|\mathscr{S}({\bf A}, {\bf b}, {\bf m}, n_0, q, h)|\lesssim_\varepsilon \frac{q^{-\frac{r}{2}+1+\varepsilon}}{\sqrt{q_1}}.
\end{align}

We can finally reward ourselves by combining (i) and (ii) of Corollary \ref{coroflem13}, \eqref{Jxgamma}, \eqref{generalq}, and the  remark that only at most $N^\varepsilon$ values of $m_j$ in the summation $\sum_{m_j}$ need to be considered, to get
\begin{align*}
|\kappa_{Q,M}(\mu,H)|&\lesssim_{\varepsilon} N^{|\Sigma^+|+|\Sigma^+_J|+\varepsilon}(NM)^{\frac{r}{2}-1} \sum_{Q\leq q<2Q}\frac{q^{-\frac{r}{2}+1+\varepsilon}}{\sqrt{q_1}}\\
&\lesssim_{\varepsilon}N^{|\Sigma^+|+|\Sigma^+_J|+\varepsilon}(NM)^{\frac{r}{2}-1} Q^{-\frac{r}{2}+1}\sum_{q_1<2Q, \ q_2\mid n_2, \ q_4\leq \left(\frac{2Q}{q_1q_2}\right)^{\frac{1}{2}},\ q_5\mid q_4}
\frac{1}{\sqrt{q_1}}.
\end{align*}
Apply the divisor bound $d(n)\lesssim_\varepsilon n^\varepsilon$, using $|n_2|\lesssim N^2$, we then get 
\begin{align*}
|\kappa_{Q,M}(\mu,H)|
&\lesssim_{\varepsilon}N^{|\Sigma^+|+|\Sigma^+_J|+\varepsilon}(NM)^{\frac{r}{2}-1} Q^{-\frac{r}{2}+1}\sum_{q_1<2Q, \ q_2\mid n_2}
\frac{1}{\sqrt{q_1}}\left(\frac{Q}{q_1q_2}\right)^{\frac{1}{2}+\varepsilon}\\
&\lesssim_{\varepsilon} N^{|\Sigma^+|+|\Sigma^+_J|+\frac{r}{2}-1+\varepsilon}M^{\frac{r}{2}-1}Q^{-\frac{r}{2}+\frac{3}{2}},
\end{align*}
which holds true uniformly for $H\in P_{I,J}$. 
\end{proof}

\section{Discussion on the optimal range of $p$}\label{evidence}

\subsection{For Strichartz estimates}

We now provide evidence for Conjecture \ref{conjectureStrichartz}. As will be seen, Strichartz type estimates on tori become crucial. We first recall the following Strichartz estimates on tori, established by Bourgain and Demeter in \cite[Theorem 2.4 and Remark 2.5]{BD15A} (see also \cite{Zha18} for the $\varepsilon$-removal). 

\begin{thm}\label{BourgainDemeter}
Let $(\cdot,\cdot)$ be a positive-definite quadratic form with integral coefficients in $r$ variables. 
Let $B$ be a bounded domain in $\R^r$ and let $I$ be a bounded interval. Then 
\begin{align*}
\left\|\sum_{\xi\in\Z^r, \ |\xi|\leq N}a_\xi e^{it(\xi,\xi)+i(\xi, x)}\right\|_{L^p((I\times B, \ dt\ dx))}\lesssim N^{\frac{r}{2}-\frac{r+2}{p}}\|a_\xi\|_{l^2(\mathbb{Z}^r)}
\end{align*}
for all $p>\frac{2(r+2)}{r}$. 
\end{thm}

As a corollary, we have the following estimates for restriction to lower dimensional subsets. 

\begin{cor}\label{restrictiontohyperplaneStrichartz}
Let $(\cdot,\cdot)$ be a positive-definite quadratic form with integral coefficients in $r$ variables. Let $s=0,1,\ldots,r$. Pick any $s$-dimensional affine subspace $\R^s$ of $\R^r$ and let $B_s$ be a bounded region in this $\R^s$. Then we have 
\begin{align*}
\left\|\sum_{\xi\in\Z^r,\ |\xi|\leq N}a_\xi e^{it(\xi,\xi)+ i(\xi,x_s)}\right\|_{L^p(I\times B_s,\ dt\ dx_s)}\lesssim  N^{\frac{r}{2}-\frac{s+2}{p}}\|a_\xi\|_{l^2(\Z^r)}
\end{align*}
for any $p>\frac{2(r+2)}{r}$.
\end{cor}
\begin{proof}
By integrality of the inner product $(\cdot,\cdot)$, we can pick  vectors $v_1,\ldots,v_{r-s}\in\R^r$ such that $(\xi,v_j)\in\Z$ for any $\xi\in\Z^r$ and that $\bigoplus_{j=1}^{r-s}\R v_j$ is transversal to $B_s$ in $\R^r$. Let $I_{r-s}:=\{\sum_{j=1}^{r-s}s_j v_j:\ s_j\in[0,1]\}$ and consider $B:=B_s\times I_{r-s}$. 
Then we have 
\begin{align*}
\left\|\sum_{|\xi|\leq N}a_\xi e^{it(\xi,\xi)+ i(\xi,x_s)}\right\|_{L^p(I\times B_s, \ dt\ dx_s)}
&\leq \left\|\sum_{|\xi|\leq N}a_\xi e^{it(\xi,\xi)+ i(\xi,x_s+\sum_{j=1}^{r-s}s_jv_j)}\right\|_{L^p((I\times B_s, \ dt\ dx_s), L^\infty([0,1]^{r-s}, \prod_{j=1}^{r-s}ds_j))}
\end{align*}
Apply Bernstein's inequality on tori to the variables $s_1,\ldots,s_{r-s}$, the above is then bounded by 
\begin{align*}
&\lesssim N^{\frac{r-s}{p}}\left\|\sum_{\xi\in\Z^r,\ |\xi|\leq N}a_\xi e^{it(\xi,\xi)+ i(\xi,x_s+\sum_{j=1}^{r-s}s_jv_j)}\right\|_{L^p((I\times B_s, \ dt\ dx_s), L^p([0,1]^{r-s}, \prod_{j=1}^{r-s}ds_j))} \\
&\lesssim N^{\frac{r-s}{p}}\left\|\sum_{\xi\in\Z^r,\ |\xi|\leq N}a_\xi e^{it(\xi,\xi)+ i(\xi,x)}\right\|_{L^p(I\times B, \ dt\ dx)}\lesssim N^{\frac{r}{2}-\frac{s+2}{p}}\|a_\xi\|_{l^2(\Z^r)}, 
\end{align*}
for all $p>\frac{2(r+2)}{r}$, using Theorem \ref{BourgainDemeter}. 
\end{proof}

As will be seen, the above Strichartz type estimates on tori would not be enough to derive the optimal range for compact Lie groups. We conjecture the following Strichartz estimates on tori for mixed Lebesgue norms. 

\begin{conj} \label{conjStritori}
We have 
\begin{align}\label{conj1}
\left\|\sum_{\xi\in\Z^r, \ |\xi|\leq N}a_\xi e^{it(\xi,\xi)+i(\xi, x)}\right\|_{L^p((I, dt),L^q(B, dx))}\lesssim N^{\frac{r}{2}-\frac{2}{p}-\frac{r}{q}}\|a_\xi\|_{l^2(\mathbb{Z}^r)}
\end{align}
for all pairs $p,q\geq 2$ with $\frac{r}{2}-\frac{2}{p}-\frac{r}{q}>0$. 
\end{conj}

The above exponent of $N$ is of course based on a scale-invariance consideration. On Euclidean spaces, Strichartz estimates for mixed Lebesgue norms as the above are indeed true, as proved in \cite{GV95} and \cite{KT98}. Arguing exactly as in the proof of Corollary \ref{restrictiontohyperplaneStrichartz}, we also arrive at the following lemma. 

\begin{lem}
Let $s=0,1,\ldots,r$. Pick any $s$-dimensional affine subspace $\R^s$ of $\R^r$ and let $B_s$ be a bounded region in this $\R^s$.  Then the above conjecture implies that 
\begin{align}\label{conj2}
\left\|\sum_{\xi\in\Z^r,\ |\xi|\leq N}a_\xi e^{it(\xi,\xi)+i (\xi, x_s)}\right\|_{L^p((I, dt),L^q(B_s, dx_s))}\lesssim N^{\frac{r}{2}-\frac{2}{p}-\frac{s}{q}}\|a_\xi\|_{l^2(\Z^r)}
\end{align}
for all pairs $p,q\geq 2$ with $\frac{r}{2}-\frac{2}{p}-\frac{r}{q}>0$.
\end{lem}

Theorem \ref{BourgainDemeter} was proved by the so-called decoupling theory, which is a Euclidean Fourier analytic tool developed in recent years. Although it has proved very successful in estimating exponential sums, the mixed-norm setting as in \eqref{conj1} seems to present new difficulties and is interesting by itself. 
We now provide the following evidence for Conjecture \ref{conjectureStrichartz}. 

\begin{thm}\label{Class}
(i) The Strichartz estimate \eqref{Strichartz} holds for class functions on compact Lie groups for any $p>2+\frac{4}{r}$. \\  
(ii) Conjecture \ref{conjStritori} implies Conjecture \ref{conjectureStrichartz}. 
\end{thm}
\begin{proof}
For the sake of simplicity of exposition, we assume that $U$ is a compact simply connected simple Lie group. The general case may be established by slightly adapting the argument. By Schur's orthogonality relations, it is well known that with respect to the normalized Haar measure on $U$,
$$\|\chi_{\mu}\|_{L^2(U)}=1,\ \forall \mu\in\Lambda^+.$$
Let $L^2_{\sharp}(U)$ denote the set of class functions in $L^2(U)$. Then $L^2_\sharp(U)\cong l^2(\Lambda^+)$, by 
$$L^2_\sharp(U)\ni f=\sum_{\mu\in\Lambda^+}a_\mu \chi_{\mu}\mapsto (a_\mu)_{\mu\in\Lambda^+}\in l^2(\Lambda^+).$$
Apply Weyl's integration formula \eqref{Weylint}, then \eqref{Strichartz} is reduced to 
\begin{align}\label{classStri}
\left\|\sum_{\mu\in\Lambda^+, |\mu|\leq N}e^{-it|\mu|^2}a_\mu\chi_{\mu}|\delta|^{\frac{2}{p}}\right\|_{L^p(I\times A)}\lesssim N^{\frac{d}{2}-\frac{d+2}{p}}\|a_\mu\|_{l^2(\Lambda^+)}.
\end{align}
Since $A=\bigsqcup_{J}P_J$, it suffices to prove the above estimate replacing $A$ by each $P_J$. 
In the following, for $a_\mu$ initially defined for $\mu\in\Lambda^+$, we let $a_{s\mu}:=a_{\mu}$, $\forall \mu\in\Lambda^+, \ s\in W$.
The fact that 
$$|\Sigma^+|=\frac{d-r}{2}$$
will be used in various places. 

\underline{Case 1. $J=\emptyset$.} We first treat part (i). 
Using the Weyl character formula 
$$\chi_\mu(H)=\frac{\sum_{s\in W}\det s \ e^{(s\mu)(H)}}{\delta(H)}$$
and \eqref{deltapreliminary}, we have 
\begin{align*}
\left\|\sum_{\mu\in\Lambda^+,|\mu|\leq N}e^{-it|\mu|^2} a_\mu\chi_\mu|\delta|^{\frac{2}{p}}\right\|_{L^p(I\times P_\emptyset)}
&\lesssim N^{|\Sigma^+|\left(1-\frac{2}{p}\right)}\sum_{s\in W}
\left\|\sum_{\mu\in s\Lambda^+,|\mu|\leq N}e^{-it|\mu|^2+\mu(H)} a_\mu\right\|_{L^p(I\times P_\emptyset)}.
\end{align*}
Here we have used $|s\mu|=|\mu|$, $\forall s\in W,\ \mu\in\Lambda$. If we apply Theorem \ref{BourgainDemeter} to the sum on the right inside of $\|\cdot\|_{L^p}$, then we have that for any $p>2+\frac{4}{r}$
\begin{align*}
\left\|\sum_{\mu\in\Lambda^+,|\mu|\leq N}e^{-it|\mu|^2} a_\mu\chi_\mu|\delta|^{\frac{2}{p}}\right\|_{L^p(I\times P_\emptyset)}
&\lesssim N^{|\Sigma^+|\left(1-\frac{2}{p}\right)+\frac{r}{2}-\frac{r+2}{p}}
\|a_\mu\|_{l^2(\Lambda^+)}\\
&\lesssim N^{\frac{d}{2}-\frac{d+2}{p}}
\|a_\mu\|_{l^2(\Lambda^+)}.
\end{align*}
For part (ii), we need to exploit $L^p$ estimates of the weight functions as in Proposition \ref{keyprop}. Write $P_\emptyset=\bigcup_{|I|=r} P_{I,\emptyset}$. 
Writing $\delta=\delta_I\cdot\delta_{I,J}$ and using the character formula again, we estimate using \eqref{deltaI}
\begin{align*}
&\left\|\sum_{\mu\in\Lambda^+,|\mu|\leq N}e^{-it|\mu|^2} a_\mu\chi_\mu|\delta|^{\frac{2}{p}}\right\|_{L^p(I\times P_{I,\emptyset})}\\
&\lesssim \sum_{s\in W}
\left\|\sum_{\mu\in s\Lambda^+, |\mu|\leq N}e^{-it|\mu|^2+\mu(H)}a_\mu\right\|_{L^p(I,L^q(P_{I,\emptyset}))}\cdot\left\|\frac{1}{|\delta_{I,\emptyset}|^{1-\frac{2}{p}}}\right\|_{L^{u}(P_{I,\emptyset})}.
\end{align*}
Here 
$$\frac{1}{u}=\frac{1}{p}-\frac{1}{q}.$$ 
Using the conjectured \eqref{conj1} and Proposition \ref{keyprop}, the above is bounded by 
\begin{align*}
\lesssim N^{\frac{r}{2}-\frac{2}{p}-\frac{r}{q}+\frac{d-r}{2}\left(1-\frac{2}{p}\right)-\frac{r}{u}}\|a_\mu\|_{l^2}=N^{\frac{d}{2}-\frac{d+2}{p}}\|a_\mu\|_{l^2}
\end{align*}
provided the conditions hold
$$2\leq p\leq q, \ \frac{r}{2}-\frac{2}{p}-\frac{r}{q}\geq 0,\ \frac{d-r}{2}\left(1-\frac{2}{p}\right)-r\left(\frac{1}{p}-\frac{1}{q}\right)>0.$$
An inspection of the above inequalities in the $\left(\frac{1}{p},\frac{1}{q}\right)$ plane shows that any $p>2+\frac{4}{d}$ is admissible.

\underline{Case 2. $|J|\geq 1$.} We first treat part (i). 
Apply formula \eqref{char}, we have 
\begin{align}\label{goal}
\left|\sum_{\mu\in\Lambda^+,|\mu|\leq N}e^{-it|\mu|^2}a_\mu\chi_\mu|\delta|^{\frac{2}{p}}\right|
&\leq\frac{|\delta^J|^{\frac{2}{p}}}{|W_J|\cdot |\delta_I|^{1-\frac{2}{p}}|\delta_{I,J}|^{1-\frac{2}{p}}}\sum_{s\in W}\left| \sum_{\mu\in s\Lambda^+,|\mu|\leq N}
e^{-it|\mu|^2+\mu(H_J^\perp)}a_\mu\chi^J_{\mu_J}(H_{J})
\right|.
\end{align}
We wish to apply Corollary \ref{restrictiontohyperplaneStrichartz}. Let 
\begin{align}\label{EJEJEJ}
\mathfrak{e}_J:=\{X\in\mathfrak{t}_J:\ 0\leq 
\alpha_j(X)/2\pi i+\delta_{0j}\leq N^{-1} \ \forall j\in J\}.
\end{align}
For $X\in \mathfrak{e}_J$, let 
$$\mathfrak{e}_J^\perp(X):=\{Y\in \mathfrak{t}_J^\perp:\ X+Y\in P_J\}.$$ 
Then we may express the polytope $P_J$ as
$$
P_J=\{X+Y: \ Y\in \mathfrak{e}^\perp_J(X), \ X\in \mathfrak{e}_J\}.
$$ 
We decompose the Lebesgue measure on $\mathfrak{t}$ into the product of the measure on $\mathfrak{t}_J^\perp$ and that on $\mathfrak{t}_J$. Fubini's theorem gives 
\begin{align*}
&\left\|\sum_{\mu\in s\Lambda^+,|\mu|\leq N}e^{-it|\mu|^2+\mu(H^\perp_J)}a_\mu\chi^J_{\mu_J}(H_J)\right\|_{L^p(I\times P_J)}\\
= &\left\|\left\|\sum_{\mu\in s\Lambda^+,|\mu|\leq N}e^{-it|\mu|^2+\mu(Y)}a_\mu\chi^J_{\mu_J}(X)\right\|_{L^p(I\times\mathfrak{e}_J^\perp(X), \ dt\ dY)}\right\|_{L^p(\mathfrak{e}_J, \ dX)} \numberthis\label{Fubini}.
\end{align*}
Note that $\mathfrak{e}_J^\perp(X)$ is a bounded region in the ($r-|J|$)-dimensional subspace $\mathfrak{t}_J^\perp$ of $\mathfrak{t}$. We apply Corollary \ref{restrictiontohyperplaneStrichartz} to obtain for $p>2+\frac{4}{r}$ that 
\begin{align*}
\left\|\sum_{\mu\in s\Lambda^+,|\mu|\leq N}e^{-it|\mu|^2+\mu(H^\perp_J)}a_\mu\chi^J_{\mu_J}(H_J)\right\|_{L^p(I\times P_J)}\lesssim N^{\frac{r}{2}-\frac{r-|J|+2}{p}}\left\|\left\|a_\mu\chi^J_{\mu_J}(X)\right\|_{l^2(\Z^r)}\right\|_{L^p(\mathfrak{e}_J, \ dX)}.
\end{align*}
Now Lemma \ref{charbound} gives $|\chi^J_{\mu_J}(X)|\lesssim N^{|\Sigma_J^+|}$ for any $X\in\mathfrak{e}_J$; observing that the measure of $\mathfrak{e}_J$ in $\mathfrak{t}_J$ is $\asymp N^{-|J|}$, then the above is bounded by 
$\lesssim N^{|\Sigma_J^+|+\frac{r}{2}-\frac{r+2}{p}}\|a_\mu\|_{l^2}$. 
Combined with Lemma \ref{deltaIdeltaJ}, this then implies that \eqref{goal} is bounded by 
$$\lesssim 
N^{|\Sigma_J^+|+\frac{r}{2}-\frac{r+2}{p}-|\Sigma_J^+|\frac{2}{p}+(|\Sigma^+|-|\Sigma_J^+|)\left(1-\frac{2}{p}\right)}\|a_\mu\|_{l^2}=
N^{\frac{d}{2}-\frac{d+2}{p}}\|a_\mu\|_{l^2}$$ for all $p>2+\frac{4}{r}$. 
For part (ii), write $P_{J}=\bigcup_{I\supset J,|I|=r}P_{I,J}$.
For $X\in \mathfrak{e}_J$, let 
$$\mathfrak{e}_{I,J}^\perp(X):=\{Y\in \mathfrak{t}_J^\perp:\ X+Y\in P_{I,J}\}.$$
Then 
$$
P_{I,J}=\{X+Y: \ Y\in \mathfrak{e}^\perp_{I,J}(X), \ X\in \mathfrak{e}_J\}.
$$  
A similar Fubini identity as \eqref{Fubini} holds, replacing $P_J$ by $P_{I,J}$ and $\mathfrak{e}_J^\perp(X)$ by $\mathfrak{e}_{I,J}^\perp(X)$. Using this Fubini identity, \eqref{goal}, \eqref{deltaI} and \eqref{deltaJ}, we estimate 
\begin{align*}
&\left\|\sum_{\mu\in\Lambda^+,|\mu|\leq N}e^{-it|\mu|^2}a_\mu\chi_\mu|\delta|^{\frac{2}{p}}\right\|_{L^p(I\times P_{I,J})}\\
&\lesssim \sum_{s\in W} N^{-|\Sigma_J^+|\frac{2}{p}} 
\left\|\left\|\sum_{\mu\in s\Lambda^+,|\mu|\leq N}e^{-it|\mu|^2+\mu(Y)}a_\mu \chi^J_{\mu_J}(X)\right\|_{L^p(I,L^q(\mathfrak{e}_{I,J}^\perp(X)))}\cdot \left\|\frac{1}{|\delta_{I,J}|^{1-\frac{2}{p}}}\right\|_{L^u(\mathfrak{e}_J^\perp(X))}\right\|_{L^p(\mathfrak{e}_J)}.
\end{align*}
Here $\frac{1}{u}=\frac{1}{p}-\frac{1}{q}$. Since $\mathfrak{e}^\perp_{I,J}(X)$ is a bounded region in the ($r-|J|$)-dimensional subspace $\mathfrak{t}_J^\perp$ of $\mathfrak{t}$, assuming the conjectured estimates \eqref{conj2}, we have that the above is bounded by 
\begin{align*}
\sum_{s\in W}N^{-|\Sigma_J^+|\frac{2}{p}+\frac{r}{2}-\frac{2}{p}-\frac{r-|J|}{q}}\left\| \left\|a_\mu\chi_{\mu^J}^J(X)\right\|_{l^2_\mu}\cdot \left\|\frac{1}{|\delta_{I,J}|^{1-\frac{2}{p}}}\right\|_{L^u(\mathfrak{e}_{I,J}^\perp(X))} \right\|_{L^p(\mathfrak{e}_J)}.
\end{align*}
By again Lemma \ref{charbound}, the above is bounded by 
\begin{align*}
\lesssim N^{|\Sigma_J^+|\left(1-\frac{2}{p}\right)+\frac{r}{2}-\frac{2}{p}-\frac{r-|J|}{q}}\left\|a_\mu\right\|_{l^2}\cdot  \left\|\left\|\frac{1}{|\delta_{I,J}|^{1-\frac{2}{p}}}\right\|_{L^u(\mathfrak{e}_{I,J}^\perp(X)} \right\|_{L^p(\mathfrak{e}_J)}
\end{align*}
which is bounded via H\"older's inequality by 
\begin{align*}
&\lesssim N^{|\Sigma_J^+|\left(1-\frac{2}{p}\right)+\frac{r}{2}-\frac{2}{p}-\frac{r-|J|}{q}}\left\|a_\mu\right\|_{l^2}\cdot  \left\|\left\|\frac{1}{|\delta_{I,J}|^{1-\frac{2}{p}}}\right\|_{L^u(\mathfrak{e}_{I,J}^\perp(X))} \right\|_{L^u(\mathfrak{e}_J)}\cdot \|1\|_{L^q(\mathfrak{e}_J)}\\
&\lesssim N^{|\Sigma_J^+|\left(1-\frac{2}{p}\right)+\frac{r}{2}-\frac{2}{p}-\frac{r-|J|}{q}-\frac{|J|}{q}}\left\|a_\mu\right\|_{l^2}\cdot  \left\|\frac{1}{|\delta_{I,J}|^{1-\frac{2}{p}}}\right\|_{L^u(P_J)},
\end{align*}
which is then bounded via Proposition \ref{keyprop} by 
$$\lesssim N^{|\Sigma_J^+|\left(1-\frac{2}{p}\right)+\frac{r}{2}-\frac{2}{p}-\frac{r-|J|}{q}-\frac{|J|}{q}+\left(|\Sigma^+|-|\Sigma^+_J|\right)\left(1-\frac{2}{p}\right)-\frac{r}{u}}\left\|a_\mu\right\|_{l^2}=N^{\frac{d}{2}-\frac{d+2}{p}}\|a_\mu\|_{l^2}.$$
In the application of \eqref{conj2} and Proposition \ref{keyprop}, we have assumed the following conditions 
$$2\leq p\leq q, \ \frac{r}{2}-\frac{2}{p}-\frac{r}{q}\geq 0, \ \frac{d-r}{2}\left(1-\frac{2}{p}\right)-r\left(\frac{1}{p}-\frac{1}{q}\right)>0.$$
These coincide with those obtained in Case 1 and any $p>2+\frac{4}{d}$ is still admissible. 

\end{proof}

\subsection{For eigenfunction bounds} 
Following a similar line of treatment as for Strichartz estimates, we provide evidence of Conjecture \ref{conjeigensym} by showing how this conjecture could be deduced from the following conjectured eigenfunction bounds on tori by Bourgain \cite{Bou93e}. 

\begin{conj}\label{conjeigentori}
Let $B$ be a bounded region in $\R^r$. For $r\geq 3$, it holds that  
\begin{align}\label{eigentoriconj1}
\left\|\sum_{\xi\in\Z^r,\ |\xi|= N}a_\xi e^{i(\xi, x)}\right\|_{L^p(B, dx)}\lesssim  N^{\frac{r-2}{2}-\frac{r}{p}}\|a_\xi\|_{l^2(\Z^r)}
\end{align}
for any $p>\frac{2r}{r-2}$, 
with an $N^\varepsilon$ loss if $r=3,4$.  
\end{conj} 

Similar to the discussion of Strichartz estimates, the above conjecture implies the following estimate for restriction to lower dimensional subsets, by an argument similar to the proof of Corollary \ref{restrictiontohyperplaneStrichartz}. 

\begin{lem}\label{restrictiontohyperplaneeigen}
Let $s=0,1,\ldots,r$. Pick any $s$-dimensional affine subspace $\R^s$ of $\R^r$ and let $B_s$ be a bounded region in this $\R^s$. Then 
\eqref{eigentoriconj1} implies 
\begin{align*}
\left\|\sum_{\xi\in\Z^r,\ |\xi|= N}a_\xi e^{i(\xi, x_s)}\right\|_{L^p(B_s, dx_s)}\lesssim  N^{\frac{r-2}{2}-\frac{s}{p}}\|a_\xi\|_{l^2(\Z^r)}
\end{align*}
for any $p>\frac{2r}{r-2}$, 
with an $N^\varepsilon$ loss if $r=3,4$.  
\end{lem}

\begin{rem}\label{r=2remark}
The following $r=2$ version of the above estimates indeed holds. We have for $s=0,1,2$
$$\left\|\sum_{\xi\in\Z^2,\ |\xi|= N}a_\xi e^{i(\xi, x_s)}\right\|_{L^\infty(B_s, dx_s)}\lesssim_\varepsilon  N^{\varepsilon}\|a_\xi\|_{l^2(\Z^r)},$$ 
by an application of the counting estimate 
$\#\{\mu\in\Z^2:\ |\mu|=N\}\lesssim_\varepsilon N^{\varepsilon}$ (see Lemma 8 in \cite{BP16}). 
\end{rem}

We are ready to provide the following evidence for Conjecture \ref{conjeigensym}. 

\begin{thm}\label{ClassEigen}
For rank $r=2$, Conjecture \ref{conjeigensym} is a theorem. For $r\geq 3$, 
Conjecture \ref{conjeigentori} implies Conjecture \ref{conjeigensym}.  
\end{thm}
\begin{proof}
The proof is similar to that of Theorem \ref{Class} and we present it in detail for the sake of completeness. 
Class eigenfunctions $f$ of eigenvalue $-N^2+|\rho|^2$ can be expressed as 
\begin{align*}
f=\sum_{\mu\in\Lambda^+,\ |\mu|=N}a_\mu\chi_\mu.
\end{align*}
Using Weyl's integration formula \eqref{Weylint}, inequality \eqref{eigenboundloss} reads 
\begin{align*}
\left\|\sum_{\mu\in\Lambda^+,\ |\mu|=N}a_\mu\chi_\mu|\delta|^{\frac{2}{p}}\right\|_{L^p(A)}\lesssim_\varepsilon N^{\frac{d-2}{2}-\frac{d}{p}+\varepsilon}\|a_\mu\|_{l^2(\Lambda^+)}. 
\end{align*}
Recalling the decomposition $A=\bigcup_{J\subset I,|I|=r} P_{I,J}$, the above estimate reduces to those replacing $A$ by each $P_{I,J}$. 

\underline{Case 1. $J=\emptyset$.} Write
$$\left|\sum_{\mu\in\Lambda^+,|\mu|=N}a_\mu\chi_\mu|\delta|^{\frac{2}{p}}\right|=
\left|\frac{1}{|\delta_I|^{1-\frac{2}{p}}|\delta_{I,\emptyset}|^{1-\frac{2}{p}}}\sum_{s\in W}\det s \sum_{\mu\in s\Lambda^+,|\mu|=N}a_\mu e^{\mu}\right|.$$
Using \eqref{deltaI}, we estimate for 
$$\frac{1}{p}=\frac{1}{u}+\frac{1}{v}$$
that 
\begin{align*}
\left\|\sum_{\mu\in\Lambda^+,|\mu|=N}a_\mu\chi_\mu|\delta|^{\frac{2}{p}}\right\|_{L^p(P_{I,\emptyset})}
&\lesssim \sum_{s\in W}\left\|\sum_{\mu\in s\Lambda^+,|\mu|=N}a_\mu e^{\mu}\right\|_{L^u(P_{I,\emptyset})}\left\|\frac{1}{|\delta_{I,\emptyset}|^{1-\frac{2}{p}}}\right\|_{L^v(P_{I,\emptyset})}\\
&\lesssim_\varepsilon N^{\frac{r-2}{2}-\frac{r}{u}+\frac{d-r}{2}\cdot\left(1-\frac{2}{p}\right)-\frac{r}{v}+\varepsilon}\|a_\mu\|_{l^2(\Lambda^+)}= N^{\frac{d-2}{2}-\frac{d}{p}+\varepsilon}\|a_\mu\|_{l^2(\Lambda^+)},
\end{align*}
where we also used the conjectured estimate \eqref{eigentoriconj1} (and Remark \ref{r=2remark}) and Proposition \ref{keyprop}, provided the following necessary conditions hold
\begin{align}\label{conditions1}
u>\frac{2r}{r-2}\ (u=\infty \t{ if }r=2),\ \left(1-\frac{2}{p}\right)/\left(\frac{1}{p}-\frac{1}{u}\right)>\frac{2r}{d-r}, \ u\geq p\geq 2.
\end{align}
An inspection shows any $p>\frac{2d}{d-2}$ is admissible. 

 \underline{Case 2. $|J|\geq 1$.} 
Using \eqref{char}, we write 
$$\left|\sum_{\mu\in\Lambda^+,|\mu|=N}a_\mu\chi_\mu|\delta|^{\frac{2}{p}}\right|=\frac{|\delta^J|^{\frac{2}{p}}}{|W_J|\cdot |\delta_I|^{1-\frac{2}{p}}|\delta_{I,J}|^{1-\frac{2}{p}}}\left|\sum_{s\in W}\det s 
\sum_{\mu\in s\Lambda^+,|\mu|=N}a_\mu e^{\mu(H_J^\perp)}\chi^J_{\mu_J}(H_J)\right|.$$
For $\frac{1}{p}=\frac{1}{u}+\frac{1}{v}$, we estimate using \eqref{deltaI} and \eqref{deltaJ} that 
\begin{align*}
&\left\|\sum_{\mu\in\Lambda^+,|\mu|=N}a_\mu\chi_\mu|\delta|^{\frac{2}{p}}\right\|_{L^p(P_{I,J})}\\
&\lesssim N^{-|\Sigma_J^+|\frac{2}{p}} \sum_{s\in W}\left\|\sum_{\mu\in s\Lambda^+,|\mu|=N}a_\mu e^{\mu(H_J^\perp)}\chi^J_{\mu_J}(H_J)\right\|_{L^u(P_{I,J})}\left\|\frac{1}{|\delta_{I,J}|^{1-\frac{2}{p}}}\right\|_{L^v(P_{I,J})}\\
&\lesssim N^{-|\Sigma_J^+|\frac{2}{p}} \sum_{s\in W}\left\|\left\|\sum_{\mu\in s\Lambda^+,|\mu|=N}a_\mu e^{\mu(Y)}\chi^J_{\mu_J}(X)\right\|_{L^u(\mathfrak{e}^\perp_{I,J}(X), \ dY)}\right\|_{L^u(\mathfrak{e}_J, \ dX)} \left\|\frac{1}{|\delta_{I,J}|^{1-\frac{2}{p}}}\right\|_{L^v(P_{I,J})}\\
&\lesssim_\varepsilon N^{-|\Sigma_J^+|\frac{2}{p}+\frac{r-2}{2}-\frac{r-|J|}{u}-\frac{|J|}{u}+|\Sigma_J^+|+\left(|\Sigma^+|-|\Sigma_J^+|\right)(1-\frac{2}{p})-\frac{r}{v}+\varepsilon}\|a_\mu\|_{l^2(\Lambda^+)}
= N^{\frac{d-2}{2}-\frac{d}{p}+\varepsilon}\|a_\mu\|_{l^2(\Lambda^+)}.
\end{align*}
Here we have used Lemma \ref{restrictiontohyperplaneeigen}, Lemma \ref{charbound}, Proposition \ref{keyprop}, and the estimate 
$\|1\|_{L^u(\mathfrak{e}_J, \ dX)}\lesssim N^{-\frac{|J|}{u}}$. In applying these estimates, we assumed the following conditions to hold 
$$u>\frac{2r}{r-2}\ (u=\infty\t{ if }r=2),\ \left(1-\frac{2}{p}\right)/\left(\frac{1}{p}-\frac{1}{u}\right)>\frac{2r}{d-r}, \ u\geq p\geq 2.$$
These are the same conditions as in Case 1 and any $p>\frac{2d}{d-2}$ is admissible. 

\end{proof}



\end{document}